\pdfoutput=1
\documentclass{amsart}
\usepackage{graphicx}
\usepackage{amsmath}
\usepackage{amssymb}
\usepackage{amsbsy}
\usepackage{amsgen}
\usepackage{amsopn}
\usepackage{amscd}
\usepackage{amstext}
\usepackage{amsthm}
\usepackage{enumerate}

\theoremstyle{definition}

\newtheorem{acknow}{Acknowledgment}

\newtheorem{remark}{Remark}

\theoremstyle{plain}
\newtheorem{theorem}{Theorem}
\newtheorem{lemma}{Lemma}

\newtheorem{proposition}{Proposition}

\begin{document}

\title[Strong and weak (1, 3) homotopies]{Strong and weak (1, 3) homotopies on knot Projections}
\author{Noboru Ito \qquad Yusuke Takimura \qquad Kouki Taniyama} 
\maketitle

\begin{abstract}
Strong and weak (1, 3) homotopies are equivalence relations on knot projections, defined by the first flat Reidemeister move and each of two different types of the third flat Reidemeister moves.  In this paper, we introduce the cross chord number that is the minimal number of double points of chords of a chord diagram.  Cross chord numbers induce a strong (1, 3) invariant.  We show that Hanaki's trivializing number is a weak (1, 3) invariant.  We give a complete classification of knot projections having trivializing number two up to the first flat Reidemeister moves using cross chord numbers and the positive resolutions of double points.  Two knot projections with trivializing number two are both weak (1, 3) homotopy equivalent and strong (1, 3) homotopy equivalent if and only if they can be related by only the first flat Reidemeister moves.  Finally, we determine the strong (1, 3) homotopy equivalence class containing the trivial knot projection and other classes of knot projections.  
\end{abstract}

\section{Introduction.}
Throughout this paper, we work in the smooth category.  A {\it{knot}} is a circle smoothly embedded into $\mathbb{R}^{3}$.  We consider a regular projection, called a {\it{knot projection}} or {\it{generic immersed spherical curve}}, of the knot to a sphere $S^{2}$, where the term regular projection is a projection to $S^{2}$ in which the image has only transverse double points of self-intersection.  When every double point of a knot projection is specified by over-crossing and under-crossing branches, we call the knot projection a {\it{knot diagram}}.  In particular, a knot projection (resp.~knot diagram) which has no double point is called the {\it{trivial knot projection}} (resp. {\it{trivial knot diagram}}).  

The first, second, and third Reidemeister moves on knot diagrams, depicted in Fig.~\ref{ReidemeisterMove}, are local moves of knot diagrams leading to an ambient isotopy of knots.  Reidemeister moves are frequently used to study knots.  
\begin{figure}[htbp]
\includegraphics[width=10cm]{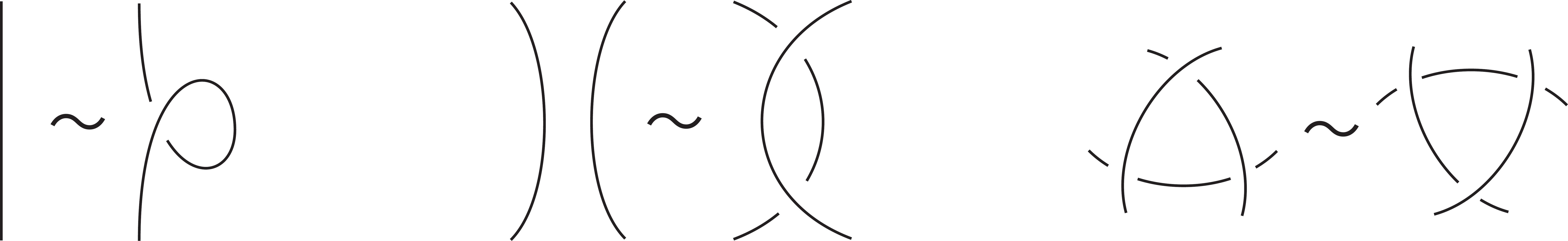}
\caption{The first, second, and third Reidemeister moves.}\label{ReidemeisterMove}
\end{figure}
Two knots are ambient isotopic if and only if two knot diagrams can be related by a finite sequence of Reidemeister moves.  If a knot diagram of an equivalence class can be related to the trivial knot diagram by a finite sequence of Reidemeister moves, the equivalence class is called the {\it{trivial knot type}}.

Naturally, we often consider a flat version of the first, second, and third Reidemeister moves on knot projections on $S^{2}$ not specifying information of over/under-crossing branches, as defined by Fig.~\ref{FlatReidemeisterMove}, and we call three local moves {\it{the first, second, and third homotopy moves}} on knot projections on a sphere $S^{2}$.  
\begin{figure}[htbp]
\includegraphics[width=10cm]{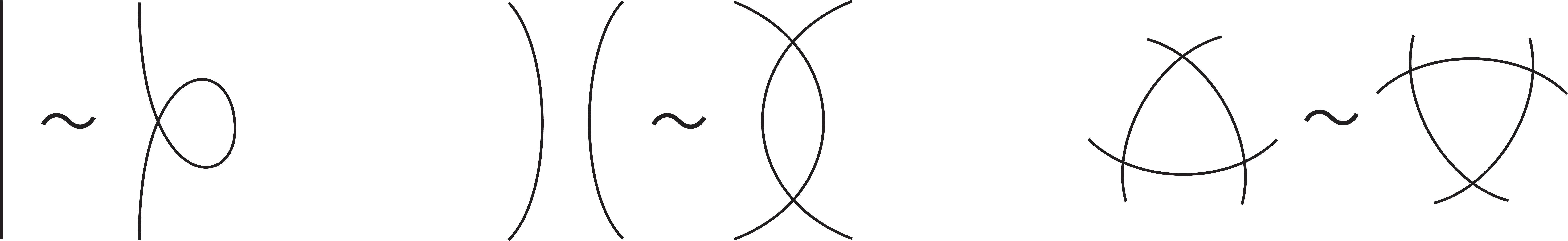}
\caption{The first, second, and third homotopy moves on knot projections.}\label{FlatReidemeisterMove}
\end{figure}

Now, we consider every possibility, i.e., seven nonempty types ($\{1, 2, 3\}$, $\{1, 2\}$, $\{2, 3\}$, $\{1, 3\}$, $\{1\}$, $\{2\}$, $\{3\}$), of choices of flat Reidemeister moves on knot projections.  Under the first, second, and third homotopy moves corresponding to $\{1, 2, 3\}$, the trivial knot projection on a sphere generates every knot projections on the sphere.  The pair of the second and third Reidemeister moves corresponding to $\{2, 3\}$ has been studied as a regular homotopy on plane curves by Whitney \cite{whitney}, whose rotation numbers imply that the trivial knot projection and the knot projection that appears similar to $\infty$ generate every knot projection on the sphere.  Based on regular homotopy theory, Arnold introduced his basic invariants $J^{+}$, $J^{-}$, and $St$ \cite{arnold1, arnold2}.  The sets of homotopy moves concerned with $\{1\}$, $\{2\}$, and $\{1, 2\}$ were already considered in \cite{ito&takimura1}.  However, basic problems of knot projections regarding $\{3\}$ and $\{1, 3\}$ still remain.  

In this paper, the equivalence relation generated by the first and third homotopy moves is called the (1, 3) homotopy.  To begin with, surprisingly, the equivalence class of knot projections containing the trivial knot projection under (1, 3) homotopy on $S^{2}$ has not been determined.  We could find only one related work by Hagge and Yazinski \cite{hagge&yazinski} contributing to the problem.  Hagge and Yazinski \cite{hagge&yazinski} found some knot projections that cannot be related to the trivial knot projection.  Thanks to the study of Hagge and Yazinski, at least we know that knot projections under (1, 3) homotopy are nontrivial.  However, other equivalence classes are still unknown.

In this paper, we consider the special case of flat third Reidemeister moves shown in Fig.~\ref{triplepoint_perestroika} for knot projections on $S^{2}$.  The {\it{strong}} (resp.~{\it{weak}}) third homotopy move is defined by ($s$) (resp.~($w$)) of Fig.~\ref{triplepoint_perestroika}.  
\begin{figure}[htbp]
\includegraphics[width=6cm]{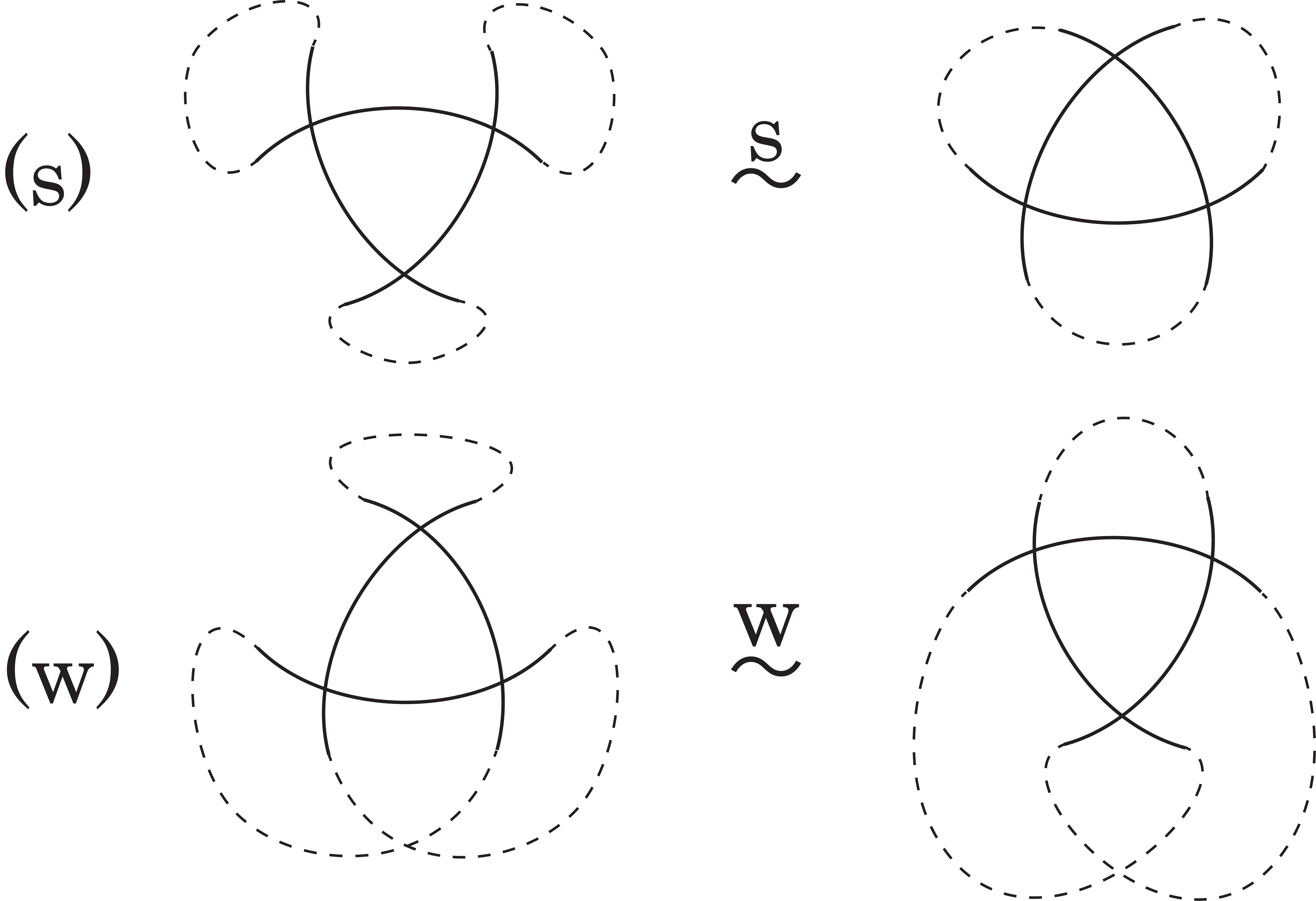}
\caption{Strong and weak triple point perestroikas (strong and weak third homotopy moves for knot projections).  The dotted arcs indicate the connection of the branches.}\label{triplepoint_perestroika}
\end{figure}
We can detect a third homotopy move as either a strong or a weak third homotopy moves by using any choice of orientations of knot projections, as in Fig.~\ref{oriented_strong_weak}.  
\begin{figure}[htbp]
\includegraphics[width=6cm]{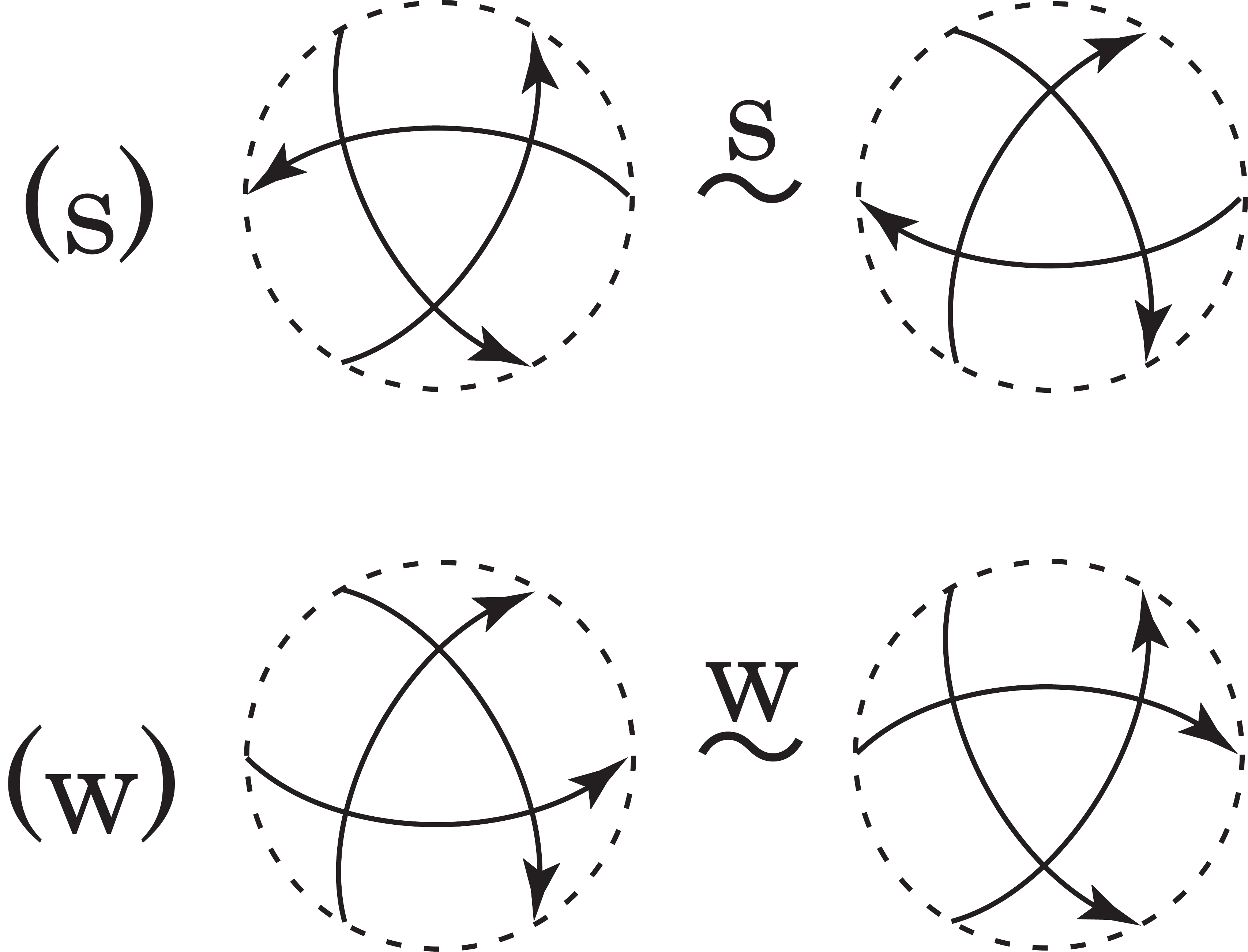}
\caption{Any orientation detects which homotopy move is strong or weak.}\label{oriented_strong_weak}
\end{figure}
Here, we would like to remark that these local moves, defined by the same parts of Fig.~\ref{triplepoint_perestroika}, were introduced by Viro \cite{viro_curve} as strong and weak triple point perestroikas on plane curves in his study of generalizations of Arnold's invariant $J^{-}$.  

For all knot projections, the equivalence relation under the first homotopy move and the weak (resp.~strong) third homotopy move is called {\it{weak}} (resp.~{\it{strong}}) (1, 3) {\it{homotopy}}.  In this paper, let us denote strong (resp.~weak) (1, 3) homotopy equivalence by $\stackrel{s}{\sim}$ (resp.~$\stackrel{w}{\sim}$).  Strong (resp.~weak) (1, 3) nonequivalence is denoted by $\stackrel{s}{\nsim}$ (resp.~$\stackrel{w}{\nsim}$).  

One more important notion in this paper is the trivializing number of a knot projection, introduced by Hanaki \cite{hanaki_osaka}.  Generally, Hanaki defined the trivializing number for spatial graph projections, but in this paper, we consider only the trivializing number of knot projections.  Some definitions, facts, and notations about the notion of trivializing numbers are given following Hanaki \cite{hanaki_osaka}.  A pseudo diagram $P$ is a generic immersed spherical curve with over/under information at some of the transverse double points.  Let $S(P)$ be the set of all the transverse double points of $P$.  If a double point of $P$ has information (resp.~no information) as to over/under-crossing branches, we call the double point a {\it{crossing}} (resp.~{\it{pre-crossing}}).  We consider the subset $c(P)$ (resp.~$pr(P)$) of $S(P)$ consisting of all the crossings (resp.~pre-crossings) of $P$, and then $S(P)$ $=$ $c(P) \sqcup pr(P)$.  The {{trivializing number of a knot projection}} $P$ is defined as the minimum number of elements $c(P)$ whose over/under information is such that even if we take any choice of over/under-crossing branches for all elements of $pr(P)$, we have a knot diagram belonging to the trivial knot type.  

A chord diagram $C\!D_P$ of a knot projection $P$ consists of a circle together with the pre-images of each double point of the knot projection connected by a chord.  Chord diagrams are often called Gauss diagrams (cf.,~\cite{GPV}).  A cross chord is a sub-chord diagram that is a pair of two chords intersecting each other, as in Fig.~\ref{cross_chord}.  
\begin{figure}[htbp]
\includegraphics[width=1.5cm]{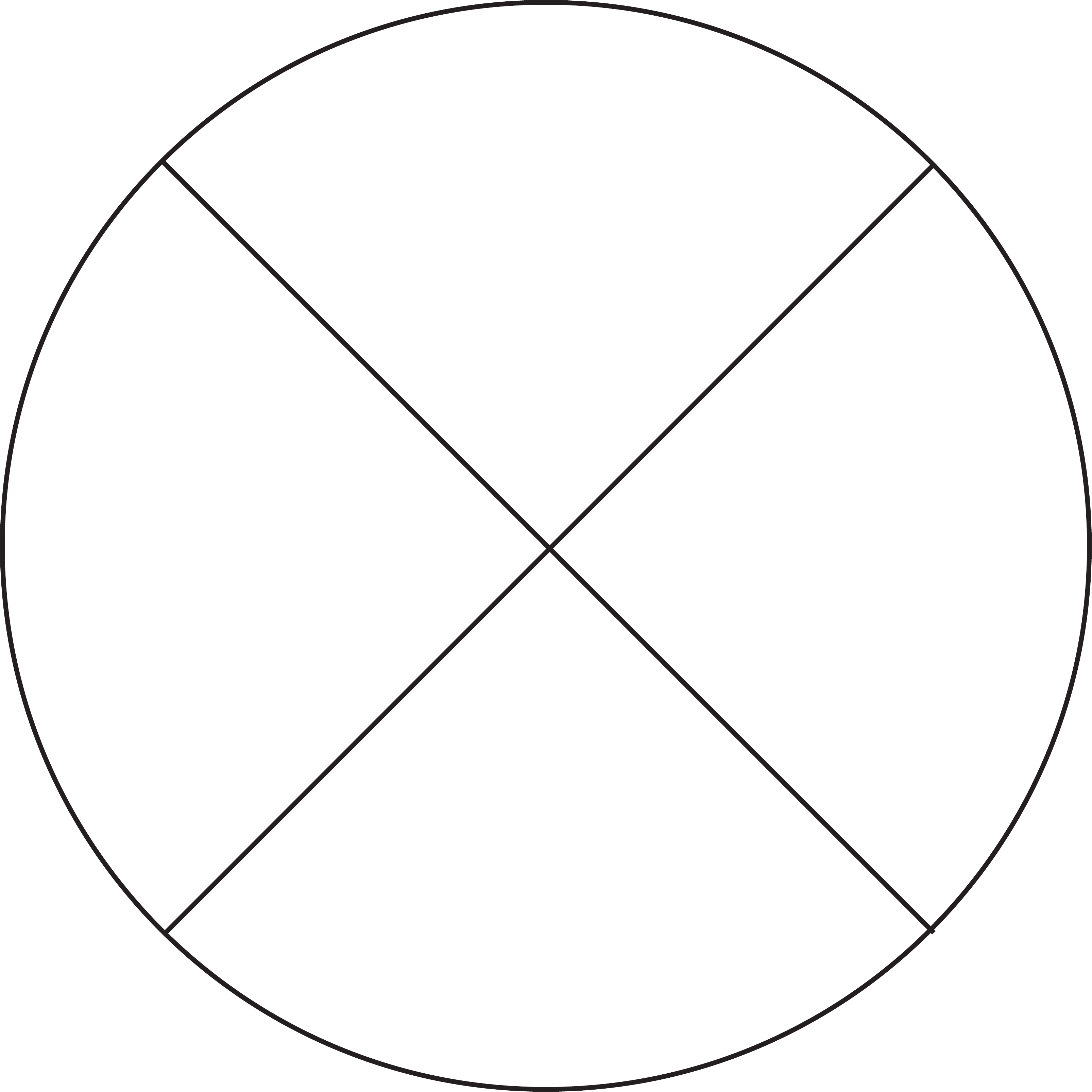}
\caption{A cross chord.}\label{cross_chord}
\end{figure}

Now, we state the main results of the paper.  
\begin{theorem}\label{weak13trivial_thm}
Let $\operatorname{tr}(P)$ be the trivializing number of an arbitrary knot projection $P$.  
\begin{enumerate}[$(1)$]
\item $\operatorname{tr}(P)$ is invariant under the first homotopy move, 
\item $\operatorname{tr}(P)$ is invariant under the weak third homotopy move, 
\item $\operatorname{tr}(P)$ is changed by $\pm 2$ or invariant by a strong third homotopy move.  
\end{enumerate}
In particular, the trivializing number is a weak $(1, 3)$ homotopy invariant for knot projections.  
\end{theorem}
In this paper, we introduce $X(P)$ for an arbitrary knot projection $P$ to present Theorem \ref{weak&strong_thm}.  
\begin{theorem}\label{strong13invariant}
Let $X(P)$ be the number of cross chords in $C\!D_P$ of an arbitrary knot projection $P$.  
\begin{enumerate}[$(1)$]
\item $X(P)$ is invariant under the first homotopy moves.  
\item $X(P)$ is changed by $\pm 3$ by a strong third homotopy move.  
\item $X(P)$ is changed by $\pm 1$ by a weak third homotopy move.
\end{enumerate}
In particular, $X(P)$ $({\operatorname{mod}}~3)$ is a strong $(1, 3)$ homotopy invariant for knot projections.  
\end{theorem}
\noindent The number $X(P)$ is called the {\it{cross chord number}} of knot projections.  
\begin{theorem}\label{weak&strong_thm}
Consider the set of all knot projections with trivializing number two.  Two knot projections are equivalent under both weak $(1, 3)$ homotopy and strong $(1, 3)$ homotopy if and only if the two knot projections can be related by only the first homotopy moves.  
\end{theorem} 
\begin{theorem}\label{strongTrivial_thm}
The knot projection of the strong $(1, 3)$ homotopy equivalence class belonging to the trivial knot projection can be represented by the connected sum of knot projections, each of which is either the trivial knot projection 
$\begin{picture}(10,15)
\put(5,4){\circle{8}}
\end{picture}$, 
the knot projection that appears similar to $\infty$, or the trefoil knot projection defined by Fig.~\ref{trefoilProjection}.  
\end{theorem}
\begin{figure}[htbp]
\includegraphics[width=1.5cm]{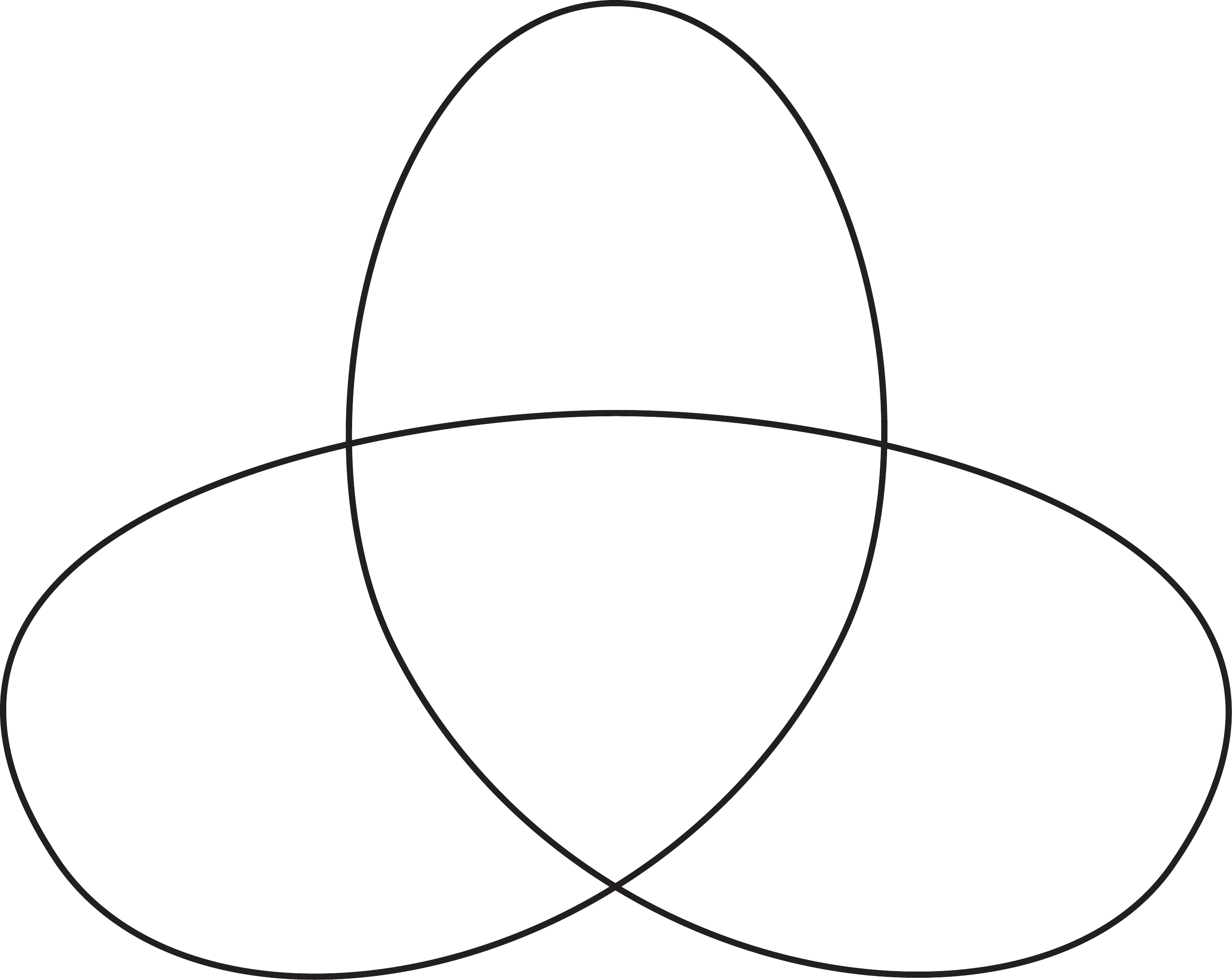}
\caption{Trefoil knot projection.}\label{trefoilProjection}
\end{figure}
\begin{theorem}\label{otherClass_thm}
Let $P_0$ be an arbitrary knot projection without $1$-gons, coherent $2$-gons, and coherent $3$-gons, as shown in Fig.~\ref{37}.  A knot projection $P$ is equivalent to a knot projection $P_0$ under strong $(1, 3)$ homotopy if and only if $P$ is realized as the connected sum of $P_0$ and the knot projections.  Each knot projection appears similar to $\infty$ or the trefoil knot projection shown in Fig.~\ref{trefoilProjection}.  
\end{theorem}
\begin{figure}[htbp]
\includegraphics[width=10cm]{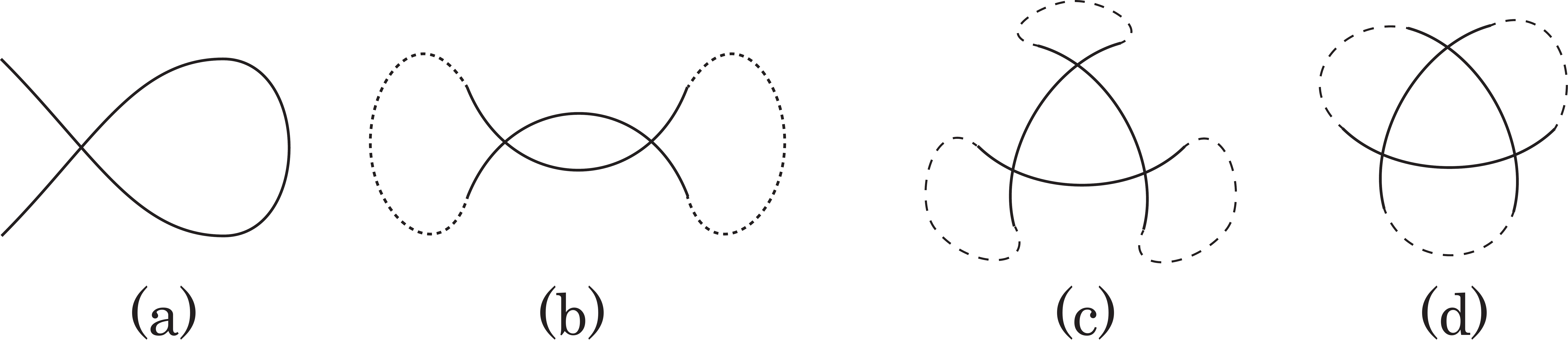}
\caption{(a) $1$-gon, (b) coherent $2$-gon, and (c)-(d) coherent $3$-gons.  }\label{37}
\end{figure}
Here, the term {\it{coherent}} $2$- and $3$-gons implies that these faces are coherently oriented if we give any orientation for a knot projection.  

In Sec.~\ref{weak&trivializing} and Sec.~\ref{CrossSec}, we present proofs of Theorem $\ref{weak13trivial_thm}$ and Theorem \ref{strong13invariant}, respectively.  Theorem \ref{weak&strong_thm} and Theorem \ref{strongTrivial_thm} are proved in Sec.~\ref{strong&weak} and Sec.~\ref{strongTrivial}, respectively.  Finally, we obtain a general result (Theorem \ref{otherClass_thm}) of Theorem \ref{strongTrivial_thm} in Sec.~\ref{strongGeneral}.  

\section{Trivializing number is a weak (1, 3) invariant.}\label{weak&trivializing}
A {\it{trivial chord diagram}} is a chord diagram not containing cross chords (Fig.~\ref{trivial_chord}).  
\begin{figure}[htbp]
\includegraphics[width=2cm]{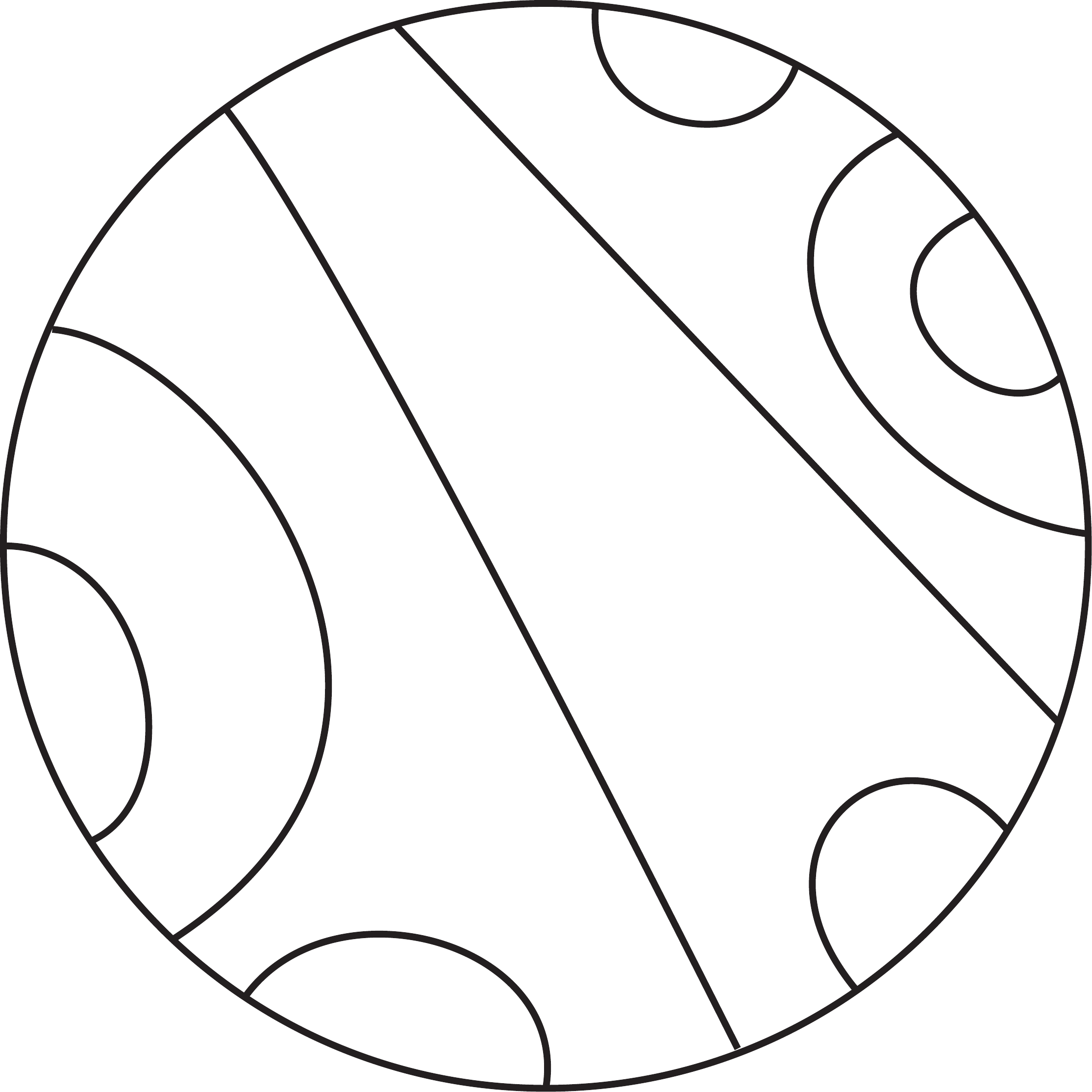}
\caption{Trivial chord diagrams.}\label{trivial_chord}
\end{figure}
Hanaki showed the following \cite[Page7, Theorem 13]{hanaki_JMJ}.  
\begin{theorem}[Hanaki]\label{hanaki_thm}
Let $\operatorname{tr}(P)$ be the trivializing number of a knot projection $P$.  We have
\begin{enumerate}[$(1)$]
\item $\operatorname{tr}(P)$ is the minimum number $n$ of chords of $C\!D_P$ such that deleting some n chords from $C\!D_P$ yields a trivial chord diagram, \label{hanaki_chord}
\item $\operatorname{tr}(P)$ is even. \label{hanaki_even} 
\end{enumerate}
\end{theorem}
\begin{remark}
As a corollary of Theorem \ref{hanaki_thm}, a knot projection $P$ is the trivial knot projection or a diagram obtained from the trivial knot projection by applying a sequence of the first homotopy moves if and only if $\operatorname{tr}(P)$ $=$ $0$.  
\end{remark}
Using Theorem \ref{hanaki_thm} (\ref{hanaki_chord}), we prove Theorem \ref{weak13trivial_thm}.  

\noindent {\bf{Proof of Theorem \ref{weak13trivial_thm}.}} 
\begin{proof}
The discussion proceeds by looking at Fig.~\ref{proof_chord_weak}.  
\begin{enumerate}
\item The first homotopy move corresponds to adding or deleting a chord not producing cross chords.  Then, Hanaki's Theorem \ref{hanaki_thm} immediately completes the proof.  
\item \label{weak_tr} Let $P_1$ (resp.~$P_2$) be the left (resp.~right) knot projection of Fig.~\ref{proof_chord_weak}.  
\begin{figure}[htbp]
\includegraphics[width=6cm]{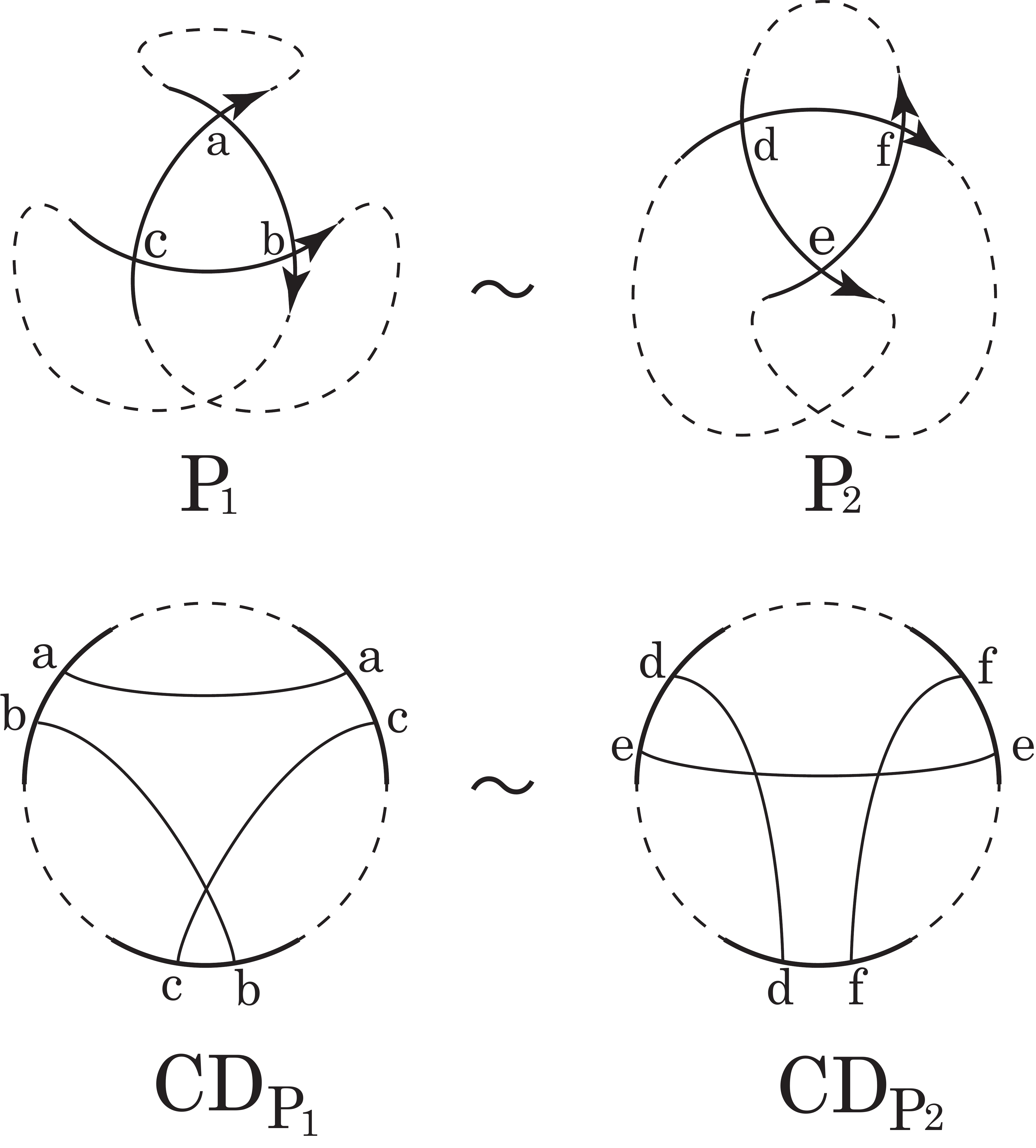}
\caption{One weak third homotopy move expressed by a knot projection (upper) and a chord diagram (lower).}\label{proof_chord_weak}
\end{figure}
Two knot projections $P_1$ and $P_2$ are related by a weak third homotopy move.  The difference between $C\!D_{P_1}$ and $C\!D_{P_2}$ corresponds to the difference between chords $a$, $b$, $c$ and $d$, $e$, $f$ (Fig.~\ref{proof_chord_weak}).  There is a one-to-one correspondence between a chord of $C\!D_{P_1}$ and a chord of $C\!D_{P_2}$ if each of the two chords connects two dotted arcs on each of $C\!D_{P_1}$ or $C\!D_{P_2}$.  Below, we show (I) $\operatorname{tr}(P_1)$ $\ge$ $\operatorname{tr}(P_2)$ and (I\!I) $\operatorname{tr}(P_1)$ $\le$ $\operatorname{tr}(P_2)$.  

\noindent (I) Assume that when we delete $\operatorname{tr}(P_1)$ chords, we delete $m$ chords among the three chords $a$, $b$, and $c$.  Below, we consider each case for $m$.  
\begin{itemize}
\item Case $m$ $=$ $0$.  In this case, the chords $b$ and $c$ are left, but the two chords become cross chords, and then, $C\!D_{P_1}$ is not a trivial chord diagram, which produces a contradiction.  Then, $m$ $\neq$ $0$, i.e., there is no need to consider the case.  
\item Case $m$ $=$ $1$.  It is necessary to delete $b$ or $c$.  First, we choose $b$.  If we delete $e$, cross chords consisting of $d$, $e$, and $f$ disappear.  Recall that there is a one-to-one correspondence between a chord $x_i$ connecting dotted arcs belonging to $C\!D_{P_1}$ and a chord $y_i$ of $C\!D_{P_2}$ at the location corresponding to $C\!D_{P_1}$ ($i$ $=$ $1, 2, \dots, {\operatorname{tr}}(P_1)-1$).  After we delete chords $b$ and $e$, we dissolve all cross chords connecting dotted arcs of $C\!D_{P_2}$ in the same way as those of $C\!D_{P_1}$.  This is because a chord $x_i$ with $a$ or $c$ creates cross chords if and only if the corresponding chord $y_i$ with $d$ or $f$ creates cross chords (Fig.~\ref{proof_chord_weak_explain}).  Below, we frequently use the same discussion involving a one-to-one correspondence, denoted by $x_i(P_1)$ $=$ $y_i(P_2)$.  Then, we express the deletion of chords $x_i(P_1)$ as ``we use $x_i(P_1)$ $=$ $y_i(P_2)$.''  Below, we use this expression.  Now, to obtain a trivial chord diagram of $P_2$, it is sufficient to use at most $\operatorname{tr}(P_1)$ chords in this case.  Then, using the minimality of the trivializing number in Hanaki's Theorem \ref{hanaki_thm}, $\operatorname{tr}(P_1)$ $\ge$ $\operatorname{tr}(P_2)$.  
\begin{figure}[htbp]
\includegraphics[width=6cm]{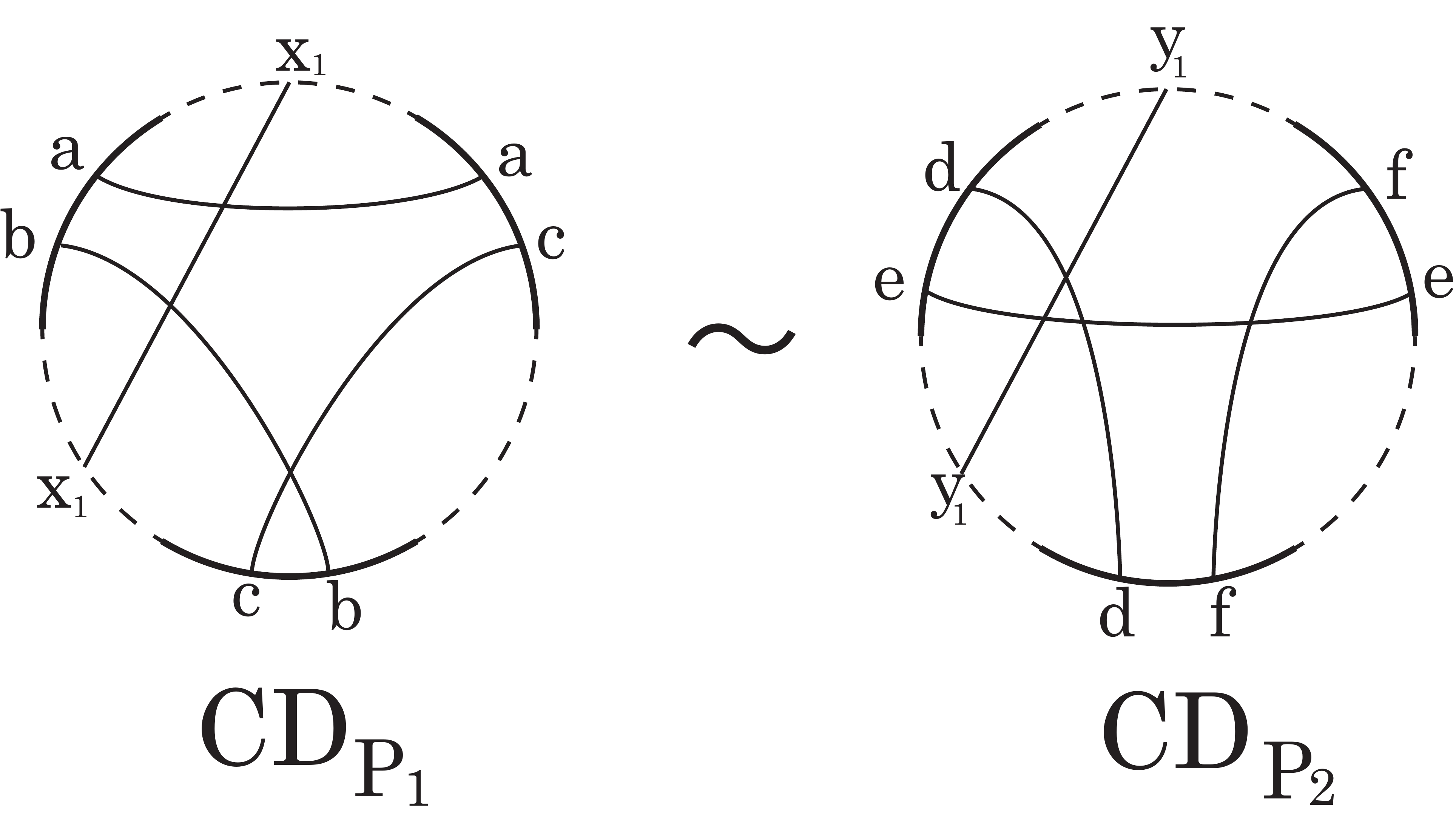}
\caption{One-to-one correspondence of chords connecting dotted arcs.}\label{proof_chord_weak_explain}
\end{figure}
\item Case $m$ $=$ $2$.  If the chords $b$ and $c$ are deleted, then we choose the deletion of $d$ and $f$.  If the chords $a$ and $b$ (resp.~$a$ and $c$) are deleted, then we choose the deletion of $d$ and $e$ (resp.~$e$ and $f$).  For other chords, we use $x_i(P_1)$ $=$ $y_i(P_2)$ ($i$ $=$ $1, 2, \dots, \operatorname{tr}(P_1)-2$).  Again, using Hanaki's Theorem \ref{hanaki_thm}, we have $\operatorname{tr}(P_1)$ $\ge$ $\operatorname{tr}(P_2)$ in this case.  
\item Case $m$ $=$ $3$.  In this case, the chords $a$, $b$, and $c$ were deleted, so we choose the deletion of $d$, $e$, and $f$.  For other chords, we use $x_i(P_1)$ $=$ $y_i(P_2)$ ($i$ $=$ $1, 2, \dots, \operatorname{tr}(P_1)-3$).  Using Hanaki's Theorem \ref{hanaki_thm}, we have $\operatorname{tr}(P_1)$ $\ge$ $\operatorname{tr}(P_2)$.  
\end{itemize}
In summary, we have $\operatorname{tr}(P_1)$ $\ge$ $\operatorname{tr}(P_2)$.  

\noindent (I\!I) The proof of this case is very similar to that of (I).  Assume that when we delete $\operatorname{tr}(P_2)$ chords, we delete $m$ chords among the three chords $d$, $e$, and $f$.  For each case of $m$, to obtain a trivial chord diagram starting from $C\!D_{P_1}$, we show that the deletion of at most $\operatorname{tr}(P_2)$ chords is sufficient, which implies that $\operatorname{tr}(P_1)$ $\le$ $\operatorname{tr}(P_2)$ for each $m$, using the minimality of the trivializing number in Hanaki's Theorem \ref{hanaki_thm}.  
\begin{itemize}
\item Case $m$ $=$ $0$.  The cross chords are left, so there is no need to consider the case.  
\item Case $m$ $=$ $1$.  In this case, the chord $e$ should be deleted and the deletion of either $b$ or $c$ is appropriate.  Similarly to Case $m=1$ of (I), for other chords, we use $x_i(P_2)$ $=$ $y_i(P_1)$.  
\item Case $m$ $=$ $2$.  If the chords $d$ and $f$ are deleted, then we choose the deletion of $b$ and $c$.  If the chords $d$ and $e$ (resp.~$e$ and $f$) are deleted, then we choose the deletion of $a$ and $b$ (resp.~$a$ and $c$).  For other chords, we use $y_i(P_2)$ $=$ $x_i(P_1)$.  
\item Case $m$ $=$ $3$.  In this case, the chords $d$, $e$, and $f$ were deleted, so the deletion of $a$, $b$, and $c$ is appropriate, and for other chords, we use $y_i(P_2)$ $=$ $x_i(P_1)$.  
\end{itemize}

\item The proof proceeds in the same manner as that of (\ref{weak_tr}).  Let $P_3$ (resp.~$P_4$) be the left (resp.~right) knot projection of Fig.~\ref{proof_chord_strong}.  
\begin{figure}[htbp]
\includegraphics[width=7cm]{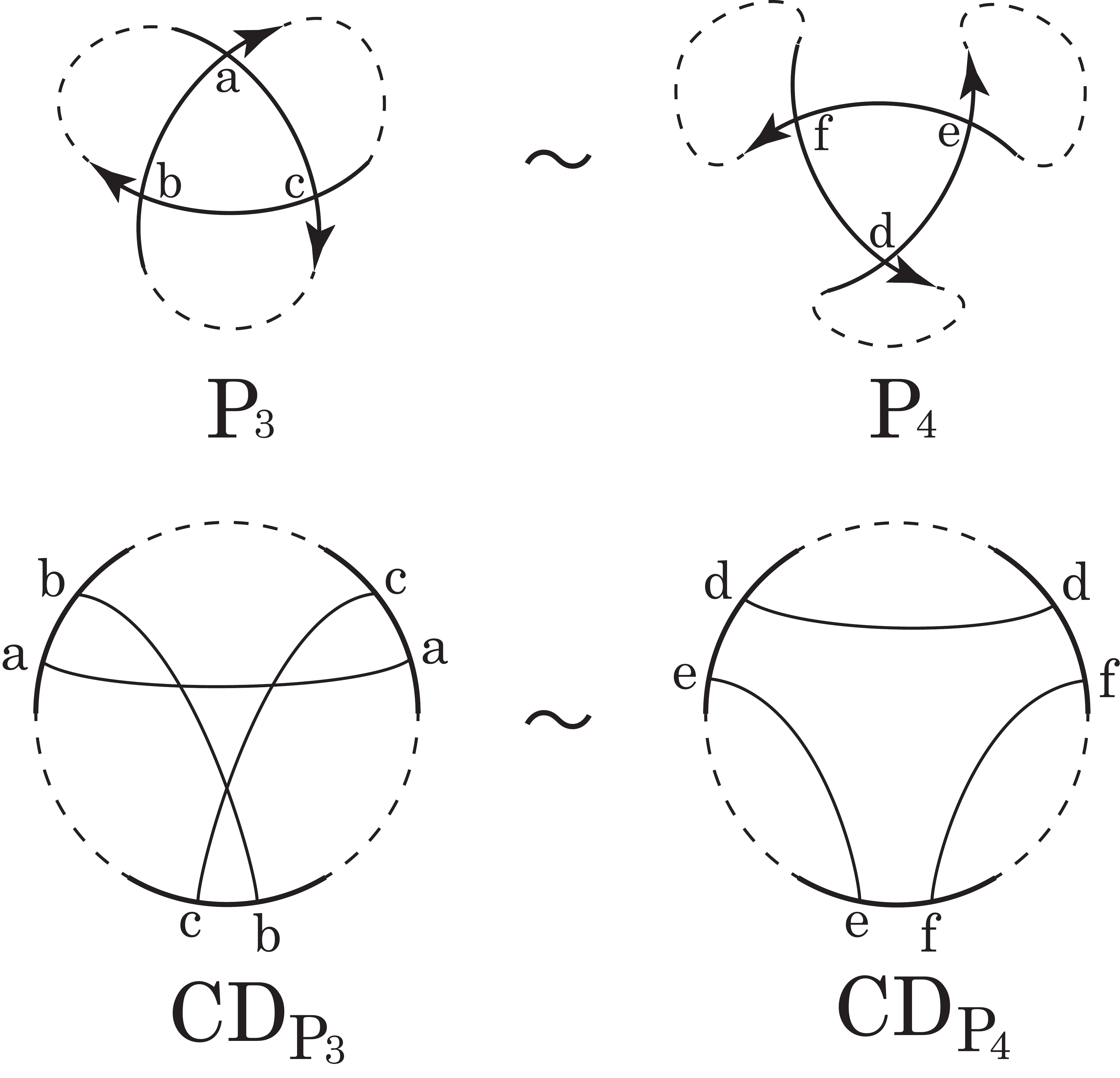}
\caption{One strong third homotopy move expressed by a knot projection (upper) and a chord diagram (lower).}\label{proof_chord_strong}
\end{figure}
Two knot projections $P_3$ and $P_4$ are related by a strong third homotopy move.  The difference between $C\!D_{P_3}$ and $C\!D_{P_4}$ corresponds to the the difference between the three chords $a$, $b$, $c$ and $d$, $e$, $f$ (Fig.~\ref{proof_chord_strong}).  Similarly to (\ref{weak_tr}), there is a one-to-one correspondence between a chord of $C\!D_{P_3}$ and a chord of $C\!D_{P_4}$, if each of the two chords connects two dot arcs on $C\!D_{P_3}$ or $C\!D_{P_4}$.  Below, we show (I) $\operatorname{tr}(P_3)$ $\ge$ $\operatorname{tr}(P_4)$ and (I\!I) $\operatorname{tr}(P_3)$ $\le$ $\operatorname{tr}(P_4)$ $+$ $2$, which imply $\operatorname{tr}(P_4)$ $\le$ $\operatorname{tr}(P_3)$ $\le$ $\operatorname{tr}(P_4)$ $+$ $2$.  Then, thanks to Hanaki's Theorem (\ref{hanaki_even}), $\operatorname{tr}(P)$ is even, so we have $\operatorname{tr}(P_3)$ $=$ $\operatorname{tr}(P_4)$ or $\operatorname{tr}(P_4)$ $+$ $2$.  

Now let us illustrate inequalities (I) and (I\!I).  The style of the proofs below is very similar to that of case (\ref{weak_tr}), so we use the same symbols and notations to minimize repetition of similar phrases.  

\noindent(I) Assume that when we delete $\operatorname{tr}(P_3)$ chords, we delete $m$ chords among the three chords $a$, $b$, and $c$, where $\operatorname{tr}(P_3)-m$ chords consists of chords $x_i(P_3)$ ($i$ $=$ $1, 2, \dots, \operatorname{tr}(P_3)-m$).  By this assumption, we have $m$ $\ge$ $2$.  Below, we consider each case for $m$.  There is a natural one-to-one correspondence between a chord $x_i$ of $C\!D_{P_3}$ and a chord $y_i$ of $C\!D_{P_4}$, except for $a$, $b$, $c$, $d$, $e$, and $f$.  Then, induced chords on $C\!D_{P_4}$ from chords $x_i(P_3)$ of $C\!D_{P_3}$ by the one-to-one correspondence are denoted by $y_i(P_3)$ ($i$ $=$ $1, 2, \dots, \operatorname{tr}(P_3)-m$).  
\begin{itemize}
\item Case $m$ $=$ $2$.  If $\operatorname{tr}(P_3)$ chords consists of $a$, $b$, and $x_i(P_3)$ ($i$ $=$ $1, 2, \dots, \operatorname{tr}(P_3)-2$), $\operatorname{tr}(P_4)$ is less than or equal to the number $\operatorname{tr}(P_3)$ of chords consisting of $e$, $d$, and $y_i(P_4)$ ($i$ $=$ $1, 2, \dots, \operatorname{tr}(P_3)-2$), i.e. $\operatorname{tr}(P_3)$ $\ge$ $\operatorname{tr}(P_4)$.  The role $(a, b, e, d)$ can be replaced with that of either $(a, c, e, f)$ or $(b, c, f, d)$.  
\item Case $m$ $=$ $3$.  For $m$ chords $(a, b, c)$ and $x_i(P_3)$, we delete $(d, e, f)$ and $y_i(P_4)$, which implies that $\operatorname{tr}(P_3)$ $\ge$ $\operatorname{tr}(P_4)$.  
\end{itemize}
In summary, $\operatorname{tr}(P_3)$ $\ge$ $\operatorname{tr}(P_4)$.  

\noindent(I\!I) Assume that when we delete $\operatorname{tr}(P_4)$ on $C\!D_{P_4}$, we delete $m$ chords in three chords $d$, $e$, and $f$ where $\operatorname{tr}(P_4)-m$ chords consists of chords $x_i(P_4)$ ($i$ $=$ $1, 2, \dots, \operatorname{tr}(P_4)-m$).  A one-to-one correspondence between chords of $C\!D_{P_3}$ and $C\!D_{P_4}$, except for $a$, $b$, $c$, $d$, $e$, and $f$, induces $y_i(P_3)$ from $x_i(P_4)$ ($i$ $=$ $1, 2, \dots, \operatorname{tr}(P_4)-m$).  Below, we consider every case for each $m$.  
\end{enumerate}
\begin{itemize}
\item Case $m$ $=$ $0$.  In this case, $\operatorname{tr}(P_3)$ is less than or equal to the number of $y_i(P_3)$ $+$ $2$ $=$ $x_i(P_4)$ $+$ $2$ $=$ $\operatorname{tr}(P_4)$ $+$ $2$, since cross chords consisting of $a$, $b$, and $c$ can be dissolved by deleting any two chords among the three chords $a$, $b$, and $c$.  Then, $\operatorname{tr}(P_3)$ $\le$ $\operatorname{tr}(P_4)$ $+$ $2$.  
\item Case $m$ $=$ $1$.  The number $\operatorname{tr}(P_3)$ is less than or equal to the number of $y_i(P_3)$ $+$ $2$ $=$ $x_i(P_4)$ $+$ $2$ $=$ $\operatorname{tr}(P_4)$ $+$ $1$, since cross chords consisting of $a$, $b$, and $c$ can be dissolved by deleting any two chords among the three chords $a$, $b$, and $c$.  For the case of deletion of $e$, the deletion of pairs $(a, b)$ or $(a, c)$ is sufficient.  For the case of deletion of $d$ (resp.~$f$), the deletion of $(b, a)$ or $(b, c)$ (resp.~$(c, a)$ or $(c, b)$) is sufficient to make a trivial chord diagram.  
\item Case $m$ $=$ $2$.  For the same reason as in the cases of $m$ $=$ $0$ and $m$ $=$ $1$, $\operatorname{tr}(P_3)$ is less than or equal to the number of $y_i(P_3)$ $+$ $2$ $=$ $x_i(P_4)$ $+$ $2$ $=$ $\operatorname{tr}(P_4)$.  For the deletion of pairs $(d, e)$, the choice of deleting $(a, b)$ is sufficient.  Similarly, for the deletion of pairs $(e, f)$ (resp.~$(f, d)$), the choice of deleting $(a, c)$ (resp.~$(b, c)$) is sufficient to make a trivial chord diagram.  Then, $\operatorname{tr}(P_3)$ $\le$ $\operatorname{tr}(P_4)$.  
\item Case $m$ $=$ $3$.  The number $\operatorname{tr}(P_3)$ is less than or equal to the number of $y_i(P_3)$ $+$ $3$ $=$ $x_i(P_4)$ $+$ $3$ $=$ $\operatorname{tr}(P_4)$, since all the choices of deletion of corresponding chords containing $a$, $b$, and $c$ are sufficient to make a trivial chord diagram.  
\end{itemize}
\end{proof}

\section{Cross chord number modulo $3$ is a strong (1, 3) homotopy invariant.}\label{CrossSec}
In this section, we prove Theorem \ref{strong13invariant}.  

\noindent {{\bf{Proof of Theorem \ref{strong13invariant}.}}
\begin{proof}
\begin{enumerate}
\item \label{proof_strong13_1st} The first homotopy move is adding or deleting an isolated chord.  The isolated chord does not produce any cross chords and does not delete any cross chords.  
\item \label{proof_strong13_strong} Let us look at Fig.~\ref{proof_chord_strong}.  One strong third homotopy move changes a cross chord number by $\pm 3$ concerned with triples of chords $(a, b, c)$ and $(d, e, f)$.  Chords other than $(a, b, c)$ produce the same cross chords as chords other than $(d, e, f)$.  Then, checking the increment of $X(P)$ in Fig.\ref{proof_chord_strong}, $X(P_3)$ $=$ $X(P_4)$ $+$ $3$.  
\item Let us look at Fig.~\ref{proof_chord_weak}.  The discussion is the same as in (\ref{proof_strong13_strong}) above, so the discussion reduces to checking the difference in the number of cross chords between $(a, b, c)$ and $(d, e, f)$.  Checking the increment of $X(P)$ in Fig.~\ref{proof_chord_weak}, $X(P_1)$ $+$ $1$ $=$ $X(P_2)$.  
\end{enumerate}
(\ref{proof_strong13_1st}) and  (\ref{proof_strong13_strong}) immediately imply that $X(P)$ ($\operatorname{mod} 3$) is invariant under strong (1, 3) homotopy.  
\end{proof}

A {\it{positive resolution}} of a knot projection is defined as local replacements at every double point, as in Fig.~\ref{positive_resolution} (cf.~\cite{ito&takimura1}).  
\begin{figure}[htbp]
\includegraphics[width=5cm]{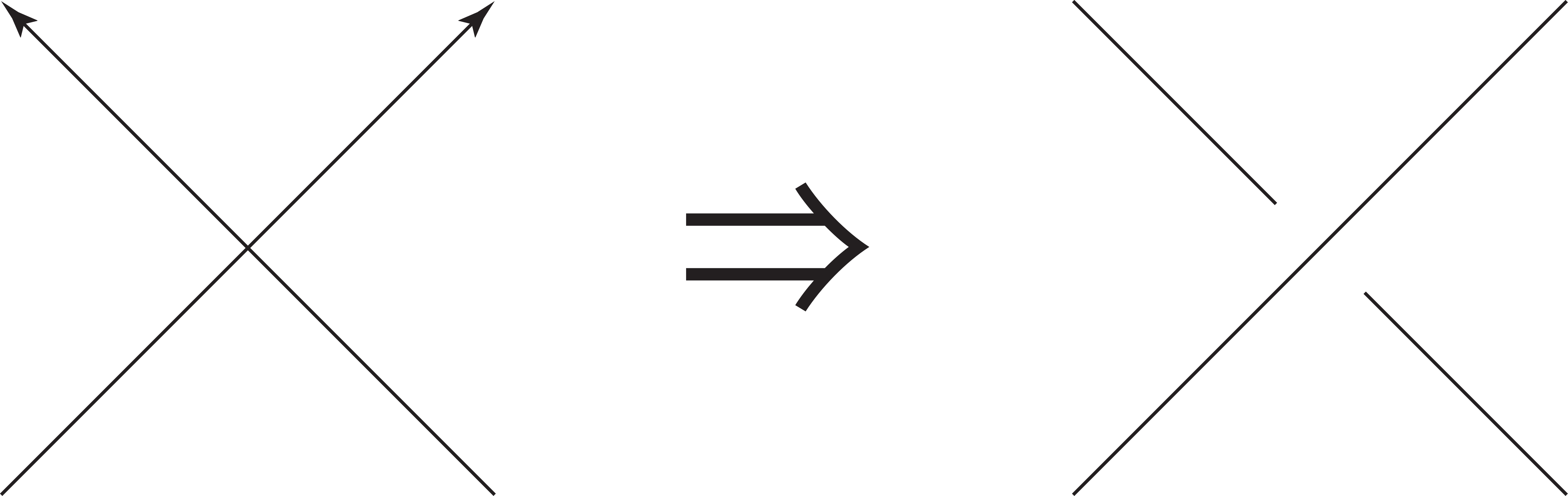}
\caption{Positive Resolution.}\label{positive_resolution}
\end{figure}
This resolution defines the map $p$ from the set of knot projections to the set of knot diagrams.  Moreover, the map $p$ induces the map from the set of weak (1, 3) homotopy classes to the set of knot isotopy classes, denoted by the same symbol $p$, if there is no danger of confusion.  The replacement of all double points as positive resolutions does not change the knot isotopy class (Figs.~\ref{1stPositive} and \ref{3rdPositive}).  
\begin{figure}[htbp]
\includegraphics[width=5cm]{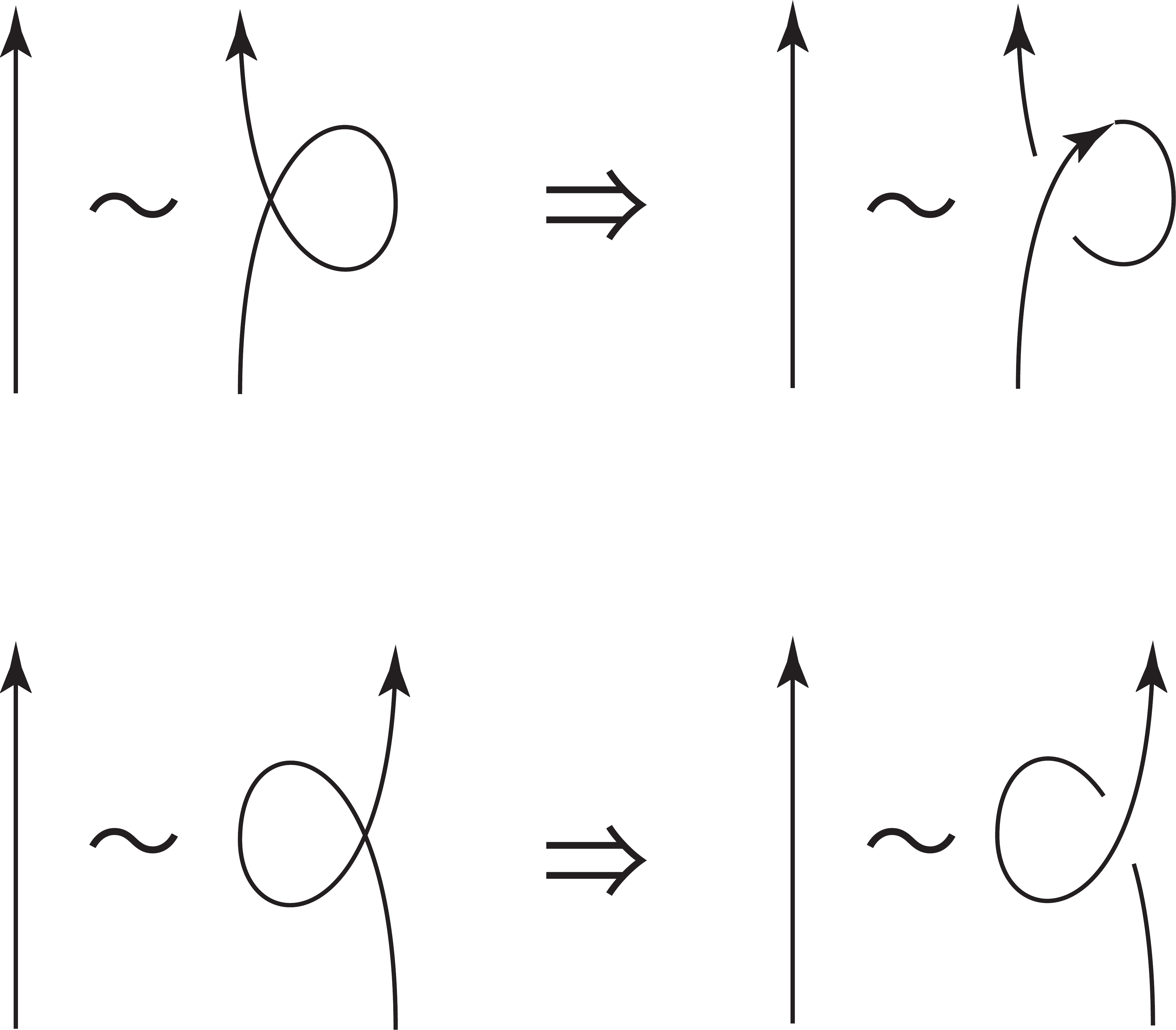}
\caption{The first homotopy move to the first Reidemeister move.}\label{1stPositive}
\end{figure}
\begin{figure}[htbp]
\includegraphics[width=6cm]{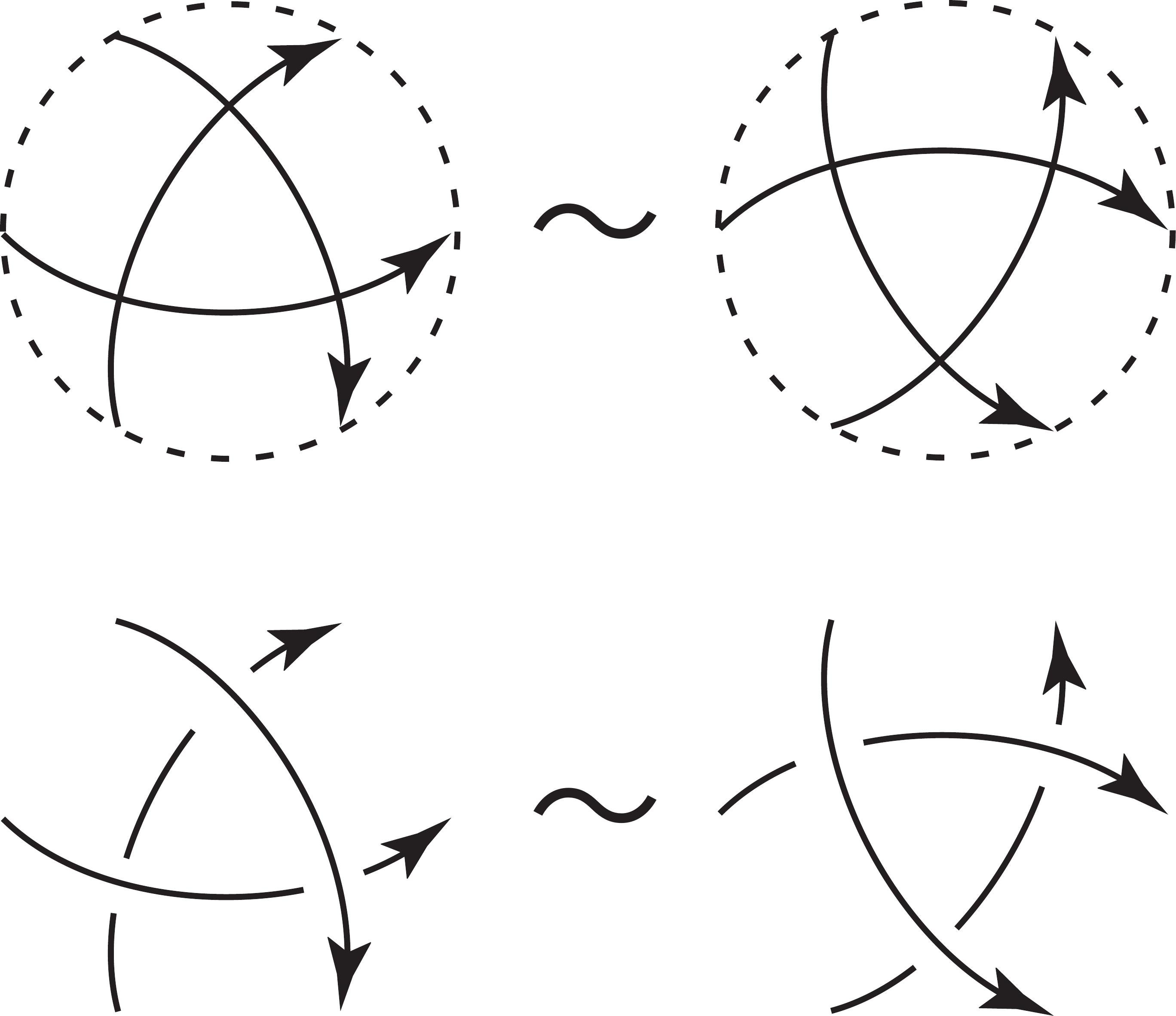}
\caption{The weak third homotopy move to the third Reidemeister move.}\label{3rdPositive}
\end{figure}

Let us recall Hanaki's Theorem (\cite[Page 867, Theorem 1.10]{hanaki_osaka} or \cite[Page 8, Theorem 17]{hanaki_JMJ}).  
\begin{theorem}[Hanaki]\label{hanaki_twist}
Let $P$ be a knot projection.  The knot projection $P$ satisfies $\operatorname{tr}(P)$ $=$ $2$ if and only if $P$ is one of the knot projections defined by Fig.~\ref{twistKnotP} or its versions obtained by the first homotopy moves.  
\end{theorem}
\begin{figure}
\includegraphics[width=4cm]{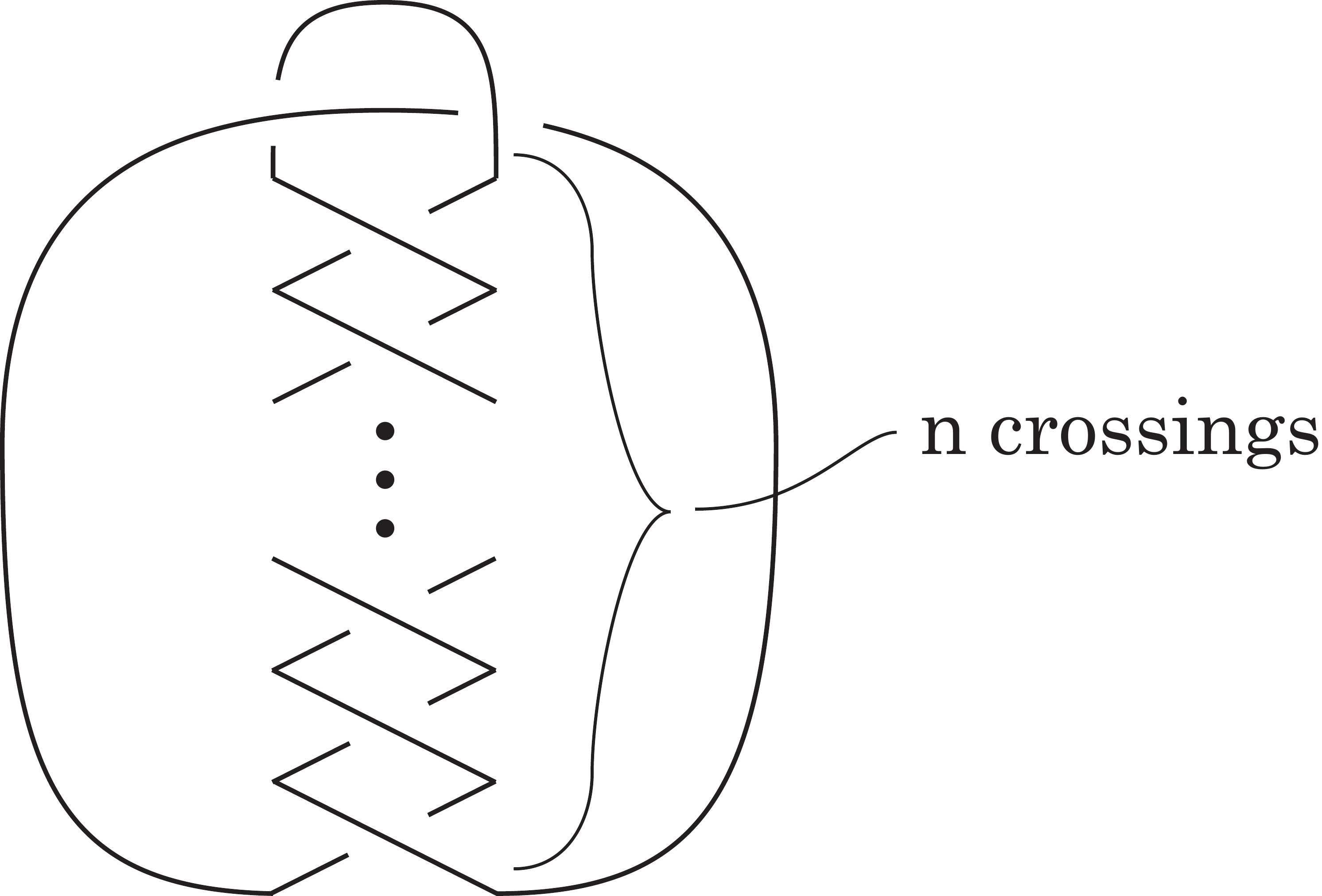}
\caption{Knots $T(n)$.}\label{twistKnot}
\end{figure}

Below, we consider that the knot projections shown in Fig.~\ref{twistKnotP} are under weak (1, 3) homotopy.  For every positive odd integer $n$, let $T(n)$ be the knot defined by Fig.~\ref{twistKnot}.  
\begin{figure}[htbp]
\includegraphics[width=4cm]{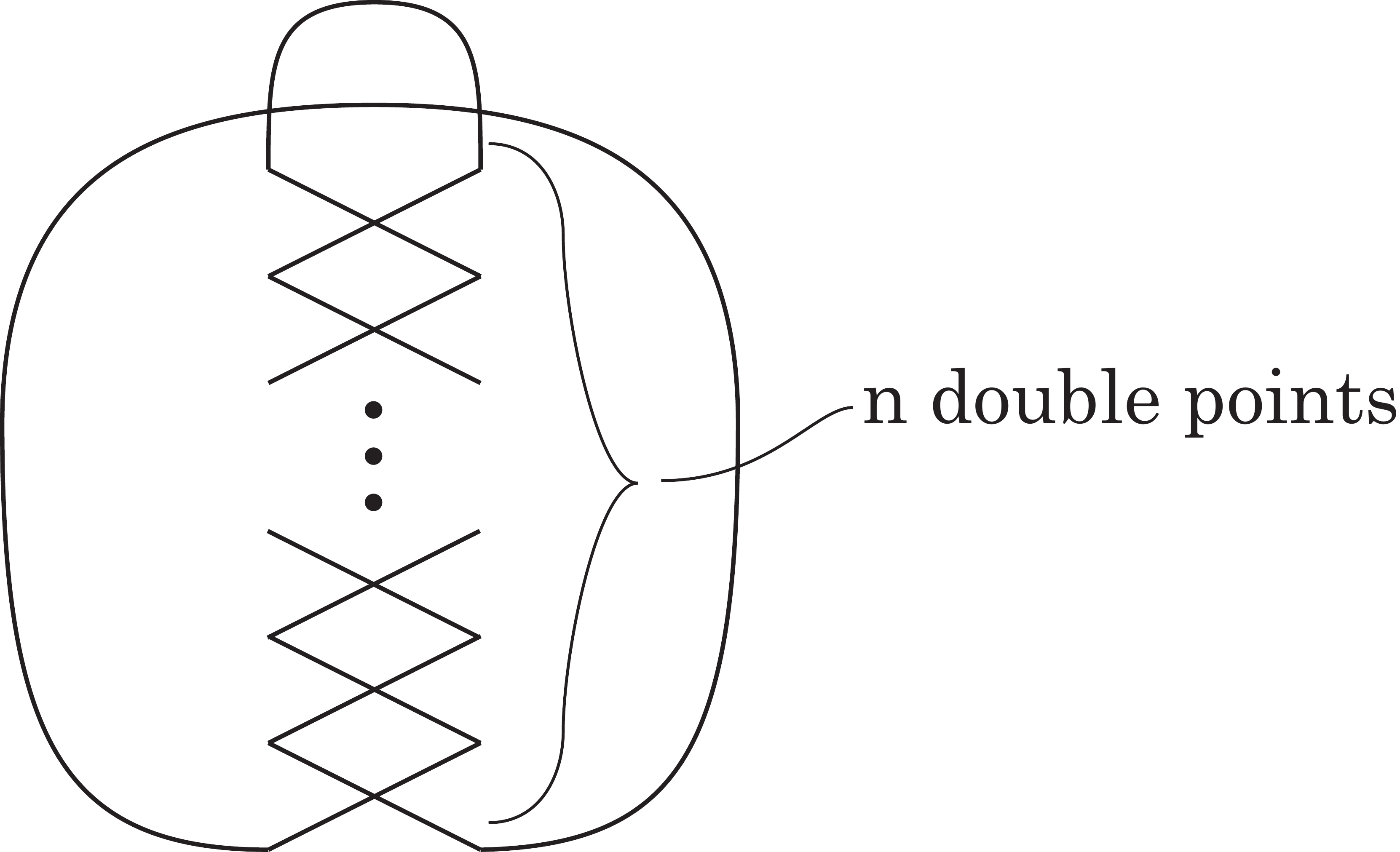}
\caption{Knot projections $\overline{T(n)}$.}\label{twistKnotP}
\end{figure}
For knots $T(n)$, it is well known that $T(n+i)$ $\neq$ $T(n+j)$ for even integers $i$, $j$ ($i \neq j$).  

On the other hand, for every positive odd integer $n$, it is easy to see that $\overline{T(n)}$ is equivalent to $\overline{T(n+1)}$ under weak (1, 3) homotopy where knot projections $\overline{T(n)}$ are defined by Fig.~\ref{twistKnotP}.  Then, we have
\begin{proposition}\label{classification_weak}
For every positive integer $n$, let $\overline{T(n)}$ be the knot projection defined by Fig.~\ref{twistKnotP}.  Under weak (1, 3) homotopy, the equivalence class $[\overline{T(n+i)}]$ is different from $[\overline{T(n+j)}]$, where $i \neq j$, and  $i$ and $j$ are even integers.  In addition, for every positive odd integer $n$, each equivalence class $[\overline{T(n)}]$ contains $\overline{T(n+1)}$.  
\end{proposition}

\section{Strong and weak (1, 3) homotopies on knot projections with trivializing number two.}\label{strong&weak}
{\bf{Proof of Theorem \ref{weak&strong_thm}.}}
\begin{proof}
If two knot projections $P_1$ and $P_2$ are related by only the first homotopy moves, then $P_1$ is equivalent to $P_2$ under not only weak (1, 3) homotopy but also under strong (1, 3) homotopy.  

Then, we will prove the converse.  Assume that two knot projections $P_1$ and $P_2$ are equivalent under not only weak (1, 3) homotopy but also strong (1, 3) homotopy.  
Let $n$ be an odd integer.  Proposition \ref{classification_weak} gives us that for every $n$, any pair of equivalence classes $\{[\overline{T(n)}]\}$ are different under weak (1, 3) homotopy.  Then, to satisfy the assumption, two knot projections $P_1$ and $P_2$ belong to one weak homotopy equivalence class $[\overline{T(n)}]$.  Here, note that representatives of $[\overline{T(n)}]$ are expressed by $\overline{T(n)}$ and $\overline{T(n+1)}$ and projections obtained from these by a repeated applications of the first homotopy moves.  

\begin{figure}[htbp]
\includegraphics[width=10cm]{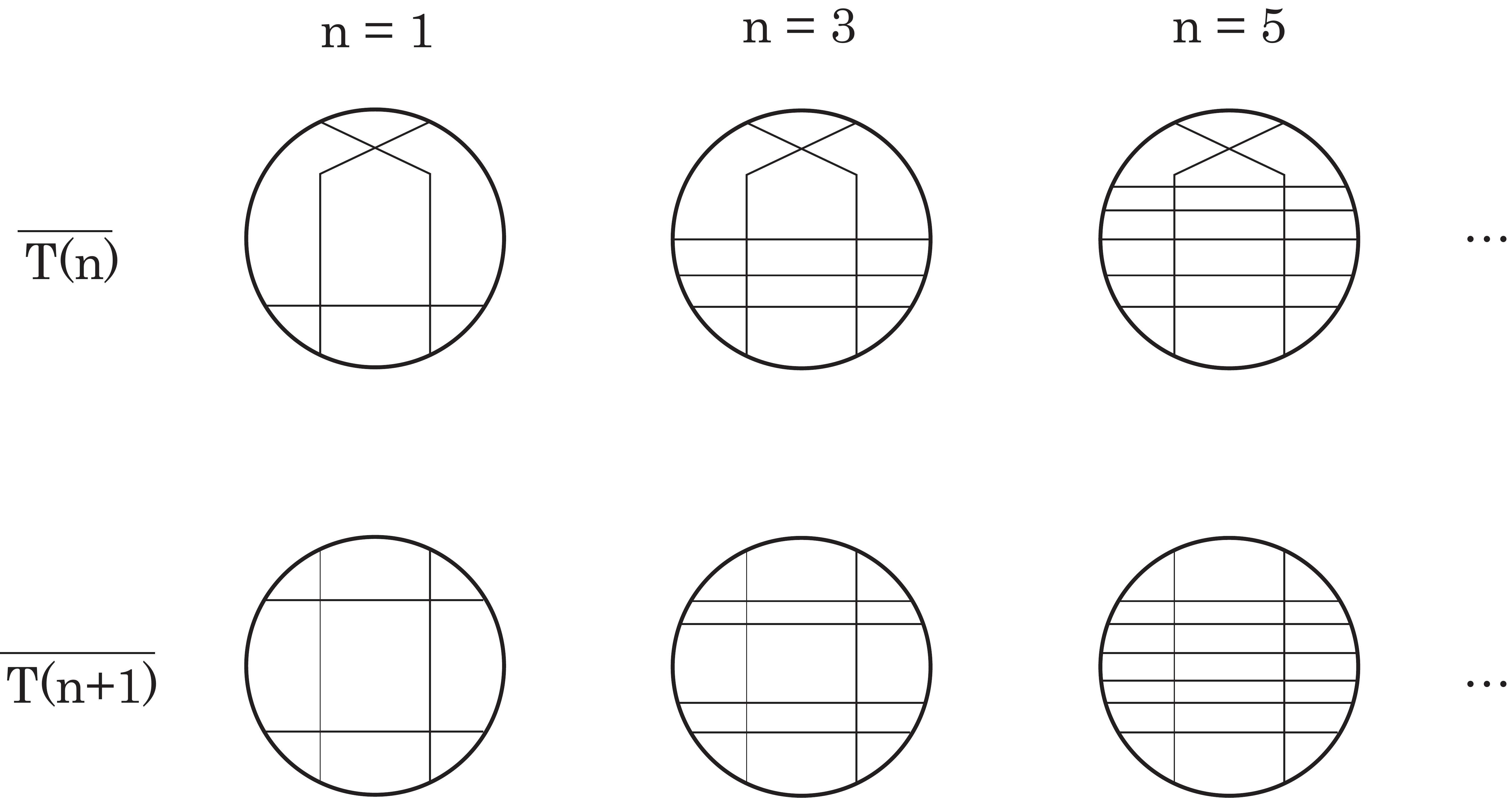}
\caption{Chord diagrams $\overline{T(n)}$ and $\overline{T(n+1)}$.}\label{proof_crossChord}
\end{figure} 
Comparing $C\!D_{\overline{T(n)}}$ and $C\!D_{\overline{T(n+1)}}$ in Fig.~\ref{proof_crossChord}, we obtain $X(\overline{T(n)})$ $+$ $1$ $=$ $X(\overline{T(n+1)})$.  Then $X(\overline{T(n)})$ $+$ $1$ $\equiv$ $X(\overline{T(n+1)})$ ($\operatorname{mod} 3$).  This implies $\overline{T(n)}$ $\stackrel{s}{\nsim}$ $\overline{T(n+1)}$.  Then, for every weak (1, 3) homotopy class $[\overline{T(n)}]$, $\overline{T(n)}$ $\stackrel{s}{\nsim}$ $\overline{T(n+1)}$.  

In summary, if two knot projections $P_1$ and $P_2$ are equivalent under not only weak (1, 3) homotopy but also strong (1, 3) homotopy, the only possibility left is that $P_1$ and $P_2$ are related only by the first homotopy moves.  This completes the proof of Theorem \ref{weak&strong_thm}.
\end{proof}

\begin{remark}
For an odd integer, it is easy to see that $\overline{T(n+1)}$ is equivalent to $\overline{T(n+2)}$ under strong (1, 3) homotopy by using one first homotopy move and one strong third homotopy move.  Then, we have
$\overline{T(n+1)}$ $\stackrel{s}{\sim}$ $\overline{T(n+2)}$ and $\overline{T(n+i)}$ $\stackrel{w}{\nsim}$ $\overline{T(n+j)}$ for $i \neq j$, and $i$ and $j$ are even integers (cf.~Proposition \ref{classification_weak}).  

On the other hand, for an odd integer $n$, it is also easy to see that $\overline{T(n)}$ is equivalent to $\overline{T(n+1)}$ under weak (1, 3) homotopy by using one first homotopy move and one weak third homotopy move, so $\overline{T(n)}$ $\stackrel{w}{\sim}$ $\overline{T(n+1)}$.  In the proof of Theorem \ref{weak&strong_thm}, we showed that $\overline{T(n)}$ $\stackrel{s}{\nsim}$ $\overline{T(n+1)}$.  

Then, for example, there exist interesting sequences such as the following.  
\begin{equation}
\begin{split}
&T(1) \stackrel{w}{\sim} T(2) \stackrel{s}{\sim} T(3) \stackrel{w}{\sim} T(4) \stackrel{s}{\sim}, \dots,\\
&T(1) \stackrel{s}{\nsim} T(2) \stackrel{w}{\nsim} T(3) \stackrel{s}{\nsim} T(4) \stackrel{w}{\nsim}, \dots.
\end{split}
\end{equation}

\end{remark}

\section{Strong (1, 3) homotopy class containing the trivial knot projection.}\label{strongTrivial}
In this section, we prove Theorem \ref{strongTrivial_thm}.  

{\bf{Proof of Theorem \ref{strongTrivial_thm}.}}
\begin{proof}
First, we prove Proposition \ref{lemStrong}.  
\begin{proposition}\label{lemStrong}
Let $(1a)$ and $(1b)$ $($resp.~$(3a)$ and $(3b)$$)$ be the first $($resp.~strong third$)$ homotopy moves defined by Fig.~\ref{Fig1a1b} $($resp.~Fig.~\ref{Fig3a3b}$)$.  
\begin{figure}[htbp]
\includegraphics[width=3cm]{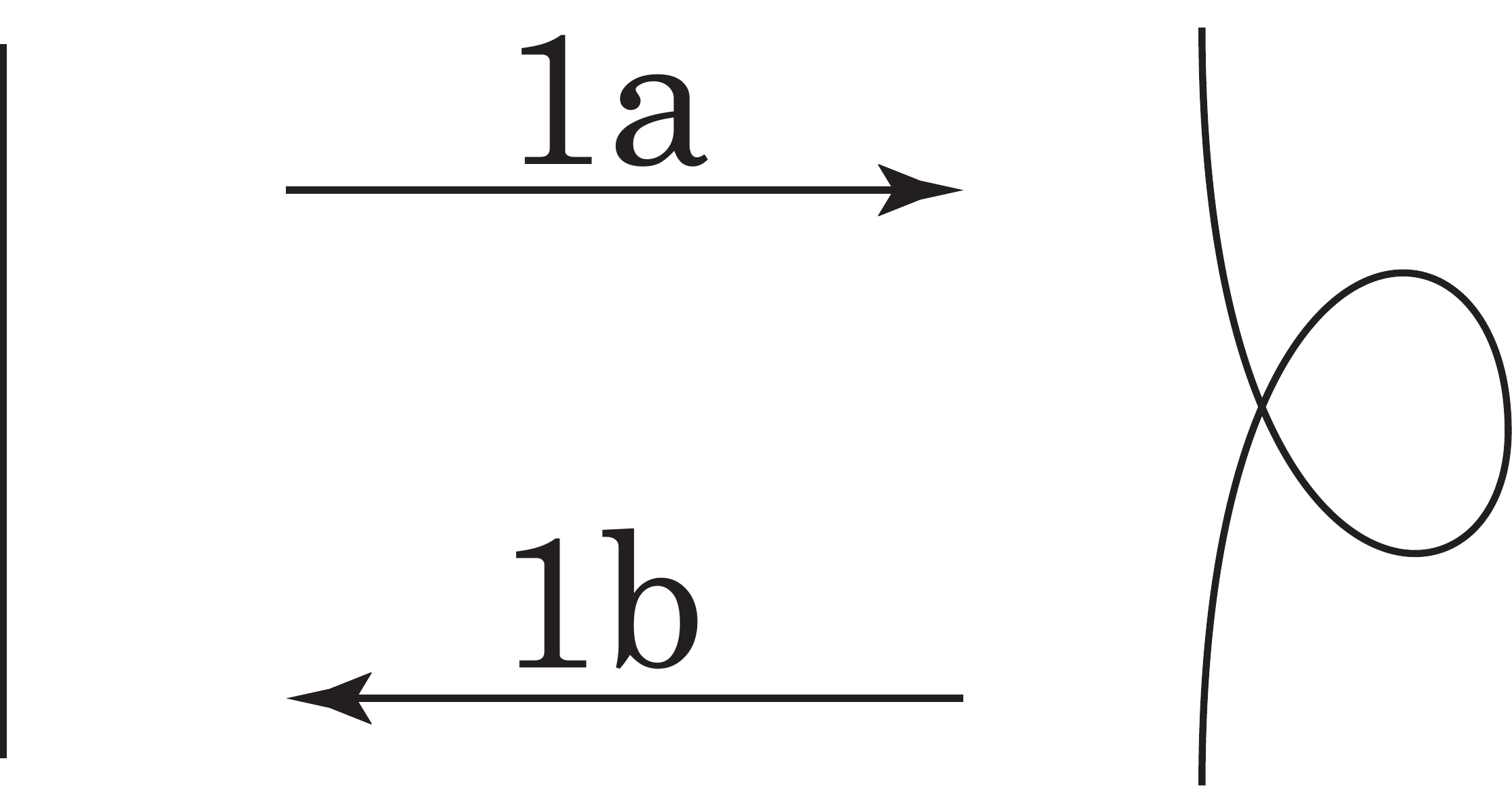}
\caption{Local move $(1a)$ and its inverse $(1b)$ of the first homotopy moves.}\label{Fig1a1b}
\end{figure}
\begin{figure}[htbp]
\includegraphics[width=4cm]{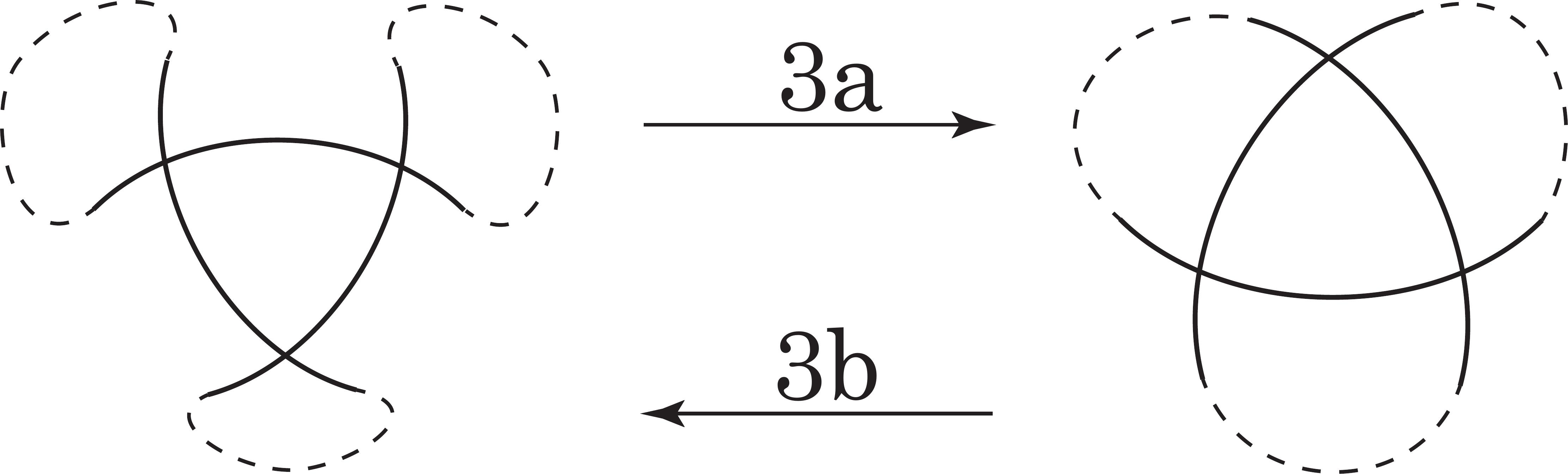}
\caption{Local move $(3a)$ and its inverse $(3b)$ of the strong third homotopy move.}\label{Fig3a3b}
\end{figure}
Let $P$ be a knot projection.  If $P$ is equivalent to the trivial knot projection \begin{picture}(10,15)
\put(5,3){\circle{8}}
\end{picture} under strong $(1, 3)$ homotopy, then $P$ is obtained from
\begin{picture}(10,15)
\put(5,3){\circle{8}}
\end{picture}
by a finite sequence of moves of type $(1a)$ and $(3a)$.  
\end{proposition}
\noindent {{\bf{Proof of Proposition \ref{lemStrong}.}}}
The style of this proof follows that of \cite[Proof of Lemma 1]{ito&takimura1}.  

Let $n$ be an arbitrary integer greater than $1$.  Let $w$ be a sequence of $n-2$ moves consisting of $(1a)$ and $(3a)$.  We use the convention that the sequence $w$ followed by one $(1a)$ move is denoted by $w(1a)$.  For the other moves, (e.g.~$(1b)$, $(3a)$, or $(3a)(1b)$), the same convention applies (e.g.,~$w(1b)$, $w(3a)$, or $w(3a)(1b)$).  Let $P_i$ be the $i$th knot projection appearing in the sequence of the first and strong third homotopy moves of length $n$.  In the following discussion, we often use the symbol $Q$, which stands for a knot projection.  We also use the convention that if the sequence $w(1a)(1b)$ can be replaced with $w$, we denoted this by $w(1a)(1b)$ $=$ $w$.  We apply the same convention to all similar cases that appear in the following.  

Below, we make claims about the four cases of the first appearance of $(1b)$ or $(3b)$ in the sequence $P_1$ $\to$ $P_2$ $\to \dots \to$ $P_{n-1}$ $\to$ $P_n$ $\to$ $P_{n+1}$ of the first and strong third homotopy moves.  
\begin{itemize}
\item Case 1: $w(1a)(1b)$ $=$ $w(1b)(1a)$, 
\item Case 2: $w(1a)(3b)$ $=$ $w(3b)(1a)$, 
\item Case 3: $w(3a)(1b)$ $=$ $w(1b)(3a)$, 
\item Case 4: $w(3a)(3b)$ $=$ $w(3b)(3a)$.  
\end{itemize}

\noindent Case 1: The last two moves $(1a)(1b)$ can be expressed as Fig.~\ref{Case1(1a)(1b)}.  Let $\partial x$ and $\partial y$ be boundaries of $1$-gons as illustrated in Fig.~\ref{Case1(1a)(1b)}.  
\begin{figure}[htbp]
\includegraphics[width=4cm]{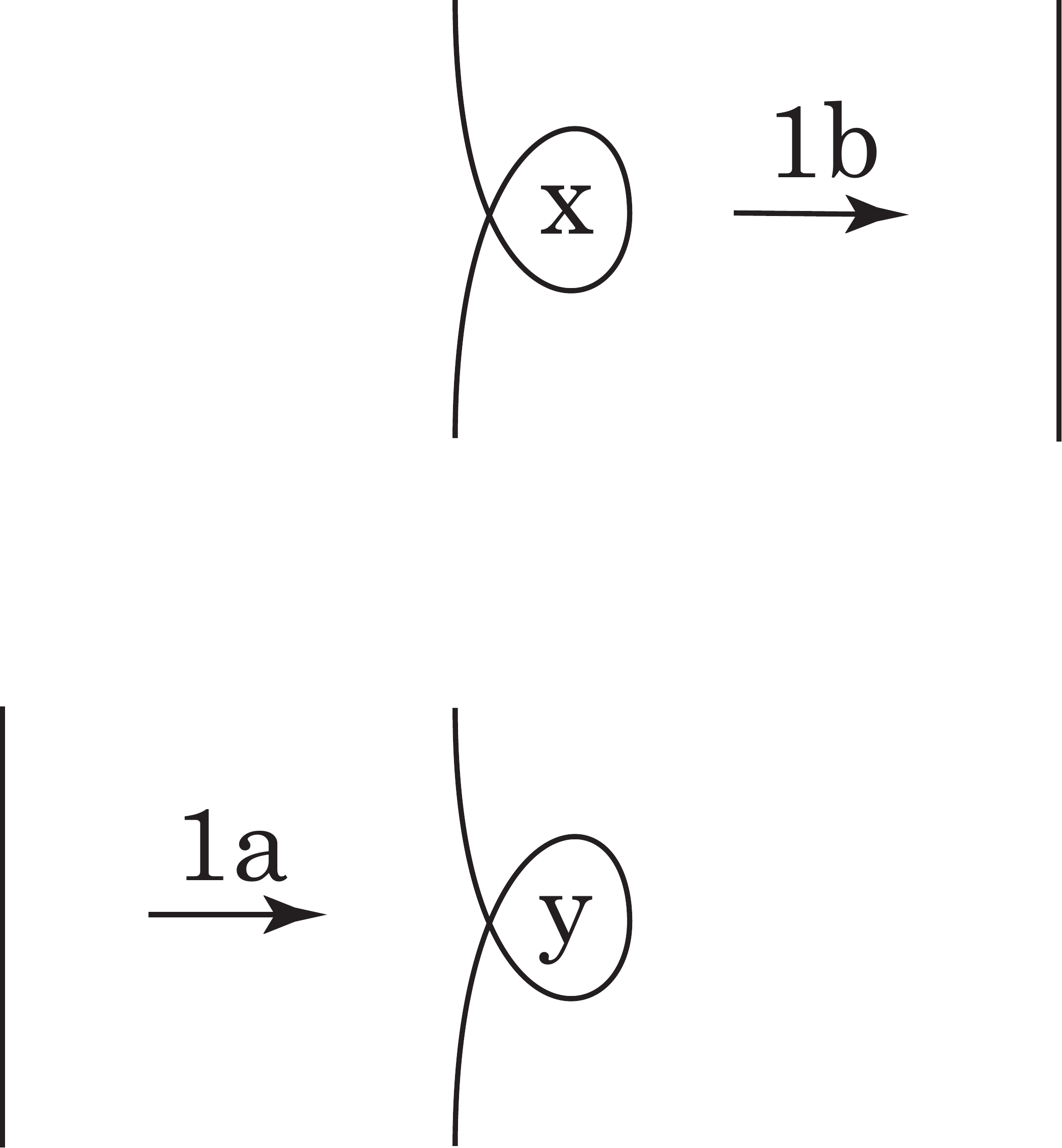}
\caption{The last two moves $(1a)(1b)$ of $w(1a)(1b)$ of Case 1.}\label{Case1(1a)(1b)}
\end{figure}
\begin{enumerate}[(i)]
\item \label{1i} If $\partial x \cap \partial y$ $\neq$ $\emptyset$, then there are two cases of the pair $\partial x$ and $\partial y$, as in Fig.~\ref{Case1-i}.  In both cases, by Fig.~\ref{Case1-i}, we have $w(1a)(1b)$ $=$ $w$.  
\begin{figure}[htbp]
\includegraphics[width=6cm]{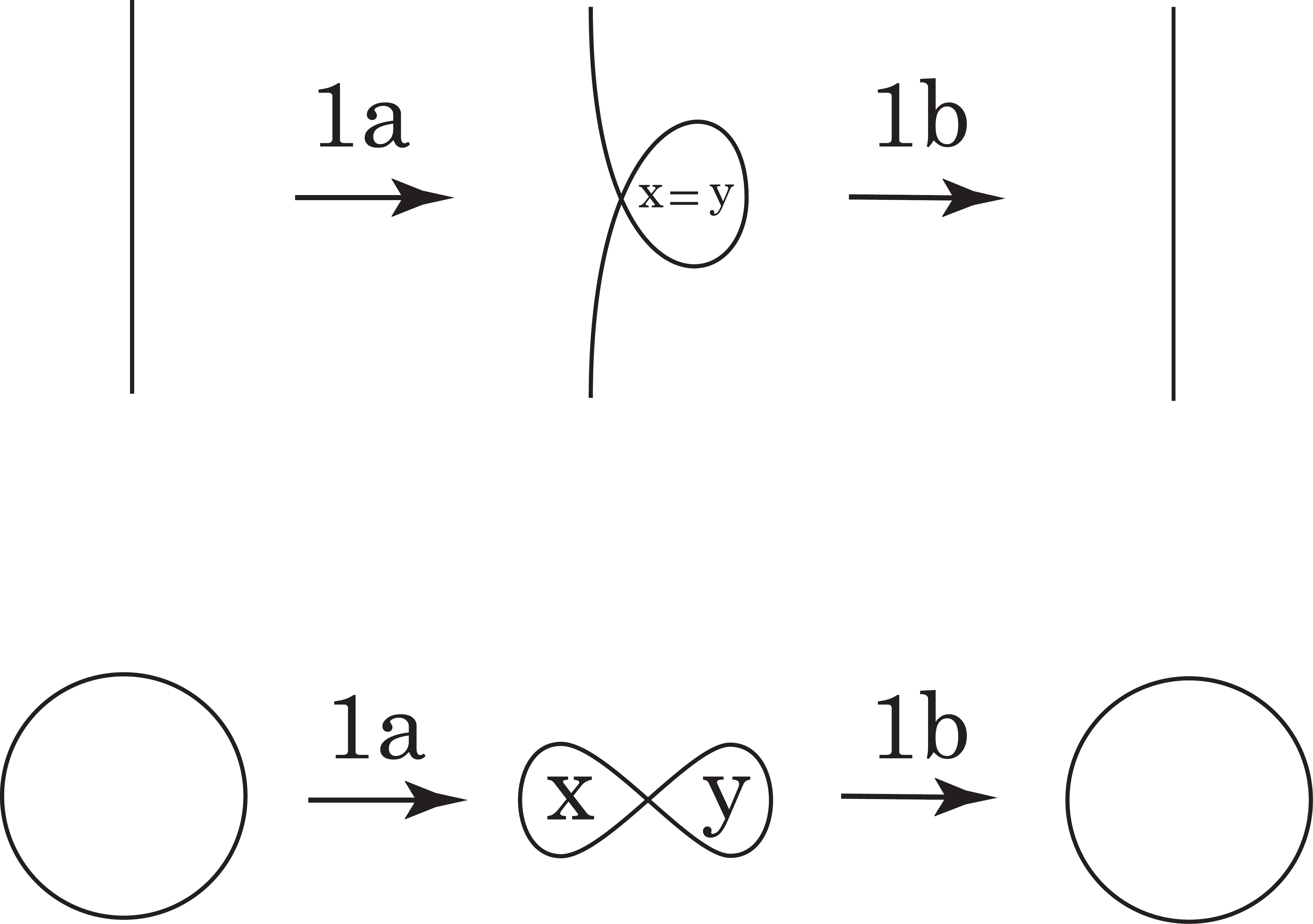}
\caption{Case 1-(\ref{1i}).  Sequence $w(1a)(1b)$ $=$ $w$ in each case of $\partial x$ $=$ $\partial y$ (upper) and $\partial x \cap \partial y$ $=$ $\{ {\text{one~vertex}} \}$.  }\label{Case1-i}
\end{figure}
\item \label{1ii} If $\partial x \cap \partial y$ $=$ $\emptyset$, by Fig.~\ref{Case1-ii}, we have $w(1a)(1b)$ $=$ $w(1a)(1b)(1b)(1a)$ $=$ $w(1b)(1a)$.  
\begin{figure}[htbp]
\includegraphics[width=5cm]{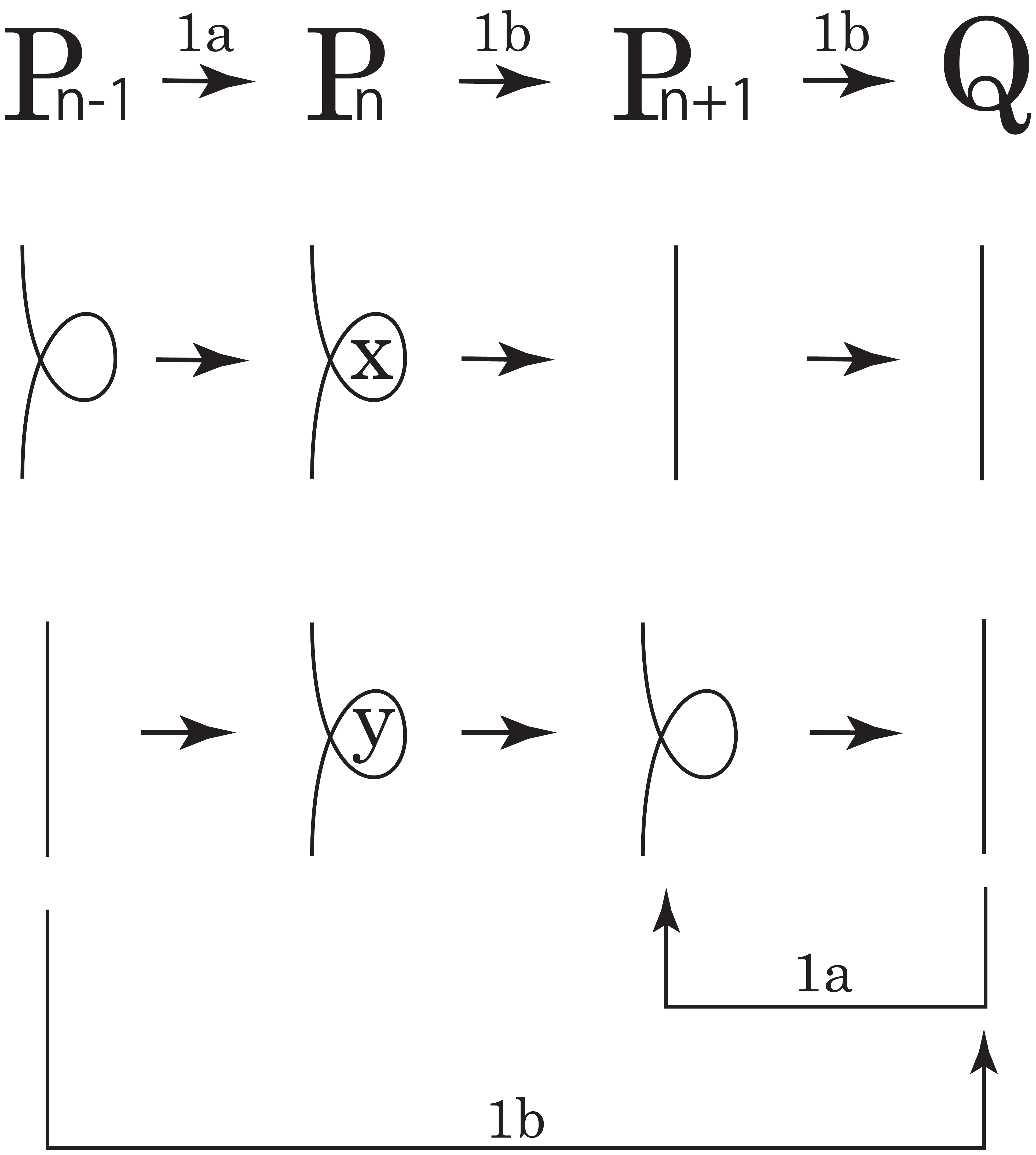}
\caption{Case 1-(\ref{1ii}).  The figure shows $w(1a)(1b)$ $=$ $w(1b)(1a)$.}\label{Case1-ii}
\end{figure}
\end{enumerate}
\noindent Case 2: The last two moves $(1a)(3b)$ of $w(1a)(3b)$ can be expressed as in Fig.~\ref{Case2(1a)(3b)}.  Let $\partial x$ be the boundary of $3$-gon $x$ and $\partial y$ be the boundary of $1$-gon $y$, as in Fig.~\ref{Case2(1a)(3b)}.  
\begin{figure}
\includegraphics[width=6cm]{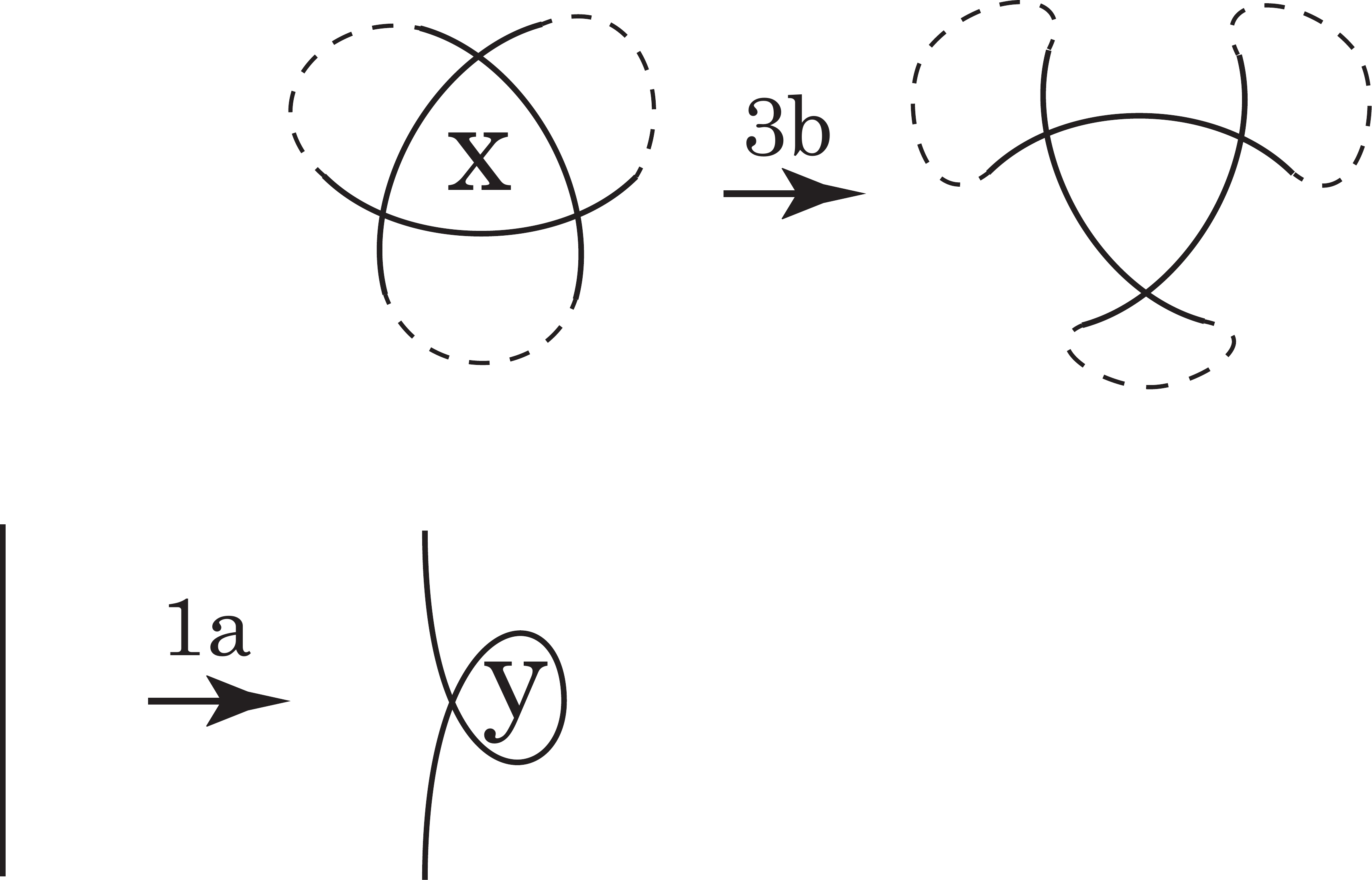}
\caption{The last two moves $(1a)(3b)$ of $w(1a)(3b)$ of Case 2.}\label{Case2(1a)(3b)}
\end{figure}
\begin{enumerate}[(i)]
\item \label{2i} Consider the case $\partial x \cap \partial y$ $\neq$ $\emptyset$.  In fact, this case does not occur.  
\item \label{2ii} If $\partial x \cap \partial y$ $=$ $\emptyset$, we have $w(1a)(3b)$ $=$ $w(1a)(3b)(1b)(1a)$ $=$ $w(3b)(1a)$ by Fig.~\ref{Case2-ii}.  
\begin{figure}[htbp]
\includegraphics[width=8cm]{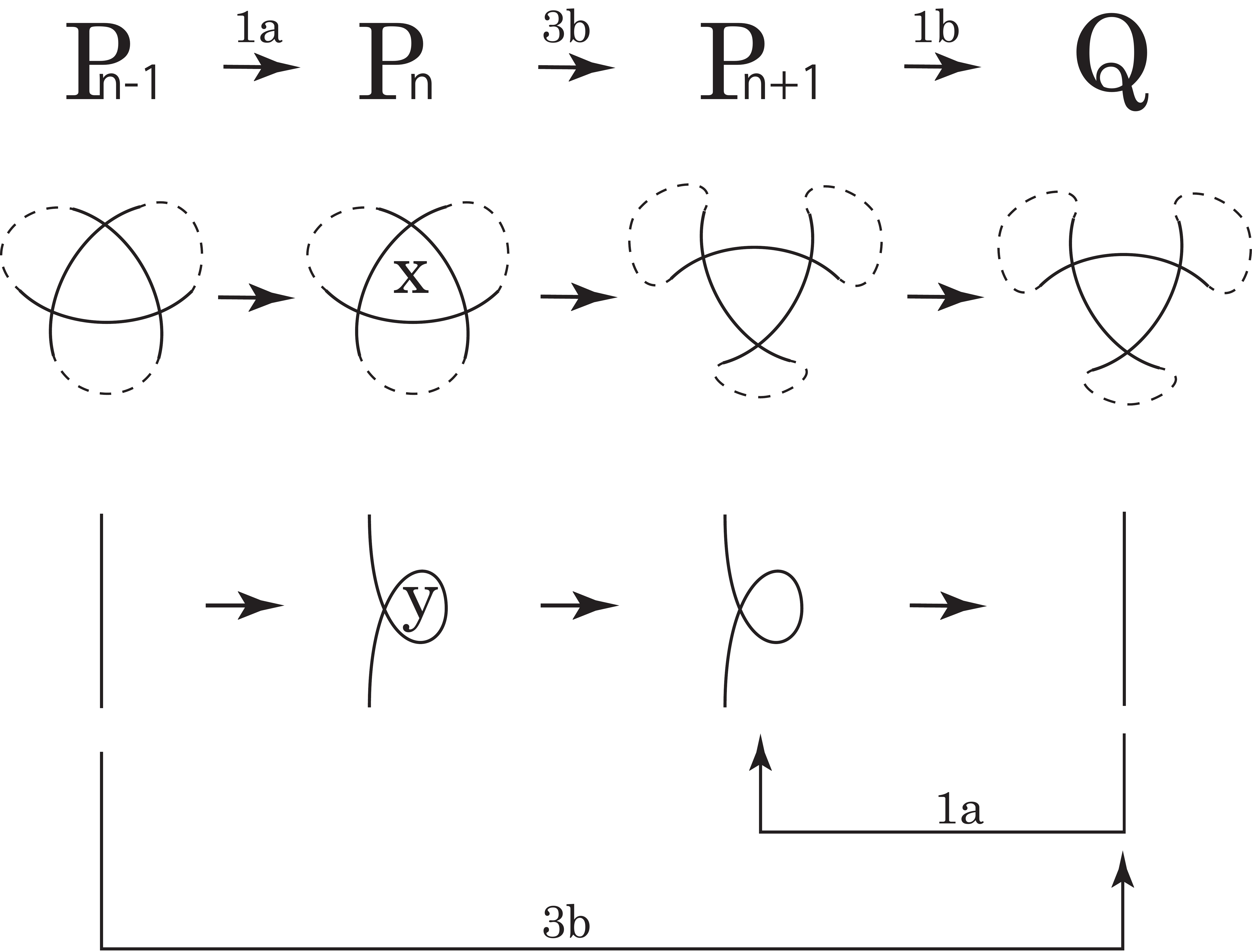}
\caption{Case 2-(\ref{2ii}).  The figure shows $w(1a)(3b)$ $=$ $w(3b)(1a)$.}\label{Case2-ii}
\end{figure}
\end{enumerate}
\noindent Case 3: The last two moves $(3a)(1b)$ are expressed as in Fig.~\ref{Case3(3a)(1b)}.  Let $\partial x$ be the boundary of $1$-gon $x$ and $\partial y$ be the boundary of $3$-gon $y$, as shown in Fig.~\ref{Case3(3a)(1b)}.  
\begin{figure}[htbp]
\includegraphics[width=8cm]{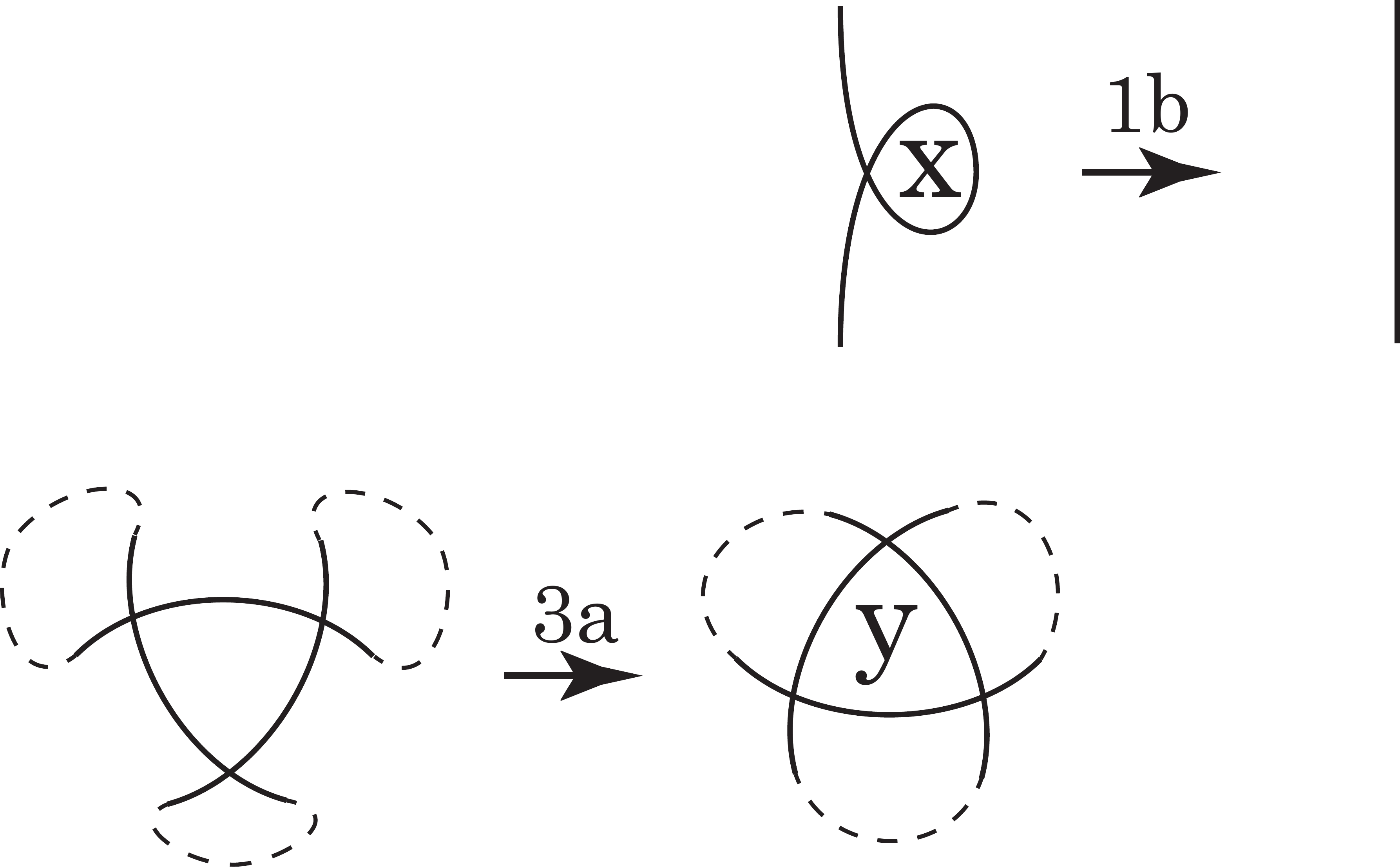}
\caption{The last two moves $(3a)(1b)$ of $w(3a)(1b)$ of Case 3.}\label{Case3(3a)(1b)}
\end{figure}
\begin{enumerate}[(i)]
\item If we consider the case $\partial x \cap \partial y$ $\neq$ $\emptyset$, there is no possibility of realizing the case.  
\item \label{3ii} If $\partial x \cap \partial y$ $=$ $\emptyset$, $w(3a)(1b)$ $=$ $w(3a)(1b)(3b)(3a)$ $=$ $w(3b)(1a)$, as shown in Fig.~\ref{Case3-ii}.  
\begin{figure}[htbp]
\includegraphics[width=8cm]{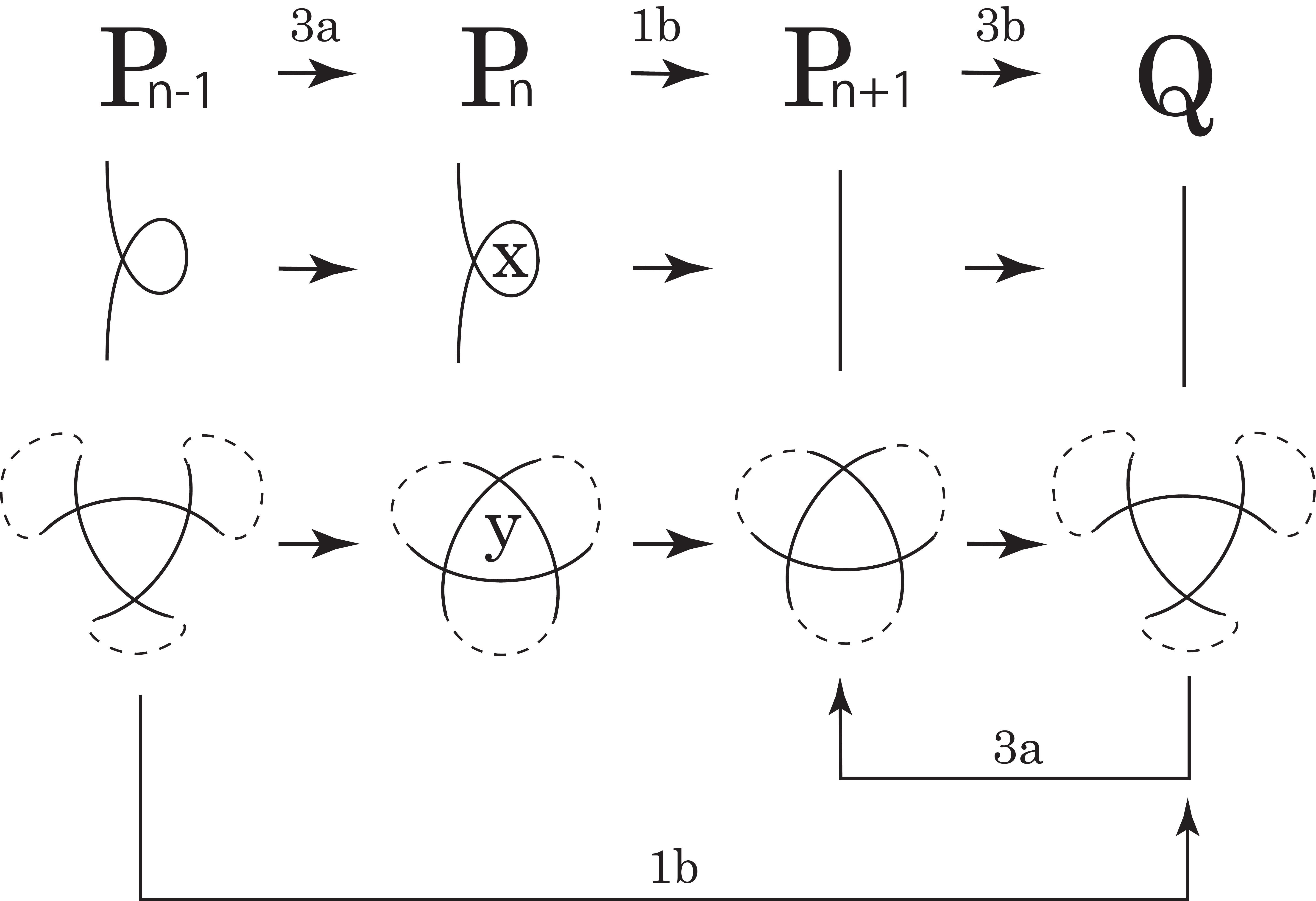}
\caption{Case 3-(\ref{3ii}).  The figure shows $w(3a)(1b)$ $=$ $w(1b)(3a)$.}\label{Case3-ii}
\end{figure}
\end{enumerate}
\noindent Case 4: The last two moves $(3a)(3b)$ of $w(3a)(3b)$ are illustrated as in Fig.~\ref{Case4(3a)(3b)}.  Let $\partial x$ be the boundary of $3$-gon $x$ and $\partial y$ be the boundary of $3$-gon $y$, as shown in Fig.~\ref{Case4(3a)(3b)}.  
\begin{figure}[htbp]
\includegraphics[width=8cm]{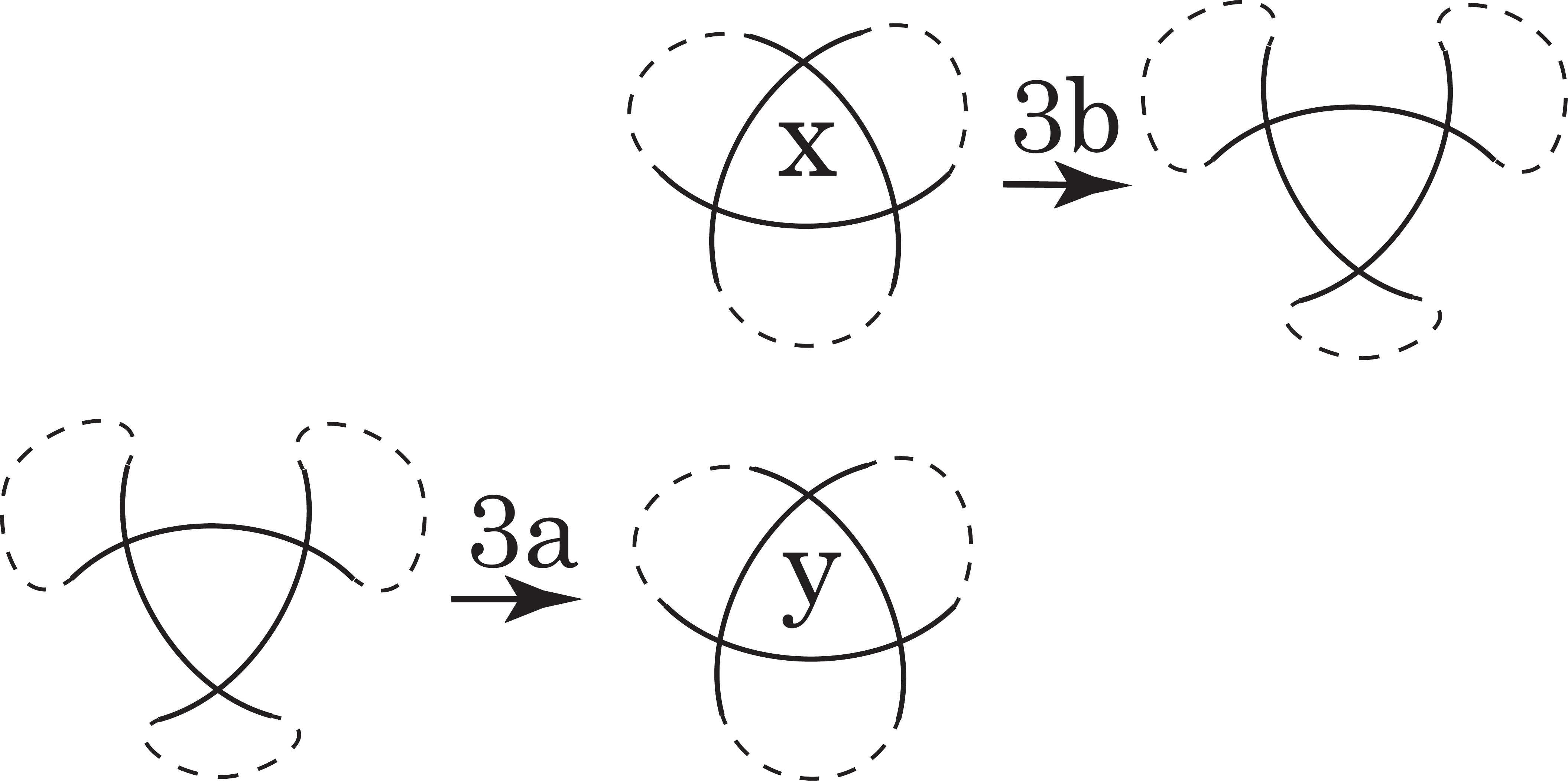}
\caption{The last two moves $(3a)(3b)$ of $w(3a)(3b)$ of Case 4.}\label{Case4(3a)(3b)}
\end{figure}
\begin{enumerate}[(i)]
\item \label{4i} If $\partial x \cap \partial y$ $\neq$ $\emptyset$, then there are three cases of the pair $x$ and $y$, as shown in Fig.~\ref{Case4-i}.  
\begin{figure}[htbp]
\includegraphics[width=7cm]{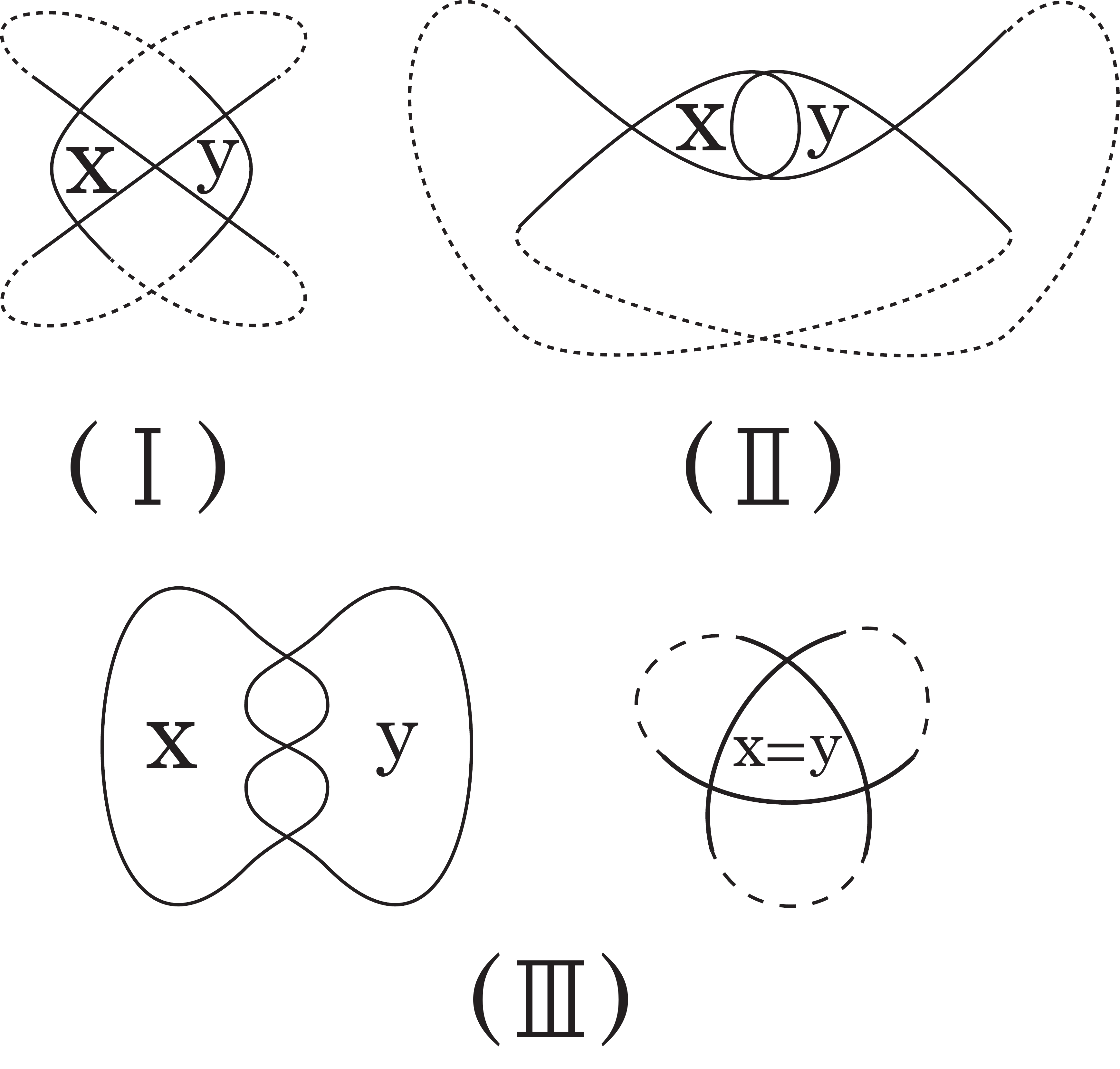}
\caption{Case 4-(\ref{4i}).  Case (I) $\partial x \cap \partial y$ $=$ $\{ {\text{one~vertex}} \}$, (I\!I) $\partial x \cap \partial y$ $=$ $\{ {\text{two~vertices}} \}$, and (I\!I\!I) $\partial x \cap \partial y$ $\supseteq$ $\{ \text{three~vertices} \}$.  }\label{Case4-i}
\end{figure}
The former two cases (I) and (I\!I) in Fig.~\ref{Case4-i} do not appear if the starting diagram of $w(3a)(3b)$ is the trivial knot projection, as shown in Lemma \ref{Lem(3a)(3b)}.  In the last case (I\!I\!I), we have $w(3a)(3b)$ $=$ $w$, by Fig.~\ref{Case4-i(c)}.  
\begin{figure}[htbp]
\includegraphics[width=7cm]{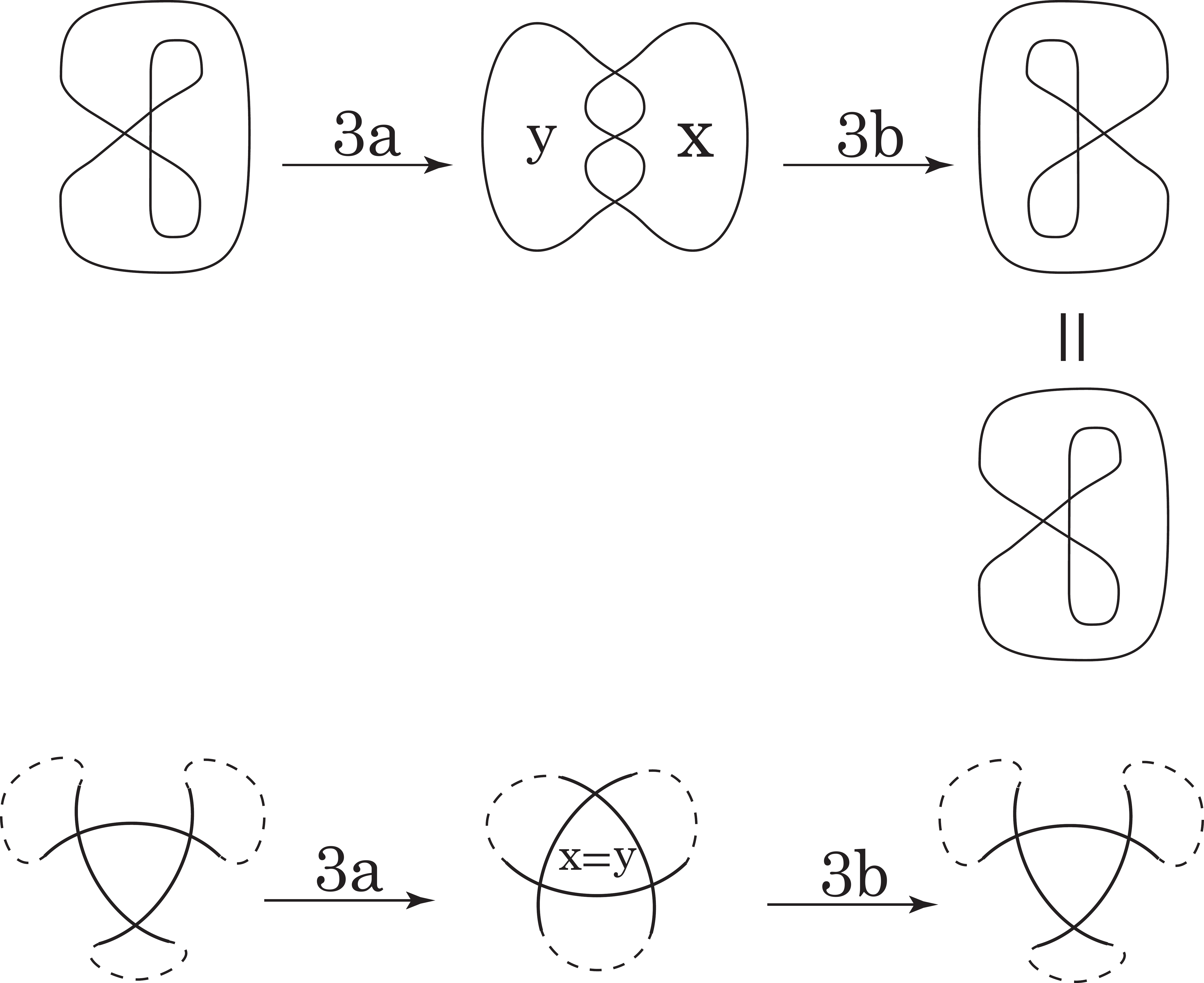}
\caption{Case 4-(\ref{4i})-(I\!I\!I).  $w(3a)(3b)$ $=$ $w$.  }\label{Case4-i(c)}
\end{figure}
\item \label{4ii} If $\partial x \cap \partial y$ $=$ $\emptyset$, Fig.~\ref{Case4-ii} shows $w(3a)(3b)$ $=$ $w(3a)(3b)(3b)(3a)$ $=$ $w(3b)(3a)$.  
\begin{figure}
\includegraphics[width=8cm]{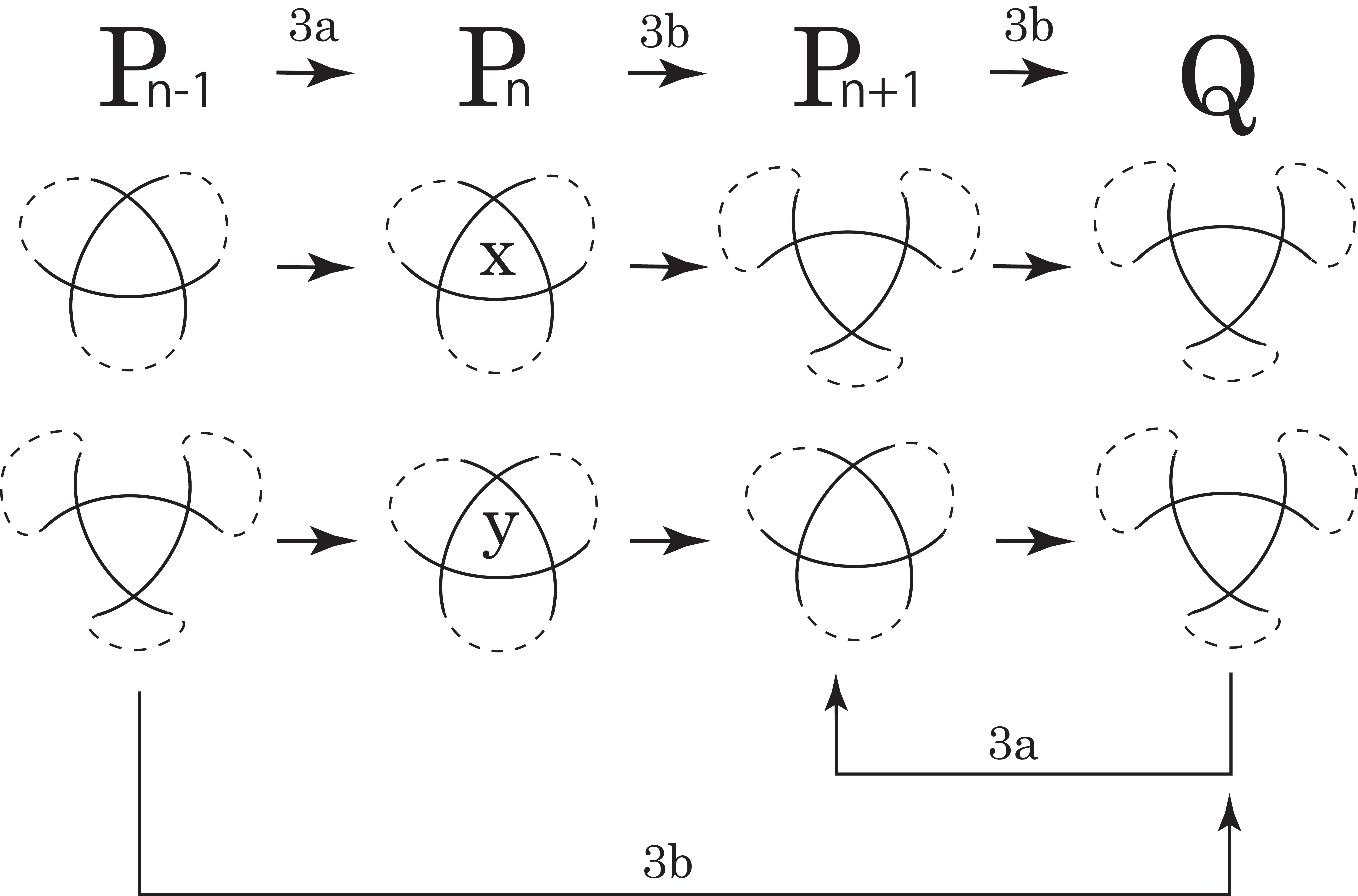}
\caption{Case 4-(\ref{4ii}).  The figure shows $w(3a)(3b)$ $=$ $w(3b)(3a)$.  }\label{Case4-ii}
\end{figure}
\end{enumerate}
\begin{lemma}\label{H_invariant}
Let $P$ be a knot projection.  Define the map $H : \{{\text{knot projections}}\} \to \{0, 1\}$ by setting $H(P)$ $=$ $1$ $($or $0$$)$ if and only if $C\!D_{P}$ contains $($or does not contain$)$ the sub-chord diagram \begin{picture}(10,10)
\put(5,3){\circle{10}}
\qbezier(0,3)(5,3)(10,3)
\qbezier(3,7)(3,3)(3,-1.5)
\qbezier(7,7)(7,3)(7,-1.5)
\end{picture}.  The map $H(P)$ is invariant under strong $(1, 3)$ homotopy.  
\end{lemma}
\noindent {\bf{Proof of Lemma \ref{H_invariant}.}}  The sub-chord diagram \begin{picture}(10,10)
\put(5,3){\circle{10}}
\qbezier(0,3)(5,3)(10,3)
\qbezier(3,7)(3,3)(3,-1.5)
\qbezier(7,7)(7,3)(7,-1.5)
\end{picture} is called an $H$ {\it{chord}}.  First, it is easy to see the application of any first homotopy move does not create or dissolve the $H$ chords.  Second, from Fig.~\ref{proof_chord_strong}, we infer that $C\!D_{P_3}$ contains $H$ chords if and only if $C\!D_{P_4}$ contains $H$ chords.  

\noindent({\bf{End of Proof of Lemma \ref{H_invariant}.}})
\begin{lemma}\label{Lem(3a)(3b)}
Assume that $P$ is a knot projection that results from the application of the local moves $(1a)$, $(1b)$, $(3a)$, and $(3b)$ to the trivial knot projection \begin{picture}(10,15) \put(5,2){\circle{8}} \end{picture}.  Then, $P$ can be neither $({\mathrm{I}})$ nor $(\mathrm{I\!I})$ of Fig.~\ref{Case4-i}.  
\end{lemma}
\noindent {\bf{Proof of Lemma \ref{Lem(3a)(3b)}.}}
To illustrate the claim of Lemma \ref{Lem(3a)(3b)}, we check the claim for each case.  
\begin{enumerate}[(I)]
\item \label{case4iI} For the knot projection $P$ on the left-hand side of Fig.~\ref{Case4-i-I}, we obtain $H(P)$ $=$ $1$ using the right-hand side of Fig.~\ref{Case4-i-I}.  
\begin{figure}[htbp]
\includegraphics[width=6cm]{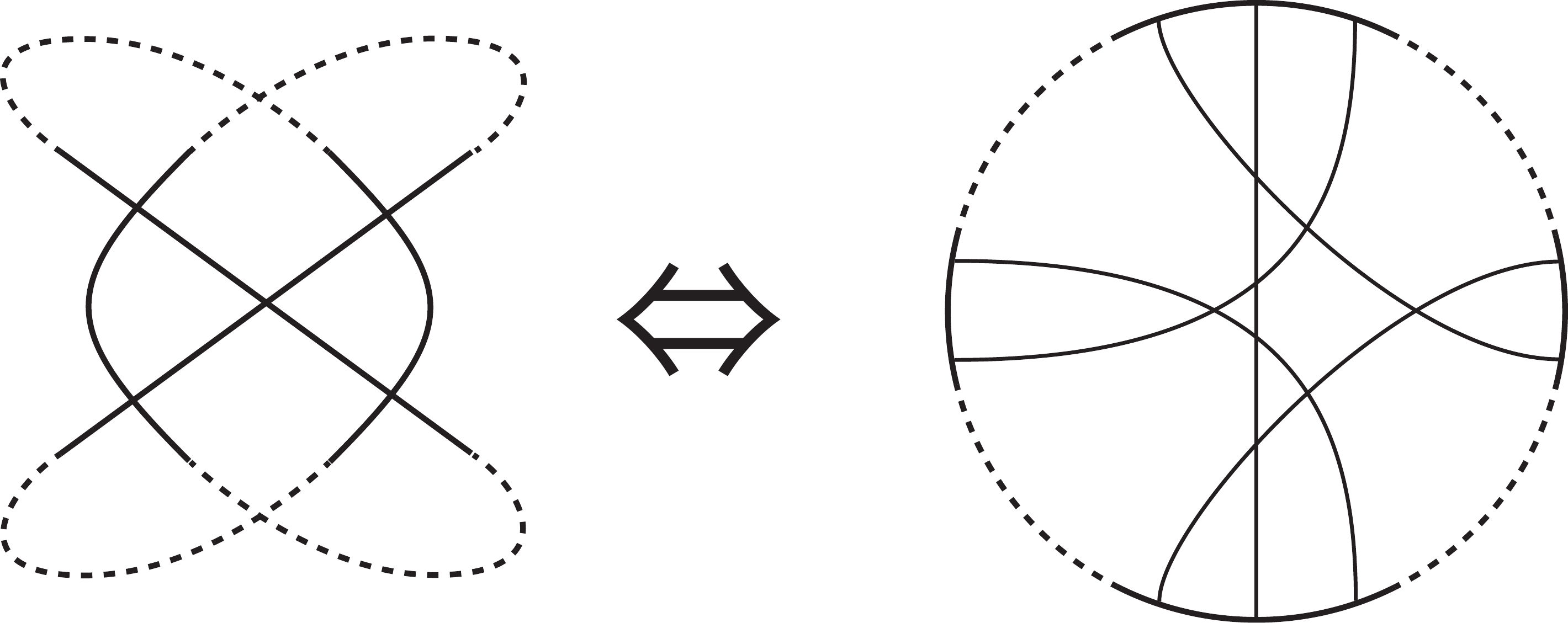}
\caption{Case 4-(\ref{4i})-(\ref{case4iI}).}\label{Case4-i-I}
\end{figure}
\item \label{case4iII} For a knot projection $P$ at the left of Fig.~\ref{Case4-i-II}, we have $H(P)$ $=$ $1$ using Fig.~\ref{Case4-i-II}.  
\begin{figure}[htbp]
\includegraphics[width=6cm]{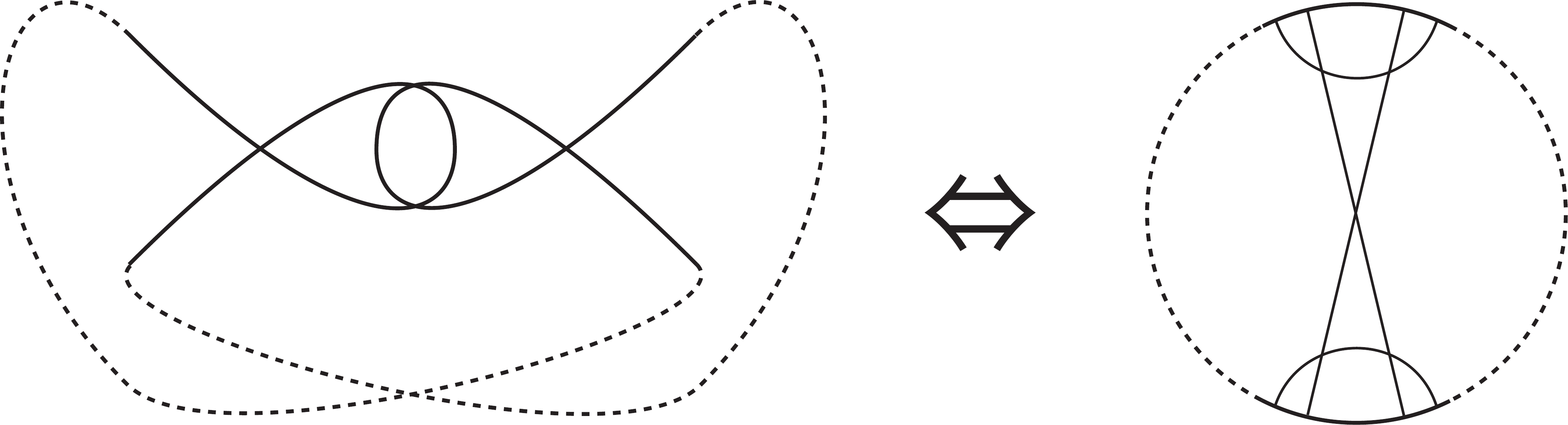}
\caption{Case 4-(\ref{4i})-(\ref{case4iII}).}\label{Case4-i-II}
\end{figure}
\end{enumerate}
For the trivial knot projection $\begin{picture}(10,15) \put(5,3){\circle{8}} \end{picture}$, we have the formula $H(\begin{picture}(10,15) \put(5,3){\circle{8}} \end{picture})$ $=$ $0$.   This formula completes the proof.  

\noindent({ \bf{End of Proof of Lemma \ref{Lem(3a)(3b)}.} })

Lemma \ref{Lem(3a)(3b)} completes the proof of Proposition \ref{lemStrong}.  

\noindent ({ \bf{End of Proof of Proposition \ref{lemStrong}.} })

Finally, we will prove the claim of Theorem \ref{strongTrivial_thm}.  Below, we denote the connected sum of the trivial knot projection \begin{picture}(10,15) \put(5,2){\circle{8}} \end{picture}, the knot projection that appears similar to $\infty$, and the trefoil knot projection as $\sharp \{U, \infty, T \}$.  This is shown by induction on the length of a sequence consisting of $(3a)$ and $(1a)$.  When $n$ $=$ $1$, the knot projection is nothing but the one that appears similar to $\infty$.  Further, when the length is equal to $n$, we assume that the abovementioned claim holds.  Therefore, by this assumption of induction, the knot projection $P_n$ belongs to $\sharp \{U, \infty, T\}$.  Consider $P_{n+1}$ and let $w$ be a sequence of moves (length: $n$) consisting of $(1a)$ and $(3a)$.  If $P_{n+1}$ is obtained by $w(1a)$, it is easy to see that $P_{n+1}$ belongs to $\sharp \{U, \infty, T \}$.  If $P_{n+1}$ is obtained by $w(3a)$, the last $(3a)$ can be presented as shown in Fig.~\ref{33}.  
\begin{figure}[htbp]
\includegraphics[width=6cm]{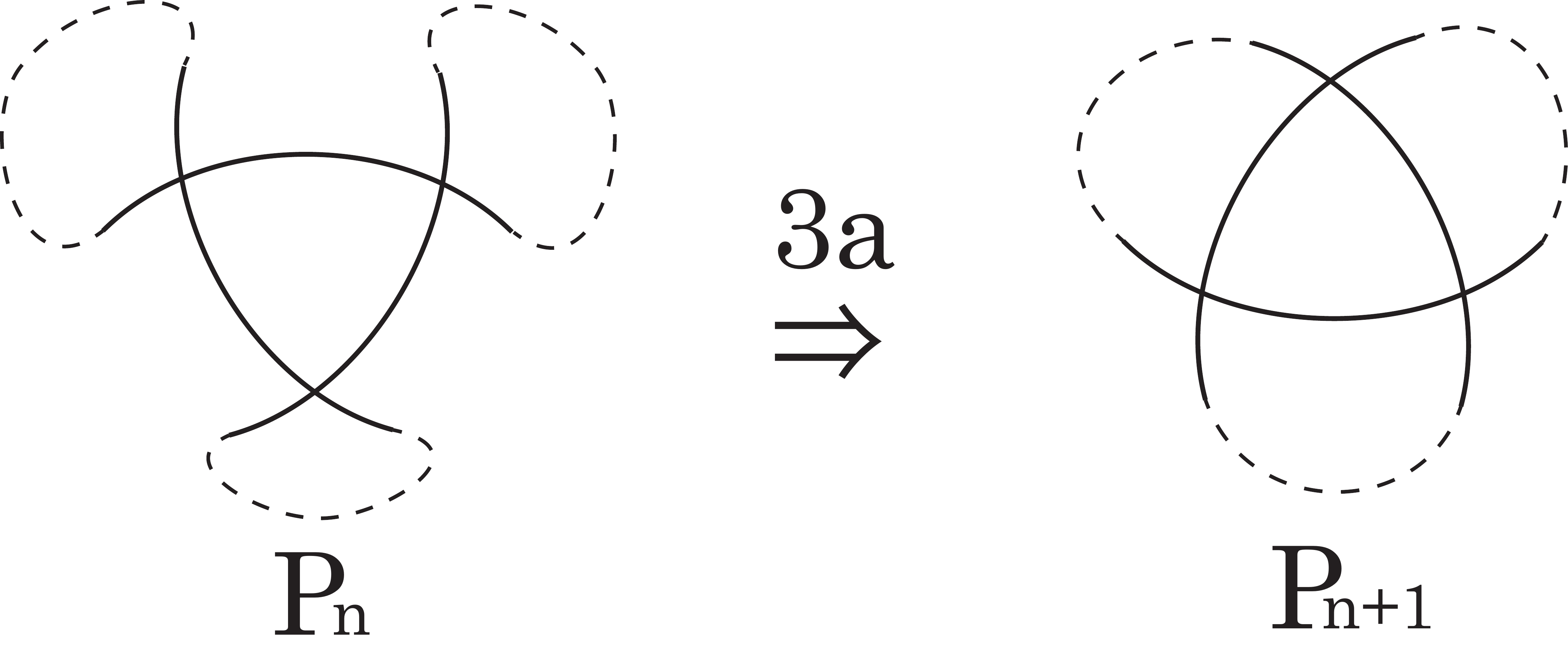}
\caption{$P_n$ $\stackrel{(3a)}{\to}$ $P_{n+1}$.}\label{33}
\end{figure}
Here, by using Lemma \ref{H_invariant}, we obtain $H(P_{i})$ $=$ $0$ for every $i \in \{1, 2, 3, \dots, n, n+1\}$.  Then, as shown on the right-hand side of Fig.~\ref{34}, there is no chord connecting between the two dotted arcs on $C\!D_{P_{n+1}}$.  In other words, there is no double point that consists of two dotted arcs of $P_{n+1}$ as shown on the left-hand side of Fig.~\ref{34}.  
\begin{figure}[htbp]
\includegraphics[width=6cm]{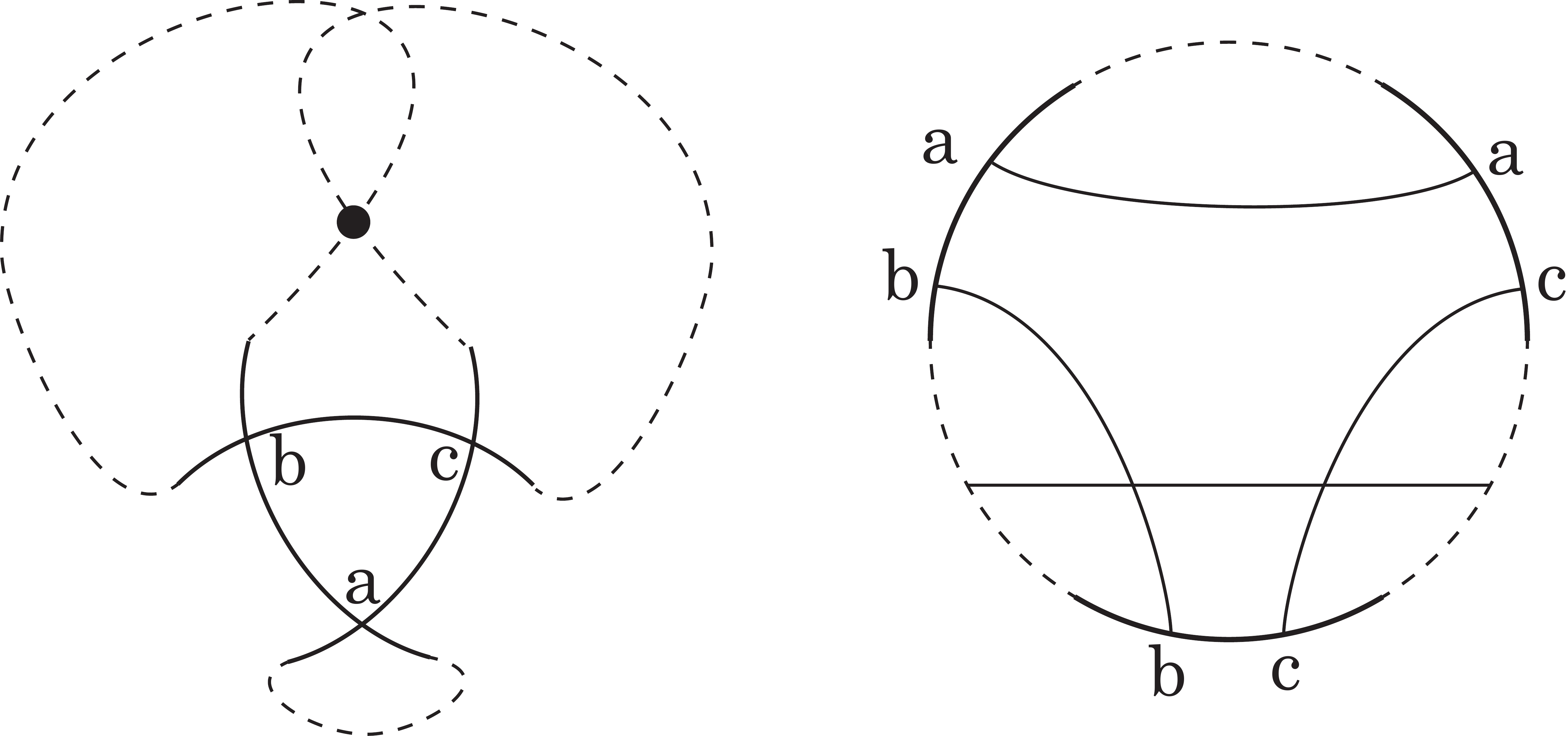}
\caption{Double point consisting of two dotted arcs and the chord connecting the two dotted arcs corresponding to the double point.}\label{34}
\end{figure}

By the assumption of induction, we conclude that $P_{n}$ belongs to $\sharp \{U, \infty, T\}$ and has no double point that contains the two dotted arcs shown on the left-hand side of Fig.~\ref{33}.  After applying $(3a)$ to $P_n$ and maintaining the property that $C\!D_{P_{n+1}}$ has no chord connecting the two dotted arcs on $C\!D_{P_{n+1}}$, we obtain the result $P_{n+1}$ shown on the right-hand side of Fig.~\ref{33}.  Therefore, $P_{n+1}$ belongs to $\sharp \{U, \infty, T\}$.  This completes the proof.  
\end{proof}
\begin{remark}
In the last step $P_n \stackrel{(3a)}{\to} P_{n+1}$ of the above proof of Theorem \ref{strongTrivial_thm}, we provide another proof; this proof is as follows: By the assumption of induction, $P_n$ belongs to $\sharp \{U, \infty, T\}$.  Then, by using \cite[Theorem 3.2]{sakamoto&taniyama}, we determine that $C\!D_{P_n}$ does not contain $H$ chords.  Then, there is no double point that contains two dotted arcs (Fig.~\ref{34}).  Then, from Fig.~\ref{33}, we infer that $P_{n+1}$ does not contain $H$ chords.  Therefore, $P_{n+1}$ belongs to $\sharp \{U, \infty, T\}$.  
\end{remark}
\begin{remark}
For the knot projection $P_{HY}$ defined by the left image of Fig.~\ref{35a}, $H(P_{HY})$ $=$ $1$ (alternatively, $X(P_{HY})$ $\equiv$ $2$ (${\operatorname{mod}}~3$) using $C\!D_{P_{HY}}$ as shown in the right image of Fig.~\ref{35a}).  Then, $P_{HY}$ is not equivalent to the trivial knot projection \begin{picture}(10,15) \put(5,3){\circle{8}} \end{picture} under strong (1, 3) homotopy.  Hagge and Yazinski \cite{hagge&yazinski} claim that $P_{HY}$ cannot be equivalent to the trivial knot projection under (1, 3) homotopy without the use of any numerical invariants.  
\begin{figure}[htbp]
\includegraphics[width=10cm]{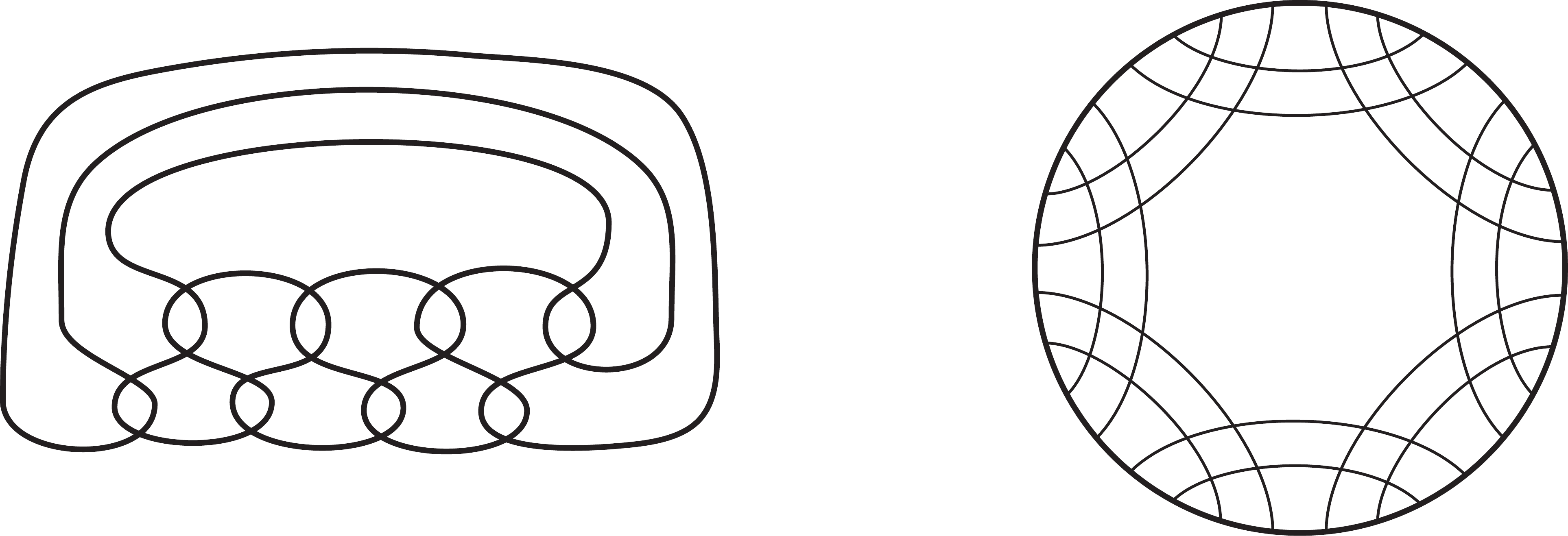}
\caption{Knot projection $P_{HY}$ and $C\!D_{P_{HY}}$.}\label{35a}
\end{figure}
\end{remark}
\begin{remark}
Proposition \ref{lemStrong} provides a finite sequence obtained by $(1a)$ and $(3a)$ from the trivial knot projection \begin{picture}(10,15) \put(5,3){\circle{8}} \end{picture} to a knot projection $P$.  From Theorem \ref{strongTrivial_thm}, the knot projection $P$ belongs to $\sharp \{U, \infty, T\}$.  In fact, Proposition \ref{(3a)times} provides the relation between the sub-chord and the number of $(3a)$ in the sequence.  
\end{remark}
\begin{proposition}\label{(3a)times}
Let $P$ be a knot projection that exists in a finite sequence obtained by $(1a)$ and $(3a)$ from the trivial knot projection \begin{picture}(10,15) \put(5,3){\circle{8}} \end{picture} to $P$.   
The number of $(3a)$ is equal to the number of the sub-chords corresponding to the trefoil projection defined by Fig.~\ref{trefoilProjection}.  Moreover, the number of $(3a)$ is equal to $(J_{S}^{+}(P) + 2 St_{S}(P))/2$ where $J_{S}^{+}(P)$ and $St_{S}(P)$ are Arnold invariants of spherical curves defined by \cite[Page 993, Sec.~2.4]{polyak}.  
\end{proposition}
\begin{proof}
The proof of the former part is shown by induction on the length $n$ of a sequence $P_1$ $\to$ $P_2$ $\to \dots \to$ $P_n$ consisting of $(1a)$ and $(3a)$.  We denote the sub-chord diagram (Fig.~\ref{36} (b)) corresponding to Fig.~\ref{36} (a) by ($\ast$).  
\begin{figure}[htbp]
\includegraphics[width=6cm]{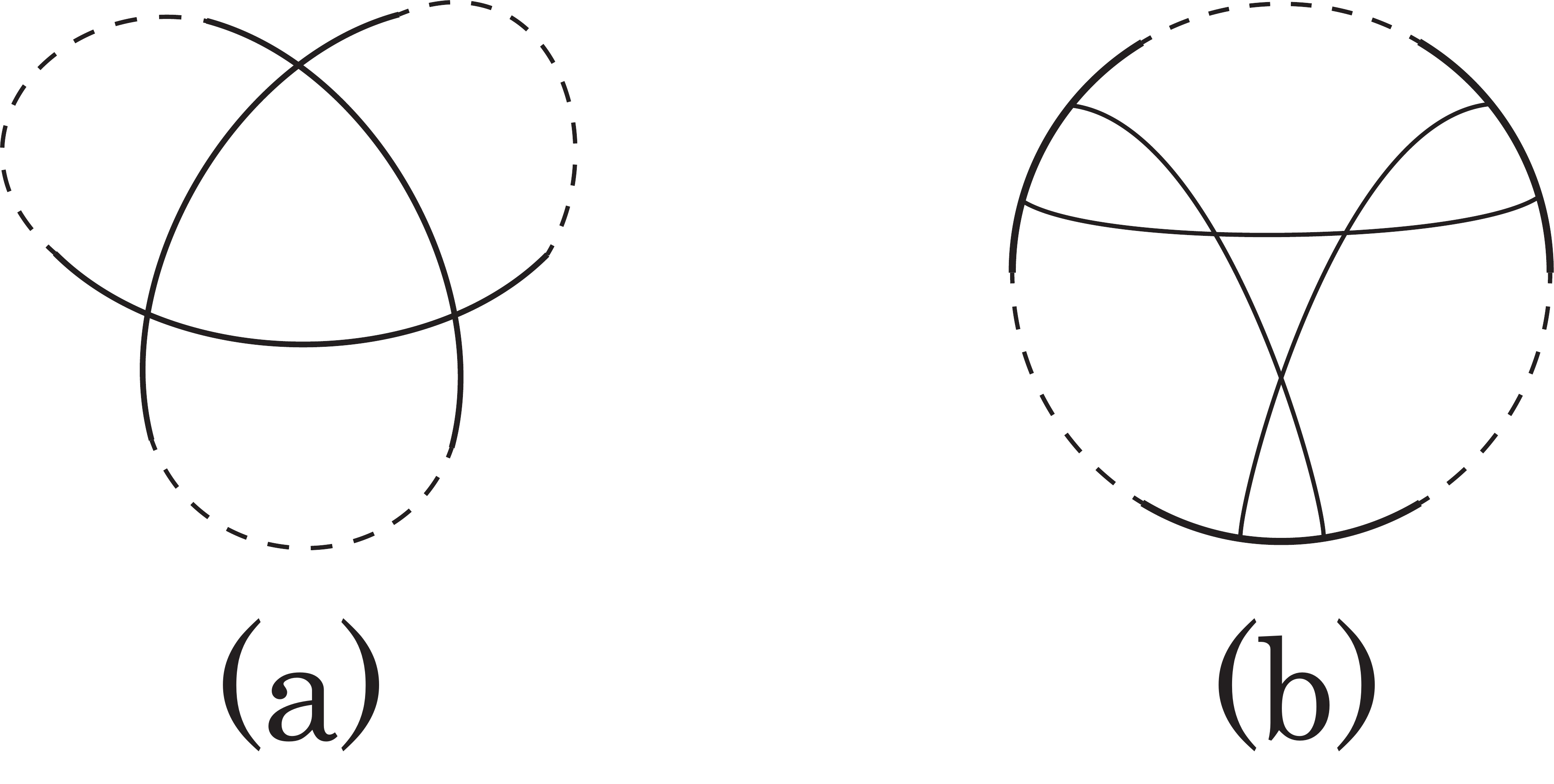}
\caption{Sub-chord diagram (b) corresponding to (a).}\label{36}
\end{figure}
In the case $n$ $=$ $1$, the only move is $(1a)$ and the number of $(3a)$ is $0$.  The number of sub-chords ($\ast$) is also $0$.  Then, in this case, the claim is true.  Now, we assume that the claim is established in case $n$.  Since the initial projection is the trivial knot projection, we obtain $H(P_i)$ $=$ $0$ for any $i$.  Every sub-chord ($\ast$) is placed on the dotted arcs shown on the left-hand side of Fig.~\ref{33}.  From the left to the right in Fig.~\ref{33}, the number of $(3a)$ and the number of sub-chords ($\ast$) increase by $1$.  By the assumption of induction, we conclude that the number of sub-chords ($\ast$) is equal to the number of $(3a)$ in the case of $P_n$.  

Next, we present the proof of the latter part.  Using formula \cite[Page 997, Formula (3)]{polyak}, $(1a)$ does not change $(J_{S}^{+} + 2 St_{S})/2$.  The $(J_{S}^{+} + 2 St_{S})/2$ increases by $1$ on applying one $(3a)$.  This completes the proof.  
\end{proof}
\begin{remark}
M. Polyak defined the Arnold invariants $J_{S}^{+}$ and $St_{S}$ for spherical curves \cite[Page 993, Sec.~2.4]{polyak}.  Readers should be careful because Theorem 1 and Corollary 1 \cite[Page 996, Theorem 1 and Page 997, Corollary 1]{polyak} have typographical errors (see \cite[Page 1217]{ito2}).  For relations between $J_{S}^{+} + 2 St_S$ and the Vassiliev knot invariant, see \cite[Sec.~6.4]{polyak} and \cite{sakamoto&taniyama}.  
\end{remark}

\section{Strong (1, 3) homotopy classes of other knot projections.}\label{strongGeneral}
In this section, we obtain the proof of Theorem \ref{otherClass_thm} via Lemma \ref{generalLemma}.  
\begin{lemma}\label{generalLemma}
Let $P_0$ be a knot projection without $1$-gons, coherent $2$-gons, and $3$-gons shown in Fig.~\ref{37}.  If a knot projection $P$ is equivalent to $P_0$ under strong $(1, 3)$ homotopy, $P$ is the connected sum of $P_0$ and the knot projections. Each knot projection is equivalent to the trivial knot projection \begin{picture}(10,15) \put(5,3){\circle{8}} \end{picture} under strong (1, 3) homotopy.  
\end{lemma}
\begin{proof}
The proof is accomplished by induction on the length $n$ of a sequence $P_{0}$ $\to$ $P_1$ $\to \dots \to P_n$ consisting of $(1a)$, $(1b)$, $(3a)$, and $(3b)$.  The knot projection $P_0$ satisfies the claim since $P_0$ is the connected sum of $P_0$ and the trivial knot projection \begin{picture}(10,15) \put(5,3){\circle{8}} \end{picture}.  An example is shown in Fig.~\ref{38}.  
\begin{figure}[htbp]
\includegraphics[width=3cm]{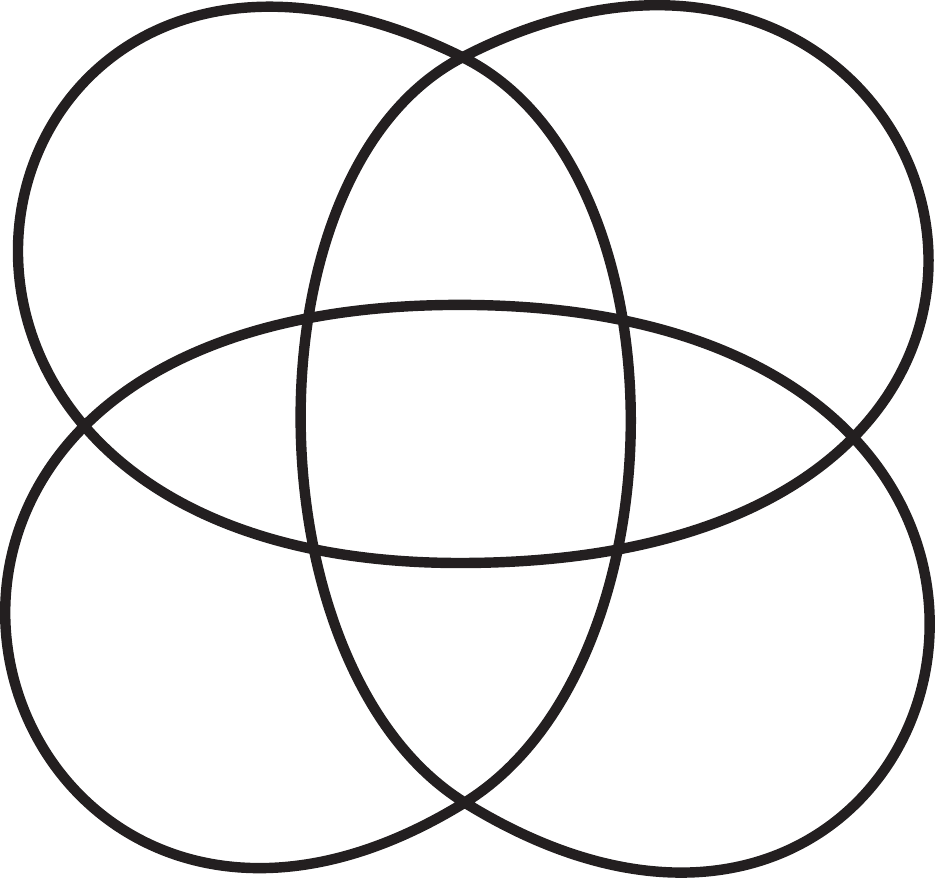}
\caption{Example of $P_0$.}\label{38}
\end{figure}
Let $P_{n+1}$ be the knot projection that we obtain after we apply $(1a)$, $(1b)$, $(3a)$, or $(3b)$ to $P_{n}$.  Assume that $P_n$ satisfies the claim and we prove that $P_{n+1}$ also satisfies the claim under this assumption.  Let $x$ be $(1a)$, $(1b)$, $(3a)$, or $(3b)$ sending $P_n$ $\to$ $P_{n+1}$.  From the definitions of $(1a)$, $(1b)$, $(3a)$, and $(3b)$ (Figs.~\ref{FlatReidemeisterMove}, \ref{triplepoint_perestroika}, and \ref{oriented_strong_weak}), each of $(1a)$, $(1b)$, $(3a)$, and $(3b)$ is a local move within a disk $x_d$ that is called $x$-{\it{disk}}.  We also consider the regular neighborhood $B$ of $P_0$ (Fig.~\ref{40} (b) for (a)) and disks, called $r$-disks, on each edge of $P_0$ apart from double points of $P_0$ (Fig.~\ref{39}(b) for (a)).  Here, we assume that the sets $x$-disk, $B$, and $r$-disks are closed sets.  

\begin{figure}[htbp]
\includegraphics[width=5cm]{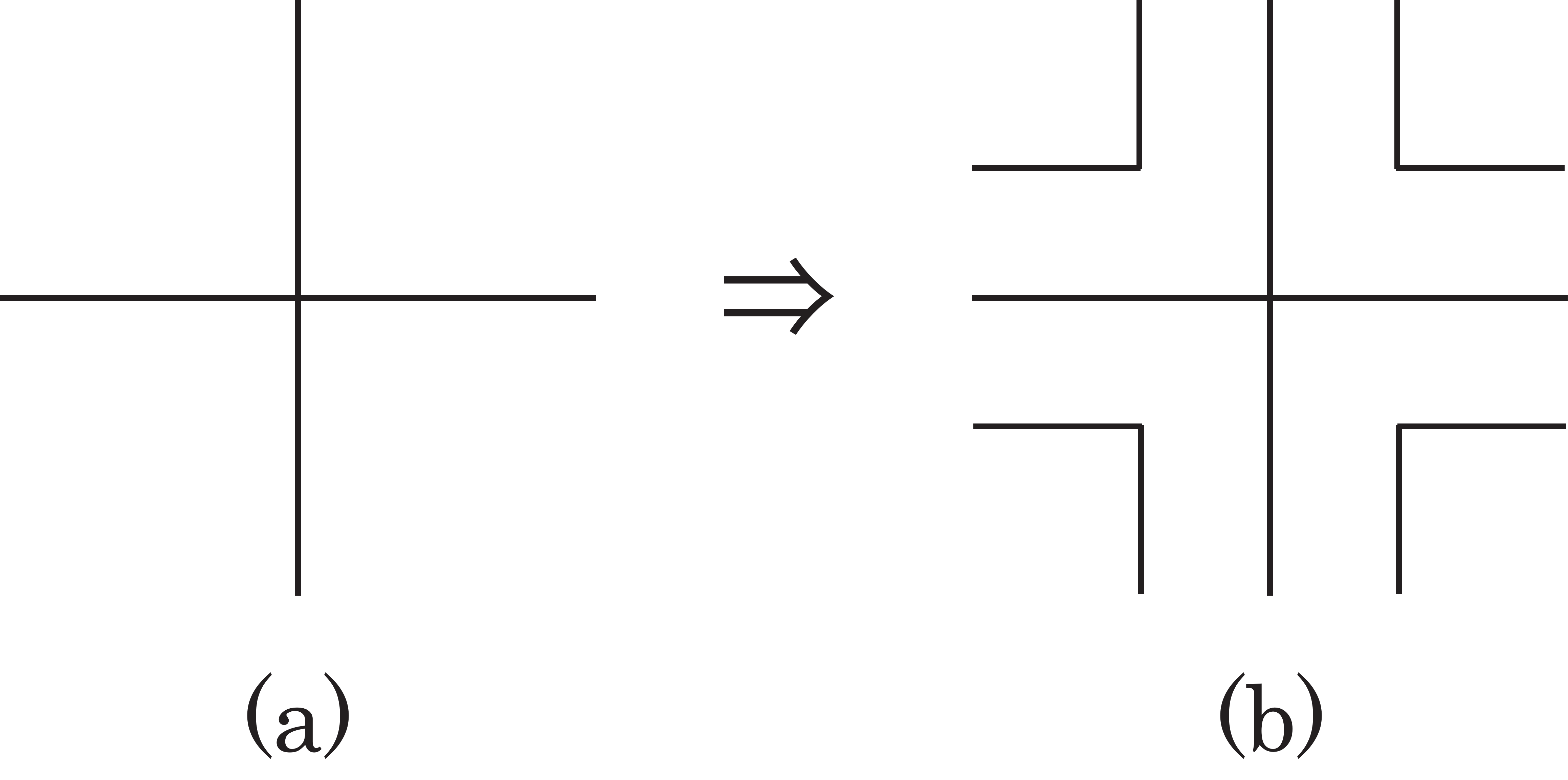}
\caption{Definition of $B$.}\label{40}
\end{figure}
\begin{figure}[htbp]
\includegraphics[width=5cm]{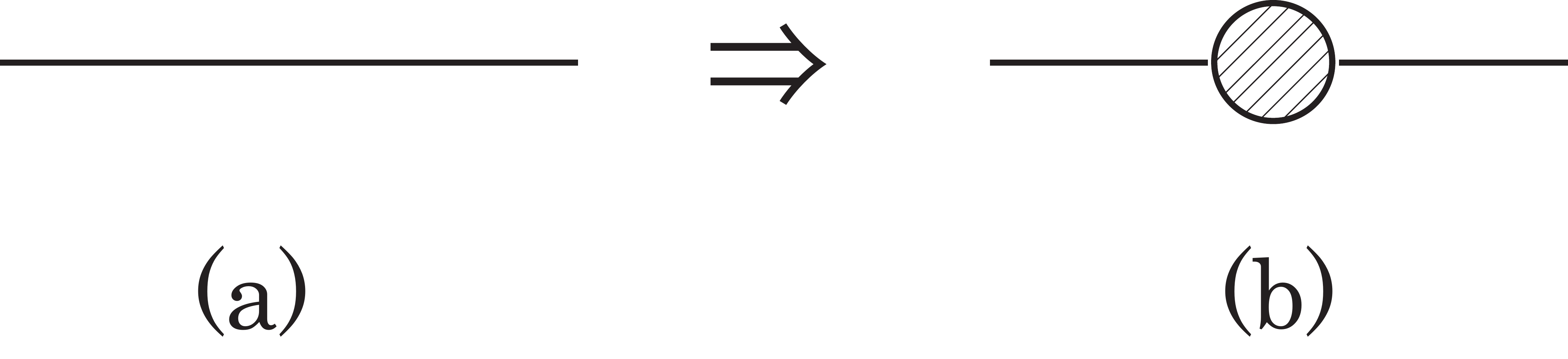}
\caption{Definition of $r$-disks.}\label{39}
\end{figure}
Fig.~\ref{45} gives an example of $B$ and $r$-disks corresponding to Fig.~\ref{38}.  
\begin{figure}[htbp]
\includegraphics[width=3cm]{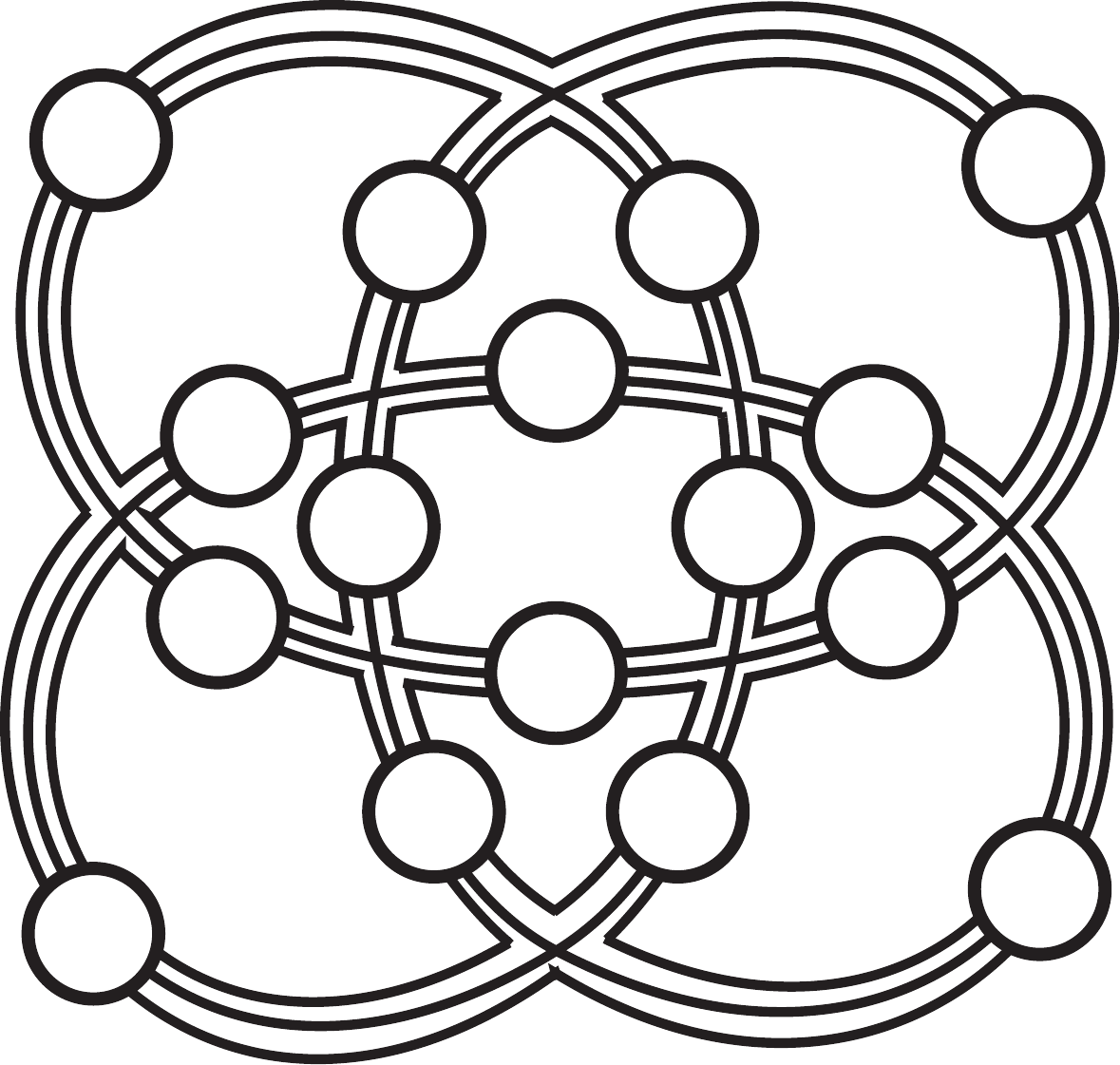}
\caption{Neighborhood $B$ and $r$-disks.}\label{45}
\end{figure}
If there exists $r$-disk $r_0$ such that $x_d \subset r_0$, then by the assumption of induction, $P_{n+1}$ satisfies the claim.  

First, we consider the $x$-disk $x_d$ that contains $m$ double points labeled as $d(P_0)$ that belong to $P_0$ ($m$ $\ge$ $1$).  
\begin{itemize}
\item Case $m$ $=$ $1$.  In this case, there are three possibilities (1), (2), and (3), as shown in Fig.~\ref{41}.  If $x_d \cap P_n$ has exactly one double point $d(P_0)$, then $P_0$ also has a double point $d(P_0)$ keeping the connection as dotted arcs as in Fig.~\ref{41} (1).  By the assumption of induction, $P_n$ is the connected sum of $P_0$, the knot projection that appears similar to $\infty$, and the trefoil projection.  Therefore, if the dotted arc has double points, $P_n$ also has these double points, which implies a contradiction.  Therefore, the dotted arc does not have double points, and thus, we have the equality as shown in Fig.~\ref{41} (1).  However, by the assumption of $P_0$, $P_0$ does not have $1$-gons.  This implies a contradiction, and therefore, the possibility (1) does not appear in $P_n$.  

Next, let us consider possibility (2).  By using the symmetry of the triangle of the strong third move, it is sufficient to consider the left-hand side figure of Fig.~\ref{41} (2) without loss of generality.  In this case, the other two double points of the triangle are not labeled as $d(p_0)$, and the corresponding $P_0$ has a local figure as shown in Fig.~\ref{41} (2).  Similarly to (1), by the assumption of induction, we have the equality in Fig.~\ref{41} (2).  Using the assumption for $P_0$ again, the possibility (2) does not occur.  From discussions similar to (1) and (2), the possibility (3) also does not occur.  Therefore, there is no possibility of case $m$ $=$ $1$.  
\begin{figure}[htbp]
\includegraphics[width=10cm]{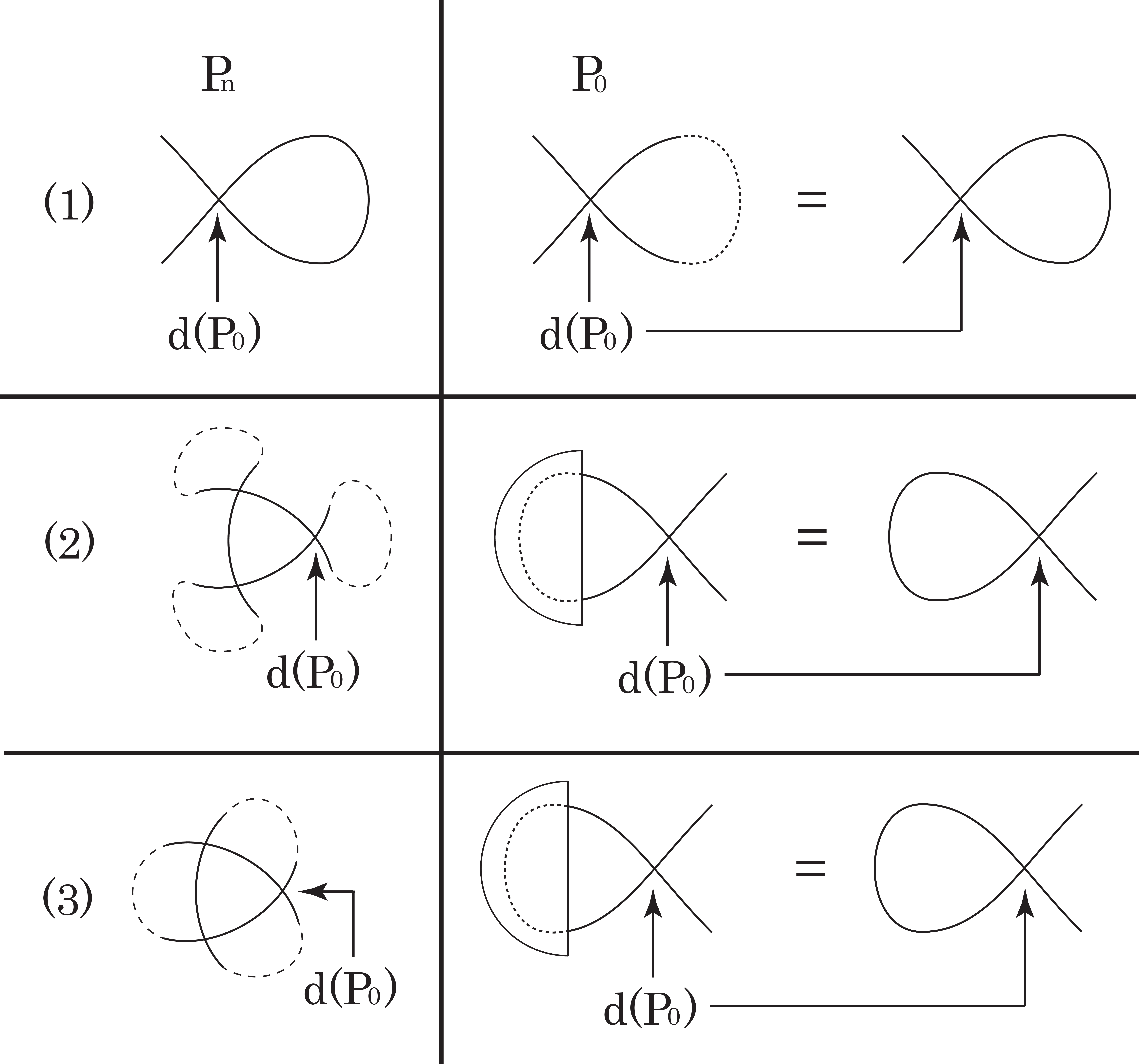}
\caption{Case $m$ $=$ $1$.}\label{41}
\end{figure} 
\item Case $m$ $=$ $2$. 
\begin{figure}[htbp]
\begin{center}
\includegraphics[width=10cm]{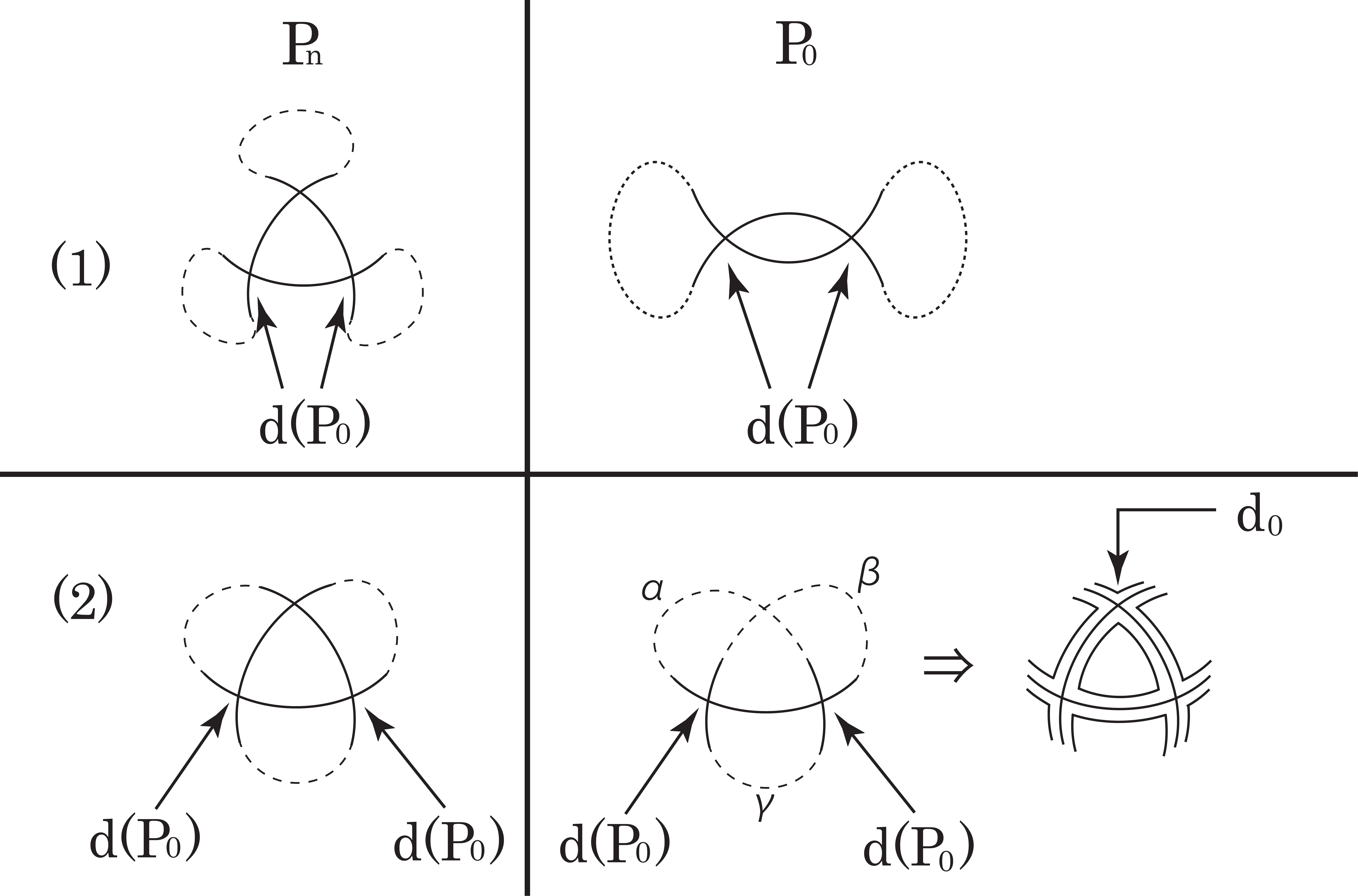}
\end{center}
\caption{Case $m$ $=$ $2$.}\label{42}
\end{figure}
In this case, there are two possibilities as shown in Fig.~\ref{42}.  If $x_d \cap P_n$ has exactly two double points labeled as $d(P_0)$ in Fig.~\ref{42} (1), the other point of the triangle is in a $r$-disk, and we determine $P_0$ locally as the middle figure of Fig.~\ref{42} (1) using the assumption of induction.  However, by the assumption of $P_0$, $P_0$ does not have coherent $2$-gons, which implies a contradiction.  Therefore, the possibility (1) does not occur.  

Next, we consider case (2) by observing Fig.~\ref{42} (2).  If the knot projection $P_n$ appears as the figure on the left-hand side, the corresponding $P_0$ keeps two double points labeled as $d(P_0)$ shown in the middle figure.  Here, we denote three dotted arcs by $\alpha$, $\beta$, and $\gamma$ as shown on the right-hand side of Fig.~\ref{42} (2).  There exists the double point $d_0$ in $P_0$ such that the arcs $\alpha$ intersects $\beta$ at $d_0$ as shown in the figure on the right-hand side.  The reason is described as follows: When we start from the right of $d(P_0)$ to $\alpha$, first, if $\alpha$ intersects other dotted arcs $\alpha$ or $\gamma$ in $P_0$, the figure $P_n$ on the left-hand side has to be modified, which implies a contradiction.  Then, $P_0$ with $B$ is locally drawn as the left-hand side figure of Fig.~\ref{42} (2).  By the assumption of induction, $P_n$ still has the local figure on the right-hand side of Fig.~\ref{42} (2).  Then, three double points appear in the left-hand side figure of Fig.~\ref{42} (2) and it contains the corresponding double point $d_0$ labeled as $d(P_0)$, which implies a contradiction.  Therefore, there is no possibility of $m$ $=$ $2$.  
\item Case $m$ $=$ $3$.
\begin{figure}[htbp]
\begin{center}
\includegraphics[width=7cm]{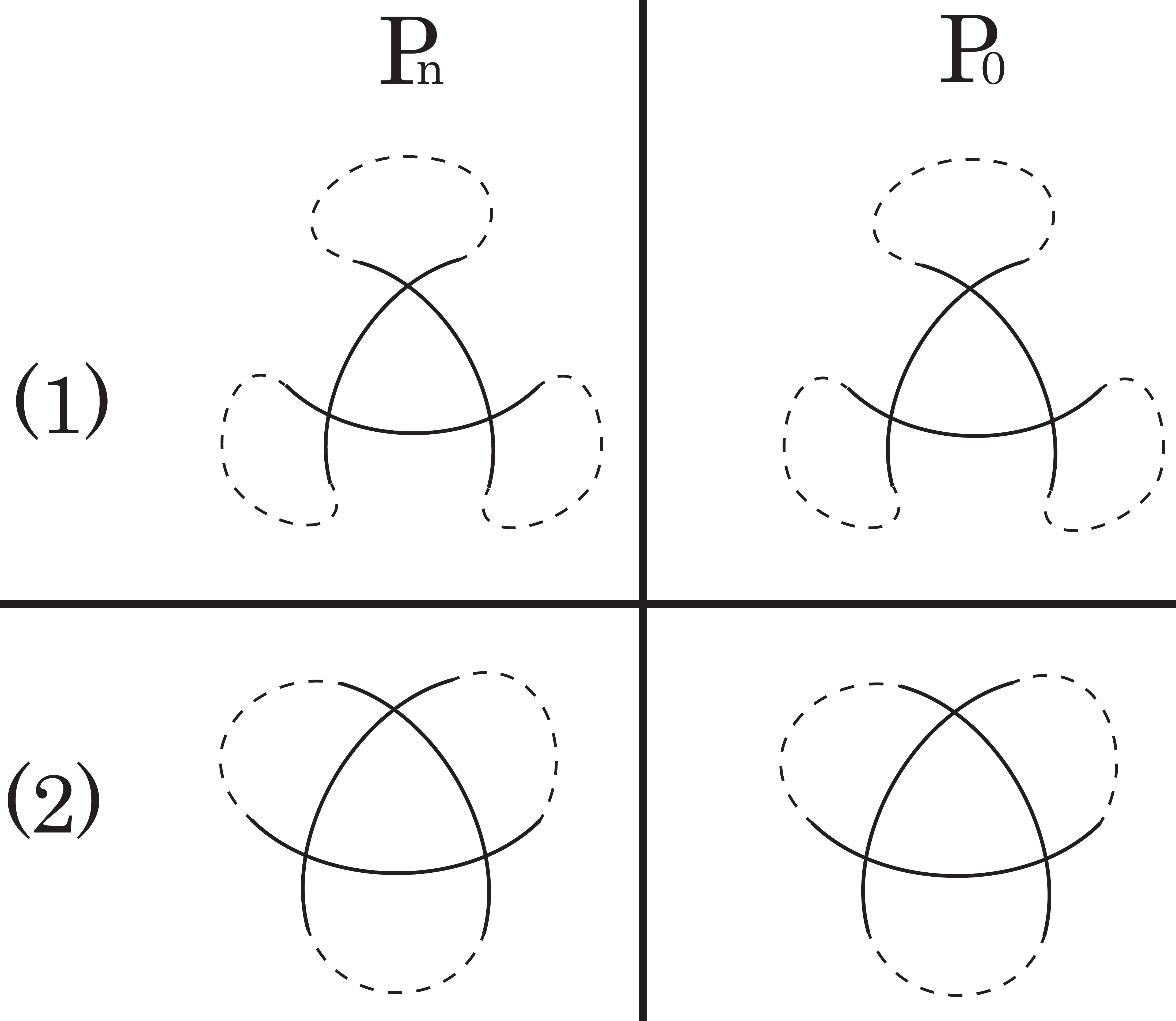}
\end{center}
\caption{Case $m$ $=$ $3$.}\label{43}
\end{figure}
We show this case observing Fig.~\ref{43}.  If $P_n$ is presented as the left column, then by the assumption of induction, $P_0$ is presented as the right column.  However, by the assumption of $P_0$, $P_0$ does not have coherent $3$-gons, which implies a contradiction.  Therefore, there is no possibility of $m$ $=$ $3$.  
\end{itemize}
Then, $x$-disk does not contain double points labeled as $d(P_0)$.  For a set $A$, we denote $S^{2} \setminus A$ by $A^{c}$.  Further, we retake sufficiently small $B$ or $x$-disk if necessary.  Accordingly, the condition implies that 
\begin{equation}\label{r-condition1}
x{\text{-disk}} \subset \bigcup (r{\text{-disk}}) \cup B^{c}.  
\end{equation}  

We also note that ``$x$-disk does not contain two double points such that one belongs to a $r$-disk and another belongs to the other $r$-disk in $P_n$ ($\star$)" (Fig.~\ref{44}).  
\begin{figure}[htbp]
\includegraphics[width=10cm]{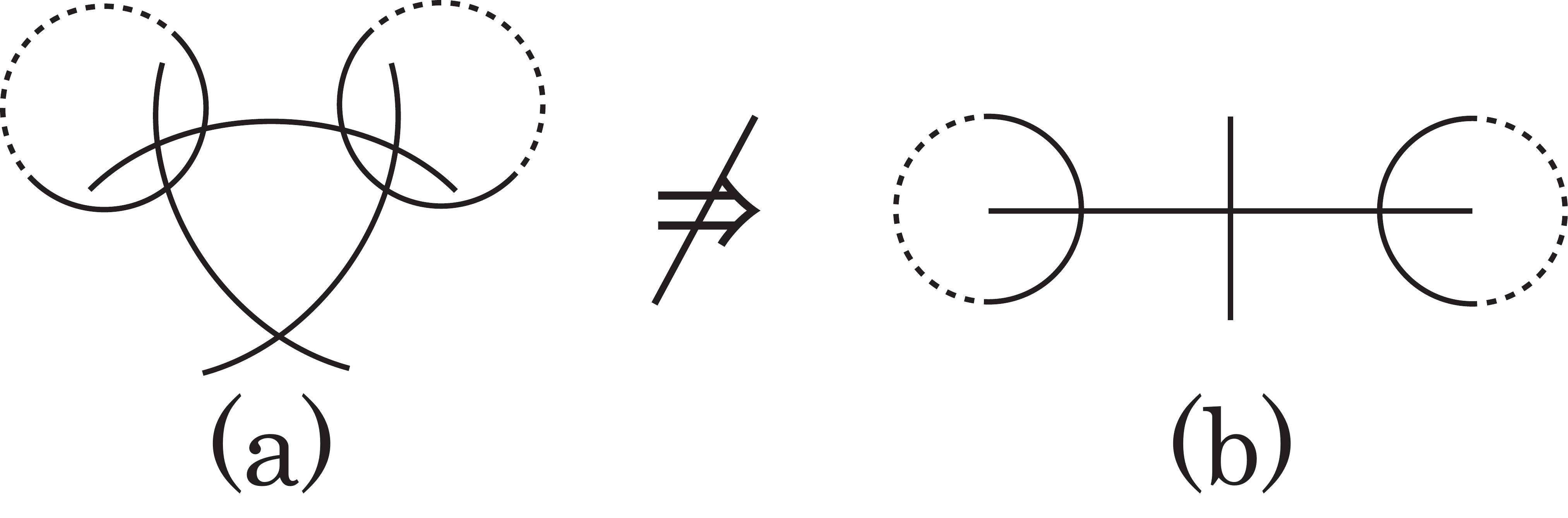}
\caption{The $x$-disk cannot be related to two $r$-disk.}\label{44}
\end{figure}
First, $(1a)$ or $(1b)$ cannot be related to two different double points.  Second, we consider coherent $3$-gons appearing in $(3a)$ and $(3b)$.  From Fig.~\ref{44}, if $x$-disk contains two double points such that one belongs to one $r$-disk and another belongs to the other $r$-disk, we have the case as shown in Fig.~\ref{44} (a) in $P_n$.  However, by the assumption of induction, two different $r$-disks have exactly one double point between them (Fig.~\ref{44} (b), see also Fig.~\ref{45}).  Therefore, Fig.~\ref{44} (a) cannot correspond to Fig.~\ref{44} (b), and therefore, we have ($\star$).  

Using ($\star$) and (\ref{r-condition1}), there exists $r$-disk $r_0$ satisfying the following formula (\ref{r-condition2}) that is the refined version of (\ref{r-condition1}).  
\begin{equation}\label{r-condition2}
x{\text{-disk}} \subset r_0 \cup B^{c}.  
\end{equation}
By the assumption of induction, there is no arc in $(\bigcup (r{\text{-disks}}) \cup B)^{c}$ (e.g. Fig.~\ref{45}).  This implies that $({r_0}^{c} \cap x{\text{-disk}}~x_d)$ does not contain any arc of $P_n$.  Then, we remove $({r_0}^{c} \cap x_d)$ from $x$-disk $x_d$.  After removing it, it is still possible to locally apply $x$ to $P_n$ in $x_d \setminus ({r_0}^{c} \cap x_d)$.  Then, we retake $x$-disk $\tilde{x_d} \subset r_0$ such that $x_d$ is the neighborhood of $(1a)$, $(1b)$, $(3a)$, or $(3b)$.  
Then, we have
\begin{equation}\label{r-condition3}
x{\text{-disk}} \subset r{\text{-disk}}~r_0.  
\end{equation}
Formula (\ref{r-condition3}) completes the proof.  
\end{proof}

Finally, we prove Theorem \ref{otherClass_thm}.  
\begin{proof}
If a knot projection $P$ is the connected sum of $P_0$, the trivial knot projection \begin{picture}(10,15) \put(5,3){\circle{8}} \end{picture}, the knot projection that appears similar to $\infty$, and the trefoil knot projection, we can easily find a path consisting of $(1a)$, $(1b)$, $(3a)$, and $(3b)$ from $P_0$ to $P$.  The converse is implied by Lemma \ref{generalLemma} and Theorem \ref{strongTrivial_thm}.  
\end{proof}

Fig.~\ref{hyou} shows a table of the reduced prime knot projections up to $7$ double points (the notion of reduced knot projections is defined in \cite[Page 2]{IS}) with their trivializing numbers ``$\operatorname{tr}$" (cf. Theorem \ref{weak13trivial_thm}) and cross chord numbers (cf. Theorem \ref{strong13invariant}) expressed by integers on the faces made by the knot projections shown in the figure.  In this table, every symbol $n_m$ (e.g. $3_1$) denotes the knot projection of the prime knot $n_m$.  Symbols $7_A$, $7_B$, and $7_C$ are knot projections that have seven double points.  Every element of $\{7_m (1 \le n \le 7), 7_A, 7_B, 7_C\}$ is different from the other elements up to isotopy on $S^{2}$.  In this figure, we connect two knot projections by a line if two knot projections are related by finite first Reidemeister moves and one third Reidemeister move.  We would like to remark that we can show that $7_4$ and $7_B$ are equivalent under strong (1, 3) homotopy via a prime knot projection with $8$ double points.

\begin{figure}
\includegraphics[width=12cm]{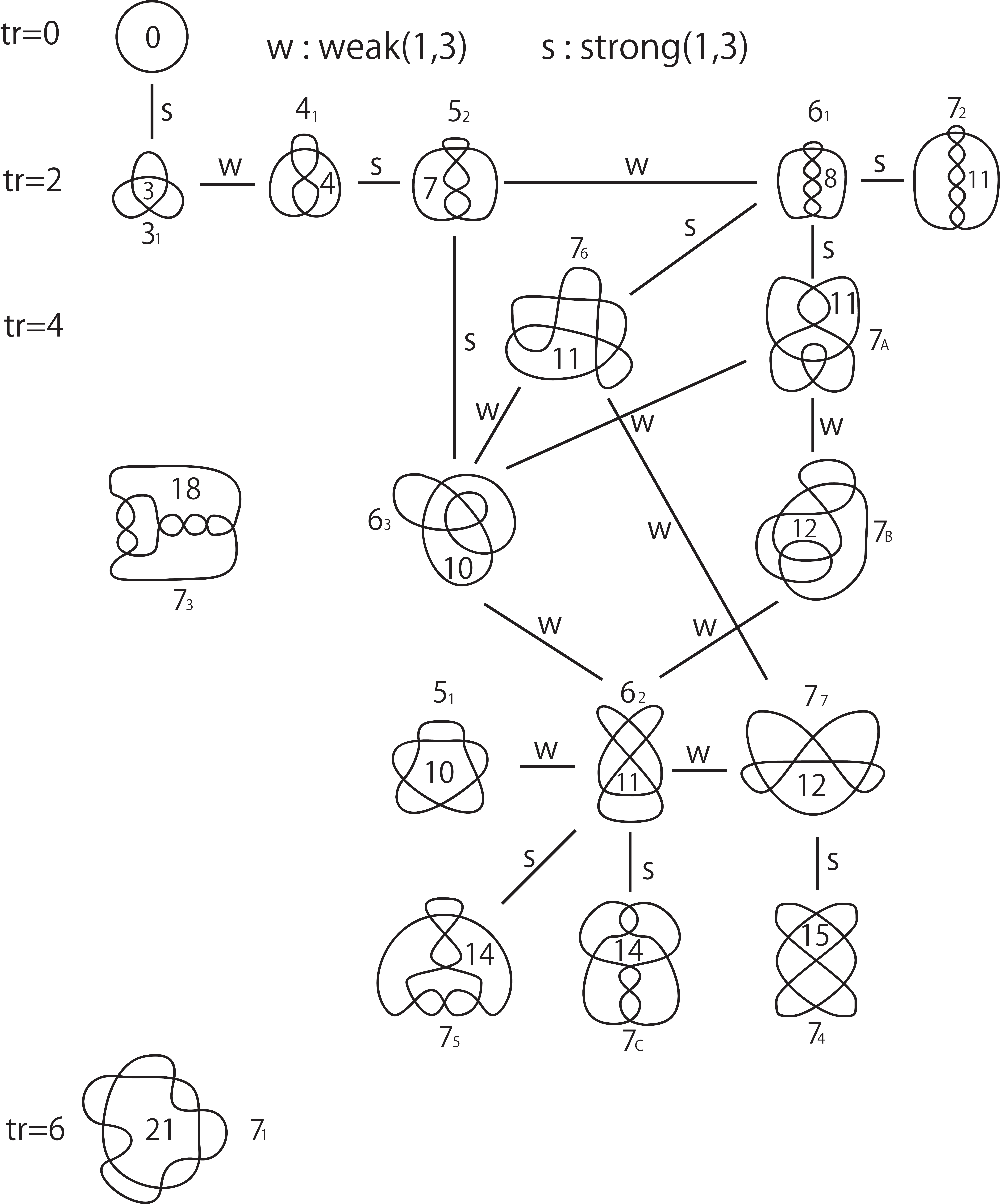}
\caption{Table of reduced prime knot projections up to $7$ double points with trivializing numbers and cross chord numbers.  }\label{hyou}
\end{figure}

{\textbf{Acknowledgments.}}  N. Ito would like to express his sincere gratitude to Professors Kenneth C. Millett and Masahico Saito for their insightful comments.  This work was partly supported by a Waseda University Grant for Special Research Projects (Project number: 2013A-6497).

\vspace{3mm}
Waseda Institute for Advanced Study, 1-6-1, Nishi-Waseda, Shinjuku-ku Tokyo 169-8050, Japan\\
{\it{E-mail address}}: {\texttt{noboru@moegi.waseda.jp}}

\vspace{3mm}
Department of Mathematics, Graduate School of Education, Waseda University, 1-6-1 Nishi-Waseda, Shinjuku-ku, Tokyo, 169-8050, Japan\\
{\it{E-mail address}}: {\texttt{max-drive@moegi.waseda.jp}}

Current address: Gakushuin Boy's Junior High School, 1-5-1 Mejiro, Toshima-ku, Tokyo, 171-0031, Japan\\
{\it{E-mail address}}: {\texttt{Yusuke.Takimura@gakushuin.ac.jp}}

\vspace{3mm}
Department of Mathematics, School of Education, Waseda University, 1-6-1 Nishi-Waseda, Shinjuku-ku, Tokyo, 169-8050, Japan\\
{\it{E-mail address}}: {\texttt{taniyama@waseda.jp}}

\end{document}